\documentclass[11pt]{article}
\usepackage{fullpage}
\usepackage[english]{babel}
\usepackage[left=3cm,right=3cm,top=3cm,bottom=4cm]{geometry}
\usepackage[bottom]{footmisc}

\usepackage{amsmath,amssymb,amsfonts,amsthm}
\usepackage{mathtools,stmaryrd,enumitem}

\SetSymbolFont{stmry}{bold}{U}{stmry}{m}{n}
\usepackage{lmodern,bm,mathrsfs,mathabx}
\usepackage{lipsum}

\usepackage{tikz,tikz-cd}
\usetikzlibrary{calc,shapes.geometric,decorations.pathmorphing,decorations.markings,decorations.pathreplacing}

\usepackage{youngtab}
\usepackage{caption,subcaption}
\usepackage{manfnt}
\usepackage{xcolor}
\usepackage{hyperref}

\theoremstyle{plain}
\newtheorem*{theorem}{Theorem}
\newtheorem*{lemma}{Lemma}
\newtheorem*{proposition}{Proposition}
\newtheorem*{corollary}{Corollary}
\newtheorem*{conjecture-theorem}{Conjecture/Theorem}
\theoremstyle{definition}
\newtheorem*{definition}{Definition}
\newtheorem*{example}{Example}
\theoremstyle{remark}
\newtheorem*{remark}{Remark}

\makeatletter
\def\namedlabel#1#2{\begingroup
    #2%
    \def\@currentlabel{#2}%
    \phantomsection\label{#1}\endgroup
}

\newcommand{\bA}{\mathbb{A}}
\newcommand{\bC}{\mathbb{C}}
\newcommand{\bP}{\mathbb{P}}
\newcommand{\bR}{\mathbb{R}}
\newcommand{\bV}{\mathbb{V}}
\newcommand{\bW}{\mathbb{W}}
\newcommand{\bZ}{\mathbb{Z}}
\newcommand{\cA}{\mathcal{A}}
\newcommand{\cE}{\mathcal{E}}
\newcommand{\cF}{\mathcal{F}}
\newcommand{\cG}{\mathcal{G}}
\newcommand{\cI}{\mathcal{I}}
\newcommand{\cK}{\mathcal{K}}
\newcommand{\cM}{\mathcal{M}}
\newcommand{\cN}{\mathcal{N}}
\newcommand{\cO}{\mathcal{O}}
\newcommand{\cP}{\mathcal{P}}
\newcommand{\cQ}{\mathcal{Q}}
\newcommand{\cR}{\mathcal{R}}
\newcommand{\cT}{\mathcal{T}}
\newcommand{\cV}{\mathcal{V}}
\newcommand{\cW}{\mathcal{W}}
\newcommand{\rV}{\mathscr{V}}
\newcommand{\rW}{\mathscr{W}}
\newcommand{\sA}{\mathrm{A}}
\newcommand{\sE}{\mathsf{E}}

\newcommand{\sK}{\mathrm{K}}
\newcommand{\sP}{\mathsf{P}}
\newcommand{\sR}{\mathsf{R}}
\newcommand{\sT}{\mathrm{T}}
\newcommand{\sV}{\mathsf{V}}
\newcommand{\sZ}{\mathsf{Z}}
\newcommand{\fC}{\mathfrak{C}}
\newcommand{\fY}{\mathfrak{Y}}
\newcommand{\fz}{\mathfrak{z}}
\newcommand{\rb}{\mathrm{b}}
\newcommand{\rr}{\mathrm{r}}
\newcommand{\vsquare}{{\textcolor{gray}{\blacksquare}}}
\newcommand{\pt}{\mathrm{pt}}
\newcommand{\vir}{\mathrm{vir}}
\newcommand{\CY}{\mathrm{CY}}
\newcommand{\orb}{\mathrm{orb}}
\newcommand{\eff}{\mathrm{eff}}
\newcommand{\cat}[1]{\mathsf{#1}}
\renewcommand{\vec}[1]{\bm{#1}}
\renewcommand{\bar}[1]{\overline{#1}}
\renewcommand{\tilde}[1]{\widetilde{#1}}
\renewcommand{\check}[1]{\widecheck{#1}}
\newcommand{\QMaps}{\mathsf{QMaps}}
\newcommand{\Pairs}{\mathsf{Pairs}}
\newcommand{\PiPairs}{\pi\text{-}\mathsf{Pairs}}
\newcommand{\Ell}{\mathsf{Ell}}
\newcommand{\cRep}{\mathcal{R}{\it ep}}
\newcommand{\cube}{\mbox{\mancube}}
\DeclareMathOperator{\Hilb}{Hilb}
\DeclareMathOperator{\SL}{SL}
\DeclareMathOperator{\GL}{GL}

\DeclareMathOperator{\Hom}{Hom}
\DeclareMathOperator{\Ext}{Ext}
\DeclareMathOperator{\cHom}{\mathcal{H}{\it om}}
\DeclareMathOperator{\cExt}{\mathcal{E}{\it xt}}
\DeclareMathOperator{\hk}{hk}
\DeclareMathOperator{\ev}{ev}

\DeclareMathOperator{\im}{im}
\DeclareMathOperator{\supp}{supp}
\DeclareMathOperator{\coker}{coker}
\DeclareMathOperator{\tot}{tot}
\DeclareMathOperator{\Pic}{Pic}
\DeclareMathOperator{\Lie}{Lie}
\DeclareMathOperator{\Stab}{Stab}

\DeclareMathOperator{\level}{level}
\DeclarePairedDelimiter{\inner}{\langle}{\rangle}

\tikzset{
  baseline={([yshift=-.5ex]current bounding box.center)},
  vertex/.style={shape=circle,fill=black,minimum size=4pt,inner sep=0},
}

\title{Quasimaps and stable pairs}
\author{Henry Liu}
\date{\today}

\begin{document}

\maketitle

\begin{abstract}
  We prove an equivalence between the Bryan--Steinberg theory of
  $\pi$-stable pairs on $Y = \cA_{m-1} \times \bC$ and the theory of
  quasimaps to $X = \Hilb(\cA_{m-1})$, in the form of an equality of
  K-theoretic equivariant vertices. In particular, the combinatorics
  of both vertices are described explicitly via box counting. Then we
  apply the equivalence to study the implications for sheaf-counting
  theories on $Y$ arising from 3d mirror symmetry for quasimaps to
  $X$, including the Donaldson--Thomas crepant resolution conjecture.
\end{abstract}

\section{Introduction}

\subsection{Curve counting}

\subsubsection{}

Given a smooth variety $Z$, one can construct many compactifications
of the moduli space of smooth curves in $Z$. These moduli spaces
differ in how they treat the data of a curve in $Z$.
\begin{itemize}
\item Viewing a curve as the data of a map $f\colon C \to Z$ and
  allowing the domain $C$ to develop nodal singularities in the
  compactification yields moduli spaces of stable maps, e.g. as in
  Gromov--Witten (GW) theory.

\item Viewing a curve as the data of an ideal sheaf $\cI_C \subset
  \cO_Z$ and allowing $C$ to degenerate into $1$-dimensional sub{\it
    schemes} in the compactification yields moduli spaces of sheaves,
  e.g. as in Donaldson--Thomas (DT) theory.
\end{itemize}
Morally, one expects all enumerative theories of curves in $Z$ to be
equivalent, possibly up to some wall-crossing behavior: a change of
variables, analytic continuation, and/or normalization. One example of
such an equivalence is the celebrated GW/DT correspondence
\cite{Maulik2006} \cite{Maulik2006a}, proved for all toric $3$-folds
in \cite{Maulik2011}. Consequently, tools from DT-like sheaf-counting
theories can be applied with great effectiveness to GW theory and vice
versa, e.g. the proof of the Igusa cusp form conjecture in
\cite{Oberdieck2018}.

\subsubsection{}

The main result of this paper,
Theorem~\ref{thm:bs-quasimap-correspondence}, is another (much
simpler) such equivalence, at the level of {\it equivariant} and {\it
  K-theoretic} curve counts. For a moduli space $\cM$, working in
equivariant K-theory means that the enumerative invariant to be
considered is not the {\it cohomological} partition function
\[ \sZ_{\cM}^{\text{coh}} \coloneqq \sum_{\deg} \vec x^{\deg} \int_{[\cM_{\deg}]^{\vir}} 1 \in H_{\sT}^*(\pt)_{\text{localized}}[[\vec x^{\pm}]] \]
but rather the {\it K-theoretic} partition function
\[ \sZ_{\cM} \coloneqq \sum_{\deg} \vec x^{\deg} \chi\left(\cM_{\deg}, \hat \cO_{\cM}^{\vir} \right) \in K_{\sT}(\pt)_{\text{localized}}[[\vec x^{\pm}]]. \]
Here, $\vec x$ are variables recording the degree of the curve along
with any other discrete data parameterizing connected components of
$\cM = \bigsqcup_{\deg} \cM_{\deg}$. The virtual structure sheaf
$\cO_{\cM}^{\vir}$ (see e.g. \cite[Section 2.3]{Lee2004},
\cite[Section 3.2]{Ciocan-Fontanine2009}) is the K-theoretic analogue
of the virtual fundamental class $[\cM]^{\vir}$. Its {\it symmetrized}
version $\hat\cO_{\cM}^{\vir}$ is roughly $\cO_{\cM}^{\vir} \otimes
\cK_{\vir}^{1/2}$ where $\cK_{\vir} \coloneqq \det (T^{\vir})^\vee$ is
the virtual canonical. The importance of the twist by
$\cK_{\vir}^{1/2}$ is discussed in \cite[Section 3.2]{Okounkov2017}.

K-theoretic invariants recover cohomological invariants in a
particular limit, and are therefore richer and more general.
Equivariance gives us a handle on partition functions via certain
quantum differential/difference equations that they satisfy. These
features make equivariant K-theory a productive setting for
enumerative geometry; see \cite{Okounkov2017} for an introduction.

\subsubsection{}

In this paper we study only the genus-$0$ setting $g(C) = 0$, where $C
= \bP^1$ when $C$ is smooth. Degeneration arguments reduce a given
general $C$ to this case.

\subsection{The enumerative theories}

\subsubsection{}

In the case where the target space $X = V \sslash_\theta G$ is a
sufficiently nice geometric invariant theory (GIT) quotient, the
theory of stable {\it quasimaps} \cite{Ciocan-Fontanine2014} provides
an alternate compactification for maps $f\colon \bP^1 \to X$ which is
more amenable to computation, among other nice properties, compared to
stable maps. It is related to GW theory via a series of wall-crossings
whose composition gives the classical mirror map of
\cite{Givental1996} \cite{Lian1999}. Roughly, a quasimap to $V
\sslash_\theta G$ is a choice of principal $G$-bundle $\cP$ on $\bP^1$
and a section $s \in H^0(\bP^1, \cP \times_G V)$ satisfying a
stability condition. See section~\ref{sec:quasimap-theory} for
details.

Quasimap theory is especially rich when $X$ is a particularly nice
class of GIT quotients called Nakajima quiver varieties
\cite{Nakajima1994}. These are smooth symplectic varieties associated
to quivers, and are closely related to moduli of quiver
representations, moduli of sheaves on symplectic surfaces, and, from
physics, moduli of vacua in 3d $\cN=4$ supersymmetric gauge theories.
Hence their curve counts often have deep representation-theoretic or
physical meaning.

\subsubsection{}

Let $\QMaps(X)$ be the moduli of quasimaps to a Nakajima quiver
variety $X$. One can ask whether there is some kind of sheaf-counting
theory, on some space related to $X$, whose partition function is
equal to $\sZ_{\QMaps}(X)$. Indeed, if $X = \Hilb(S)$ is the Hilbert
scheme of points on a surface $S$, the graph construction gives a
rough equivalence
\begin{equation} \label{eq:motivation-ADE-fibration-over-P1}
  \begin{pmatrix} f\colon \bP^1 \to \Hilb(S) \end{pmatrix} \approx \begin{pmatrix} 1\text{-dimensional subscheme of } S \times \bP^1 \\ \text{with non-zero degree along } \bP^1 \end{pmatrix}.
\end{equation}

The simplest Nakajima quiver varieties of the form $\Hilb(S)$ are when
$S$ is an ADE surface, namely a minimal resolution of $\bC^2/\Gamma$
for a finite subgroup $\Gamma \subset \SL(2, \bC)$. So, from now on,
let $S$ be an ADE surface. Then $\Hilb(S)$ is the Nakajima quiver
variety associated to the affine ADE quiver corresponding to $S$.

\subsubsection{}

Let $Y \coloneqq S \times \bP^1$. To match with quasimaps to $X =
\Hilb(S)$, DT theory is {\it not} the correct sheaf-counting theory to
take on $Y$. DT theory counts $1$-dimensional subschemes $C \subset
Y$, and therefore its moduli space is
\[ \cat{DT}(Y) \coloneqq \Hilb(Y, \text{curves}). \]
But arbitrary $1$-dimensional subschemes include two types of unwanted
components which do not occur for quasimaps:
\begin{enumerate}
\item $0$-dimensional components that range freely over $Y$;
\item $1$-dimensional components that lie purely in a fiber $S \times
  \{\pt\} \subset Y$.
\end{enumerate}

\subsubsection{}

One way to remove the unwanted contributions of type 1 is to take the
DT partition function $\sZ_{\cat{DT}}$ and divide by the partition
function for the moduli
\[ \cat{DT}_0(Y) \coloneqq \Hilb(Y, \text{points}) \]
of {\it points} on $Y$. A better and more geometric way is to use
Pandharipande--Thomas (PT) theory \cite{Pandharipande2009}. Roughly,
if DT theory counts surjections
\[ \cO_Y \xrightarrow{s} \cO_C \to 0, \]
PT theory counts {\it stable pairs} $[\cO_Y \to \cF]$ in
\[ \cO_Y \xrightarrow{s} \cF \to \cQ \to 0 \]
where $\cQ$ is zero-dimensional. We think of this as allowing the map
$s$ in DT theory to develop a zero-dimensional cokernel $\cQ$ instead
of being a surjection.

If $\Pairs(Y)$ denotes the moduli of stable pairs, the well-known
conjectural DT/PT correspondence predicts that
\footnote{Here is one place to emphasize the importance of using the
  symmetrized $\hat\cO^{\vir}$ instead of $\cO^{\vir}$: the DT/PT
  correspondence (and probably similar wall-crossings) {\it fails to
    hold} if we just use $\cO^{\vir}$.}
\begin{equation} \label{eq:dt-pt-correspondence}
  \sZ_{\Pairs}(Y) = \frac{\sZ_{\cat{DT}}(Y)}{\sZ_{\cat{DT}_0}(Y)}.
\end{equation}

\subsubsection{}

To remove the unwanted contributions of type 2, we can repeat the
DT/PT story as follows. A fiber of $Y \to \bP^1$ is an ADE surface
$S$, so it has an exceptional divisor $E$ for the resolution
$\pi\colon S \to \bC^2/\Gamma$. Let $\Pairs_{\text{exc}}(Y)$ be
the moduli of stable pairs on $Y$ supported only on $E$ for some
fiber. For our purposes, $\sZ_{\Pairs_{\text{exc}}}$ is the
correct series to normalize by.

Bryan--Steinberg (BS) theory \cite{Bryan2016} provide a geometric
approach to this normalization. Roughly, while PT theory allows the
cokernel $\cQ$ to be $0$-dimensional, BS theory allows $\cQ$ to
develop $1$-dimensional components supported only on $E$. Such pairs
$[\cO_Y \xrightarrow{s} \cF]$ are called {\it $\pi$-stable pairs}. A
precise definition is in section~\ref{sec:bs-theory-definition}. Since
it depends not only on $Y$ but also on the resolution $\pi$, we denote
the BS moduli space $\PiPairs(Y)$.

As with the DT/PT correspondence, the conjectural PT/BS correspondence
is that
\footnote{As with the DT/PT correspondence, at the K-theoretic level
  this also requires the symmetrized $\hat\cO^{\vir}$.}
\begin{equation} \label{eq:bs-pt-correspondence}
  \sZ_{\PiPairs}(Y) = \frac{\sZ_{\Pairs}(Y)}{\sZ_{\Pairs_{\text{exc}}}(Y)}.
\end{equation}
Assuming the DT/PT correspondence, one can replace the rhs by
$\sZ_{\cat{DT}}(X)/\sZ_{\cat{DT}_{\text{exc}}}(X)$ where
$\cat{DT}_{\text{exc}}$ is defined in exactly the same way as
$\Pairs_{\text{exc}}$.

\subsection{Main result(s)}

\subsubsection{}

The main result of this paper is the {\it BS/quasimaps correspondence}
of Theorem~\ref{thm:bs-quasimap-correspondence}, which implies that
\[ \sZ_{\PiPairs}(\cA_{m-1} \times \bP^1) = \sZ_{\QMaps}(\Hilb(\cA_{m-1})) \]
but is a more precise statement. Namely, for toric geometries, $\sZ$
factors as a product of contributions from torus-fixed points and
torus-invariant curves. These contributions are called {\it vertices}
$\sV$ and {\it edges} $\sE$ respectively, and one can match vertices
and edges for different enumerative theories {\it individually}. While
it is very straightforward to match edges for DT/PT/BS and quasimaps,
vertices have a certain combinatorial complexity. The more precise
content of Theorem~\ref{thm:bs-quasimap-correspondence} is that
\begin{equation} \label{eq:bs-quasimap-equivalence-P1}
  \sV_{\PiPairs}(\cA_{m-1} \times \bC) = \sV_{\QMaps}(\Hilb(\cA_{m-1}))
\end{equation}
Both \eqref{eq:dt-pt-correspondence} and
\eqref{eq:bs-pt-correspondence} should also be refined in this way,
using the DT \cite{Maulik2006} and PT \cite{Pandharipande2009a}
vertices.

For $3$-folds, DT/PT vertices are objects on $\bC^3$ and have three
{\it legs}, corresponding to what happens along non-compact
torus-fixed curves. Legs are where vertices glue onto edges. In BS
theory, the notion of a vertex depends on the resolution $\pi$. For $Y
= S \times \bP^1$, the BS vertex is an object on $S \times \bC$. It
has one set of $m$ legs from the $\bC$ direction, and two legs from
$S$. We say the BS vertex has {\it one leg} when these latter two legs
are empty. Then $\sZ_{\PiPairs}(S \times \bP^1)$ is the gluing of two
BS $1$-leg vertices along an edge. In this language,
\eqref{eq:bs-quasimap-equivalence-P1} says the BS $1$-leg vertex is
equal to the quasimap vertex. Note that in quasimap theory there is no
notion of $2$-leg or $3$-leg vertices.

\subsubsection{}

In section~\ref{sec:basic-geometry}, we set up notation for $Y =
\cA_{m-1} \times \bC$ and its torus action, and then in
section~\ref{sec:bs-vertex} we define the BS vertex. The proof of the
BS/quasimaps correspondence goes via equivariant localization, and so
section~\ref{sec:bs-vertex-boxes} gives an explicit combinatorial
description of the BS $1$-leg vertex as a weighted sum over certain 3d
box configurations similar to 3d partitions. An important detail,
discussed in section~\ref{sec:bs-fixed-loci-structure}, is that even
for the $1$-leg vertex the moduli of $\pi$-stable pairs has
torus-fixed loci of arbitrarily large dimension; such a phenomenon is
not present in DT or PT theory.

In section~\ref{sec:two-hilbert-schemes}, we begin by understanding
$\Hilb(\cA_{m-1})$ and a related space $\Hilb([\bC^2/\Gamma])$ as
Nakajima quiver varieties. Here $[\bC^2/\Gamma]$ is the CY3 orbifold
associated to $\cA_{m-1}$. While we will not need
$\Hilb([\bC^2/\Gamma])$ for the BS/quasimaps correspondence, it plays
a significant role later in the DT crepant resolution conjecture.
Section~\ref{sec:hilbert-schemes-equivariant-geometry} gives a
combinatorial description of the torus-equivariant geometry of both
Hilbert schemes using 2d box configurations similar to Young diagrams.
Then section~\ref{sec:quasimap-theory} defines the quasimap vertex,
and explicitly describe it as a weighted sum over labeled such 2d box
configurations.

\subsubsection{}

The statement and proof of the BS/quasimap correspondence occupies
section~\ref{sec:bs-quasimaps}. Although we prove it only for the
geometry $\cA_{m-1} \times \bP^1$ for ease of exposition, the
correspondence certainly extends beyond the type A case to any ADE
bundle over $\bP^1$. This is the most general setting in which
comparable BS and quasimap theories can be defined. For type D and
type E, the combinatorial interpretation of contributions to the
quasimap vertex in terms of colored boxes is less straightforward.

The proof of the BS/quasimaps correspondence involves constructing an
isomorphism of torus-fixed loci which respects the tangent-obstruction
theories. This is done using the torus-equivariant derived McKay
equivalence of section~\ref{sec:bs-quasimaps-Tvir-correspondence}.
That it matches the stability conditions defining $\pi$-stable pairs
and quasimaps is the content of
section~\ref{sec:stability-conditions}, where we also show that the
isomorphism of fixed loci extends to an isomorphism of BS and quasimap
$1$-leg moduli spaces, thereby showing that the two theories are truly
equivalent.

\subsubsection{}

The theory of quasimaps to Nakajima quiver varieties can be used to
study {\it 3d mirror symmetry}, also known as {\it symplectic duality}
\cite{Intriligator1996}. This is an intimate (conjectural)
relationship between two {\it mirror} Nakajima quiver varieties $X$
and $\check X$ which relates their quasimap vertices; see
section~\ref{sec:3d-mirror-symmetry} for more details. The
relationship between the two quasimap vertices, once pushed through
the BS/quasimaps correspondence, yields a remarkable range of known
results, e.g. the geometric engineering of certain Nekrasov partition
functions (section~\ref{sec:3d-mirror-symmetry-CY}), and the DT
crepant resolution conjecture (section~\ref{sec:DT-CRC}). In
particular, these results are very special limits of the full 3d
mirror symmetry, which yields a (conjectural) statement of the DT CRC
for equivariant K-theoretic vertices.

\subsection{Acknowledgements}

This project sprouted from an (ongoing) attempt to prove a
PT/quasimaps correspondence, to which I was first introduced by Noah
Arbesfeld, Andrei Okounkov, and Petr Pushkar. I would like to thank
them along with Jim Bryan, Yakov Kononov, and Andrey Smirnov for many
productive discussions. I am especially grateful to Jim Bryan and
Andrei Okounkov for reading a preliminary draft and for numerous
helpful suggestions which improved the readability and content of this
paper.

\section{Bryan--Steinberg pairs}

\subsection{The threefold}
\label{sec:basic-geometry}

\subsubsection{}

While many constructions in this paper can be done more generally, the
basic geometry of interest is
\[ Y \coloneqq \cA_{m-1} \times C \]
for a curve $C$ which will either be $\bC$ or $\bP^1$. The surface $S
\coloneqq \cA_{m-1}$ is the minimal resolution of the type A
singularity $\bC^2/\Gamma$. Here, $\bC^2$ has the canonical symplectic
form, and
\[ \Gamma \coloneqq \bZ/m \subset \SL(2, \bC) \]
acts by the symplectomorphism
\[ \xi \cdot (x, y) \coloneqq (\xi x, \xi^{-1} y). \]
By an abuse of notation, we conflate the coordinate functions $(x, y, z)$ on
$\bC^2 \times C$ with the weights of the torus
\begin{equation} \label{eq:full-torus-weights}
  \sT \coloneqq \bC^\times_x \times \bC^\times_y \times \bC^\times_z
\end{equation}
acting on it. \footnote{Our convention for weights is {\it opposite}
  to some parts of the literature, notably \cite{Okounkov2017} (which
  develops K-theoretic quasimap theory) and successive works. For
  example, if $\bC_z^\times$ acts on $\bA^1$, then for us the
  character of $\cO_{\bA^1}$ is $1/(1-z)$, whereas for
  \cite{Okounkov2017} it would be $1/(1-z^{-1})$. On the other hand,
  our convention does agree with older works such as \cite{Maulik2006}
  (which develops DT theory).}
Then the minimal resolution $\cA_{m-1}$ has exceptional divisor
consisting of a chain of $m-1$ copies of $\bP^1$, which we denote
$E_1, \ldots, E_{m-1}$, with weights as depicted in
Figure~\ref{fig:toric-diagram-Am-C}.

\begin{figure}[h]
  \centering
  \begin{tikzpicture}[scale=1.7, decoration={markings,mark=at position 0.5 with {\arrow{>}}}]
    \node[vertex] (p0) at (-3,0) {};
    \node[vertex] (p1) at (-2,0.5) {};
    \node[vertex] (p2) at (-1,0.75) {};
    \node (etc) at (0,0.85) {$\cdots$};
    \node[vertex] (pm2) at (1,0.75) {};
    \node[vertex] (pm1) at (2,0.5) {};
    \node[vertex] (pm) at (3,0) {};
    \draw[postaction={decorate}] (-4,-0.75) -- node[label=below:$y^{-m}$]{} (p0);
    \draw[thick,postaction={decorate}] (p0) -- node[label=below:$xy^{-m+1}$]{}node[label=above:$E_1$]{} (p1);
    \draw[thick,postaction={decorate}] (p1) -- node[label=below:$x^2y^{-m+2}$]{}node[label=above:$E_2$]{} (p2);
    \draw[thick,postaction={decorate}] (p2) -- (etc);
    \draw[thick,postaction={decorate}] (etc) -- (pm2);
    \draw[thick,postaction={decorate}] (pm2) -- node[label=below:$x^{m-2}y^{-2}$]{}node[label=above:$E_{m-2}$]{} (pm1);
    \draw[thick,postaction={decorate}] (pm1) -- node[label=below:$x^{m-1}y^{-1}$]{}node[label=above:$E_{m-1}$]{} (pm);
    \draw[postaction={decorate}] (pm) -- node[label=below:$x^m$]{} (4,-0.75);
    \draw[postaction={decorate}] (p0) -- node[label=left:$z$]{} (-3,1);
    \draw[postaction={decorate}] (p1) -- (-2,1.5);
    \draw[postaction={decorate}] (p2) -- (-1,1.75);
    \draw[postaction={decorate}] (pm2) -- (1,1.75);
    \draw[postaction={decorate}] (pm1) -- (2,1.5);
    \draw[postaction={decorate}] (pm) -- node[label=right:$z$]{} (3,1);
  \end{tikzpicture}
  \caption{Toric diagram of $\cA_{m-1} \times \bC$ (exceptional divisor of $\cA_{m-1}$ in bold)}
  \label{fig:toric-diagram-Am-C}
\end{figure}
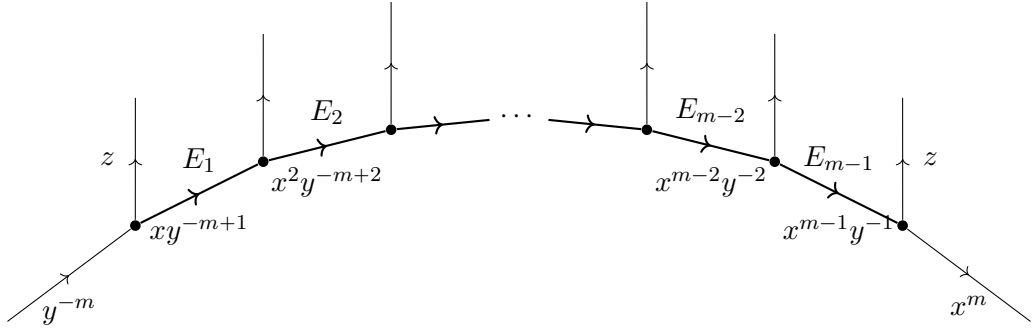

Let $p_0, \ldots, p_m$ denote the $\sT$-fixed points, such that $E_a$
connects $p_{a-1}$ and $p_a$. Let
\begin{align*}
  x_a &\coloneqq x^{a+1}y^{-m+a+1} \\
  y_a &\coloneqq x^{-a}y^{m-a}
\end{align*}
so that $(x_a, y_a, z)$ are coordinates for the toric chart $U_a =
\bC^3$ around $p_a$. 

\subsubsection{}
\label{sec:tori-definitions}

Often, it will be necessary to switch between working on $Y$ and
working on $S$. To prevent confusion, let
\[ \sT' \coloneqq \bC_x^\times \times \bC_y^\times \subset \sT \]
be the torus acting on the $\bC^2$ defining $S$. In general, we will
add a prime to any equivariant object when considering it on $S$
instead of $Y$. For example, let
\[ \sA \coloneqq \{xyz = 1\} \subset \sT \]
be the Calabi--Yau sub-torus. Then
\[ \sA' = \{(t, t^{-1})\} \subset \sT' = \sA' \times \bC_{\hbar}^\times, \]
where we use $\hbar \coloneqq 1/xy$ to denote the weight of the
symplectic form on $S$. Then $\sA'$ preserves the symplectic form on
$S$, and contains the $\Gamma$-action defining $S$.

\subsection{The 1-leg vertex}
\label{sec:bs-vertex}

\subsubsection{}
\label{sec:bs-theory-definition}

Very generally, let $Y_0$ be a quasi-projective $3$-fold with rational
Gorenstein singularities, and let $\pi\colon Y \to Y_0$ be a
resolution of singularities of relative dimension $\le 1$. Let
$\cat{Coh}_{\le i}(Y) \subset \cat{Coh}(Y)$ denote the full
sub-category of coherent sheaves with support in dimension $\le i$.
Associated to $\pi$ is a torsion pair $(\cat{T}, \cat{F})$ in
$\cat{Coh}_{\le 1}(Y)$, given by
\begin{align*}
  \cat{T} &\coloneqq \left\{\cQ \in \cat{Coh}_{\le 1}(Y) \; | \; R\pi_*\cQ \in \cat{Coh}_{\le 0}(Y_0)\right\} \\
  \cat{F} &\coloneqq \left\{\cF \in \cat{Coh}_{\le 1}(Y) \; | \; \Hom(\cQ, \cF) = 0 \text{ for all } \cQ \in \cat{T}\right\} = \cat{T}^\perp.
\end{align*}
A Bryan--Steinberg (BS) pair \cite{Bryan2016} for $\pi$, also called a
$\pi$-stable pair, is a map
\[ s\colon \cO_Y \to \cF \]
such that $\cF \in \cat{F}$ and $\coker(s) \in \cat{T}$. Equivalently,
a $\pi$-stable pair is a short exact sequence
\[ 0 \to \cO_C \to \cF \to \cQ \to 0, \]
where $\cO_C \coloneqq \im(s)$ and $\cQ \coloneqq \coker(s)$. We know
$\im(s)$ is the structure sheaf of some curve $C$ because it is a
quotient of $\cO_Y$ supported in dimension $1$.

\begin{remark}
  Due to \eqref{eq:motivation-ADE-fibration-over-P1}, we only want to
  consider $3$-folds $Y$ of the form $S \times \bP^1$, where the
  surface $S$ contains the rational Gorenstein singularities. For
  surfaces it is known that the only such singularities are of ADE
  type \cite[Theorem 7.5.1]{Ishii2014}.
\end{remark}

\subsubsection{}

Let $\PiPairs(Y)$ denote the moduli of $\pi$-stable pairs on $Y$. When
$Y = S \times \bP^1$, let $D \coloneqq S \times \{\infty\}$ be the
divisor at infinity and consider the open locus
\[ \PiPairs(Y)_{\text{nonsing }\infty} \subset \PiPairs(Y) \]
where the evaluation map
\begin{align*}
  \ev_\infty\colon \PiPairs(Y)_{\text{nonsing }\infty} &\to \Hilb(D, \text{points}) \\
  [\cO_Y \to \cF] &\mapsto \cF\big|_{\infty}
\end{align*}
lands in the Hilbert scheme of {\it points} on $D$, instead of the
Hilbert scheme of {\it curves}.

\begin{definition}
  The {\it BS $1$-leg vertex} is the series
  \[ \sV_{\PiPairs}(Q, \vec A) \coloneqq \sum_{\substack{n \in \bZ\\\beta \in H_2(S, \bZ)}} Q^n \vec A^\beta \ev_{\infty,*}\left(\PiPairs^{n,\beta}, \hat\cO^{\vir}\right) \in K_{\sT}(S)_{\text{localized}}((Q))[[\vec A]]. \]
  The variables $Q$ and $\vec A \coloneqq (A_1, \ldots, A_{m-1})$
  record certain discrete data indexing the connected components
  $\PiPairs^{n,\beta} \subset \PiPairs(Y)_{\text{nonsing }\infty}$:
  \begin{itemize}
  \item the box-counting variable $Q$ records the (possibly negative)
    renormalized volume $n \coloneqq \chi_{\text{normalized}}(\cF)$;
  \item the K\"ahler variables $\vec A$ record the degree $\beta
    \coloneqq \deg(\cF)$ along components $E_1, \ldots, E_{m-1}$ of the
    exceptional divisor.
  \end{itemize}
  This is all the same as in DT or PT theory; see \cite[Section
    4.4]{Maulik2006} for details.
\end{definition}

\subsubsection{}

The moduli $\PiPairs(Y)$ inherits the $\sT$-action on $Y$. For $p
\in \Hilb(S)$, let
\[ \PiPairs_p \coloneqq \ev_\infty^{-1}(p). \]
It has components $\PiPairs_p^{n,\beta}$ consisting of $\pi$-stable
pairs with $\chi_{\text{normalized}} = n$ and $\deg = \beta$. Then the
$\sT$-fixed locus decomposes as
\[ \PiPairs(Y)_{\text{nonsing }\infty}^{\sT} = \bigoplus_{p \in \Hilb(S)^{\sT'}} \bigoplus_{\beta \in H_2(S, \bZ)} \bigoplus_{n \in \bZ} \left(\PiPairs_p^{n,\beta}\right)^{\sT}, \]
and the $p$-th component of the BS $1$-leg vertex is therefore
\begin{equation} \label{eq:bs-vertex-definition}
  \sV^p_{\PiPairs}(Q, \vec A) = \sum_{\beta \in H_2(S, \bZ)} \sum_{n \in \bZ} \chi\left(\PiPairs_p^{n,\beta}, \hat\cO^{\vir}\right) Q^n \vec A^\beta.
\end{equation}
In section~\ref{sec:bs-fixed-loci-structure}, we roughly describe the
fixed loci $(\PiPairs_p^{n,\beta})^{\sT}$ and argue they are proper.
Hence the BS $1$-leg vertex is well-defined via localization. Note
that for a given $\beta$, the renormalized volume $n$ is bounded from
below, and so the coefficients of the power series in $\vec A$ are
Laurent series in $Q$.

\subsubsection{}
\label{sec:bs-vertex-Tvir}

The appropriate tangent-obstruction theory for a $\pi$-stable pair
arises from viewing it as a two-term complex
\[ I^\bullet \coloneqq [\cO_Y \to \cF] \in D^b\cat{Coh}(Y), \]
like for stable pairs \cite[Section 2]{Pandharipande2009}. Then there
is a universal formula
\begin{equation} \label{eq:bs-Tvir}
  \begin{split}
    T^{\vir}_{[I^\bullet]}
    &= \Ext^1(I^\bullet, I^\bullet)_0 - \Ext^2(I^\bullet, I^\bullet)_0 \\
    &= \chi(\cO_Y) - \chi(I^\bullet, I^\bullet)
  \end{split}
\end{equation}
for the virtual tangent space. Here $\chi(-, -) \coloneqq
\sum_i (-1)^i \Ext^i(-, -)$. As in DT or PT theory, at fixed points
$[I^\bullet] \in (\PiPairs)^{\sT}$ the $\sT$-character of $T^{\vir}$
can be obtained via \v Cech cohomology. The computation is identical
in DT/PT/BS theory and we will not repeat it here.

Usually for the vertex of a DT-like theory, one redistributes vertex
and edge contributions so that the $1$-leg vertex $\sV^p$ does not
include the contribution of the infinite leg(s) (see \cite[Section
  4.9]{Maulik2006}). We have not performed such a redistribution in
\eqref{eq:bs-Tvir}; in our setup, for $I^\bullet \in
\PiPairs_{\text{nonsing }\infty}^{\sT}$, the redistribution says to
use
\begin{equation} \label{eq:bs-Tvir-normalization}
  T^{\vir}_{[I^\bullet],\sim} \coloneqq T^{\vir}_{[I^\bullet]} - T_{\ev_\infty(I^\bullet)}\Hilb(S)
\end{equation}
instead of $T^{\vir}_{[I^\bullet]}$. However, this redistribution is
essentially responsible for the renormalization of the Euler
characteristic $n = \chi_{\text{normalized}}(\cF)$, and we {\it do}
perform this renormalization.

\subsection{Boxes and rods}
\label{sec:bs-vertex-boxes}

\subsubsection{}

Since the geometry $Y = \cA_{m-1} \times \bP^1$ is toric, a
$\sT$-fixed $\pi$-stable pair $[\cO_Y \to \cF]$ can be described by
toric data. Namely, $\cF$ can be described as a configuration of boxes
in each toric chart $U_a \subset Y$. This is the combinatorial
approach taken in DT and PT theory as well; see \cite[Section
  4.2]{Maulik2006} and \cite[Section 2]{Pandharipande2009a} for
details.

Our convention for box diagrams is as follows. All box diagrams will
be drawn on the toric skeleton of Figure~\ref{fig:toric-diagram-Am-C}.
A box drawn with smallest $x$-, $y$- and $z$-coordinates $(i,j,k)$ in
the chart $U_a = \bC^3$ indicates that the $\bC[x_a,y_a,z]$-module
$\cF(U_a)$ has an element of weight $\cube = x_a^i y_a^j z^k$. (We
often conflate the coordinates $(i,j,k)$ of a box with its weight
$x_a^i y_a^j z^k$.) Importantly, it is possible for $\cF(U_a)$ to
contain an $m$-dimensional vector subspace of weight $\cube$, where $m
> 1$. In this case, we label the box with the integer $m$, which we
think of as a multiplicity. \footnote{This phenomenon is not new to BS
  theory; the PT $3$-leg vertex is allowed to have certain boxes of
  multiplicity $2$, and this phenomenon is important there because it
  leads to positive-dimensional $\sT$-fixed loci.}

\subsubsection{}

Elements $\cF \in \PiPairs(Y)_{\text{nonsing }\infty}^{\sT}$
correspond to box configurations on the geometry $\cA_{m-1} \times
\bC$ with infinite legs along the $z$ direction, as in
Figure~\ref{fig:bs-fixed-point-generic-example}. These legs are
described by an $m$-tuple of partitions
\[ \vec\lambda = (\lambda_0, \ldots, \lambda_{m-1}) \]
corresponding to a $\sT'$-fixed point in $\Hilb(\cA_{m-1})$. More
precisely, the infinite leg in the chart $U_a$ is the module
\[ L_a \coloneqq \bC[x_a, y_a, z] / I_{\lambda_a} \]
where $I_{\lambda_a}$ is the ideal generated by
$x_a^{i(\square)}y_a^{j(\square)}$ for every $\square \notin \lambda_a$.

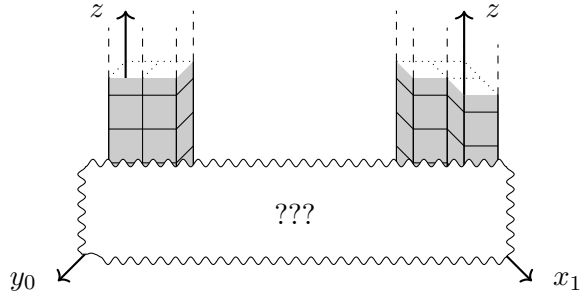
\begin{figure}[h]
  \centering
  \begin{tikzpicture}[scale=0.45]
    \draw[thick,->] (0,0) -- (0,6) node[label=left:$z$]{};
    \draw[thick,->] (10,0) -- (10,6) node[label=right:$z$]{};
    \draw[thick] (0,0) -- (10,0);
    \draw[thick,->] (0,0) -- (-2,-2) node[label=left:$y_0$]{};
    \draw[thick,->] (10,0) -- (12,-2) node[label=right:$x_1$]{};

    \draw[fill=gray!40] (-0.5,4) -- (-0.5,-0.5) -- (1.5,-0.5) -- (1.5,4);
    \draw[fill=gray!40] (1.5,4) -- (1.5,-0.5) -- (2,0) -- (2,4.5);
    \foreach \x/\y in {8/0,9.5/-0.5} {
      \draw[fill=gray!40] (\x,\y+4.5) -- (\x,\y) -- (\x+0.5,\y-0.5) -- (\x+0.5,\y+4);
    };
    \foreach \x/\y in {8.5/0,10/-0.5} {
      \draw[fill=gray!40] (\x,\y+4) -- (\x,\y-0.5) -- (\x+1,\y-0.5) -- (\x+1,\y+4);
    };
    \foreach \x in {0,1,2} {
      \draw (\x-0.5,-0.5) -- (\x-0.5,4);
    };
    \draw (2,0) -- (2,4.5);
    \foreach \x/\y in {8/0,9/0,9.5/-0.5,10.5/-0.5} {
      \draw (\x+0.5,\y-0.5) -- (\x+0.5,\y+4);
    };
    \draw (8,0) -- (8,4.5);
    \foreach \y in {0,1,2,3,4} {
      \draw (-0.5,\y-0.5) -- (1.5,\y-0.5) -- (2,\y);
      \draw (8,\y) -- (8.5,\y-0.5) -- (9.5,\y-0.5) -- (10,\y-1) -- (11,\y-1);
    };
    \draw[fill=white,decorate,decoration={snake,segment length=2mm,amplitude=0.5mm}] (-1.3,-1.2) rectangle ++(12.5,2.5);
    \node at (5,0) {???};
    \foreach \x/\y in {-0.5/4,0.5/4,1.5/4,2/4.5,8/4.5,8.5/4,9.5/4,10/3.5,11/3.5} {
      \draw[dashed] (\x,\y) -- (\x,\y+1.5);
    };
    \foreach \x/\y in {-0.5/4,0.5/4} {
      \draw[dotted] (\x,\y) -- ++(0.5,0.5);
    };
    \draw[dotted] (0,4.5) -- ++(2,0);
    \draw[dotted] (9.5,4) -- ++(-0.5,0.5);
    \draw[dotted] (11,3.5) -- ++(-1,1);
    \draw[dotted] (8,4.5) -- ++(2,0);
    \draw[dotted] (9.5,4) -- ++(1,0);
  \end{tikzpicture}
  \caption{An element of $\PiPairs_{(2),(2,1)}(\cA_1)_{\text{nonsing }\infty}$}
  \label{fig:bs-fixed-point-generic-example}
\end{figure}

Among the $\sT$-invariant curves in $Y$, the non-compact (resp.
compact) ones are called {\it external} (resp. {\it internal}) legs in
box configurations. In addition to the legs $L_a$ of weight $z$, there
are two other external legs with weights $x^m$ and $y^m$. In
principle, one can set up the theory so that these external legs can
be non-empty as well. The resulting vertices would be the BS $2$-leg
or $3$-leg vertices. In our setting they must be empty, and hence we
call the resulting series the $1$-leg vertex.

\subsubsection{}

The overarching goal is to characterize the valid box configurations
that can occur in the unspecified region of
Figure~\ref{fig:bs-fixed-point-generic-example}. It is productive to
first understand valid box configurations for $\cF$ and then to
precisely identify the $\cO_Y$-module structure dictating which boxes
generate which other ones. In other words, we first describe $\cF$ as
an element of $K_{\sT}(Y)$, and then as an element of
$\cat{Coh}_{\sT}(Y)$. Since the description is more combinatorially
involved than for DT or PT theory, we outline the main ingredients
here.
\begin{itemize}
\item Definition~\ref{def:rods-terminology} introduces {\it rods} and
  various terminology for them. They are the main new ingredient for
  BS theory, in comparison with DT and PT theory.

\item Lemmas~\ref{lem:bs-rod-degree-lower-bound} and
  \ref{lem:bs-rod-degree-upper-bound} characterize the rods that can
  appear in $\cF$ via the restrictions $\cQ \in \cat{T}$ and $\cF \in
  \cat{T}^\perp$ respectively.

\item Lemma~\ref{lem:bs-boxes-module-structure} describes the way in
  which $\cF(U_a)$ is built from boxes and rods, in a single chart.

\item Proposition~\ref{prop:bs-local-model} describes all possible
  $\cF$, in the case of a single non-trivial external leg $\lambda_a =
  \square$. Such $\cF$ are called {\it local models}, and essentially
  consist of the single leg along with some rods.

\item Proposition~\ref{prop:bs-local-model-filtration} characterizes
  all possible $\cF$ in general, in terms of local models.
\end{itemize}

\subsubsection{}

Let $0 \to \cO_C \to \cF \to \cQ \to 0$ be a $\sT$-fixed element in
$\PiPairs(Y)_{\text{nonsing }\infty}$. The new feature in BS theory
that is not present in DT or PT theories is the possibility of
$1$-dimensional components in $\cQ$. Since $\cQ$ is supported on $E$,
we now introduce all the necessary terminology for box configurations
pertaining to $1$-dimensional sheaves on $E$.

\begin{definition} \label{def:rods-terminology}
  Let $E_c \cong \bP^1$ be the components of the exceptional divisor
  $E$, and let $E_{ab} \coloneqq E_a \cup E_{a+1} \cup \cdots \cup
  E_b$. Let
  \[ \cR \coloneqq \cO_{E_{ab}}(d_a, d_{a+1}, \ldots, d_b) \]
  denote the line bundle on $E_{ab}$ such that $\cR|_{E_c} =
  \cO_{\bP^1}(d_c)$ for each $a \le c \le b$, with trivial gluing at
  nodes.
  \begin{itemize}
  \item A {\it rod} is a connected collection of boxes forming a line
    bundle of the form $\cR$, with any linearization.
  \item The {\it length} of the rod $\cR$ is $b-a+1$, and its degrees
    are $\vec d(\cR) \coloneqq (\deg_{E_c} \cR)_{a \le c \le b}$.
  \item In the chart $U_a$, we say the rod is pointing {\it
    rightward}; analogously, it points {\it leftward} in the chart
    $U_b$. Slightly abusing terminology, we say the leftmost/rightmost
    boxes in the rod {\it generate} the rod.
  \end{itemize}
\end{definition}

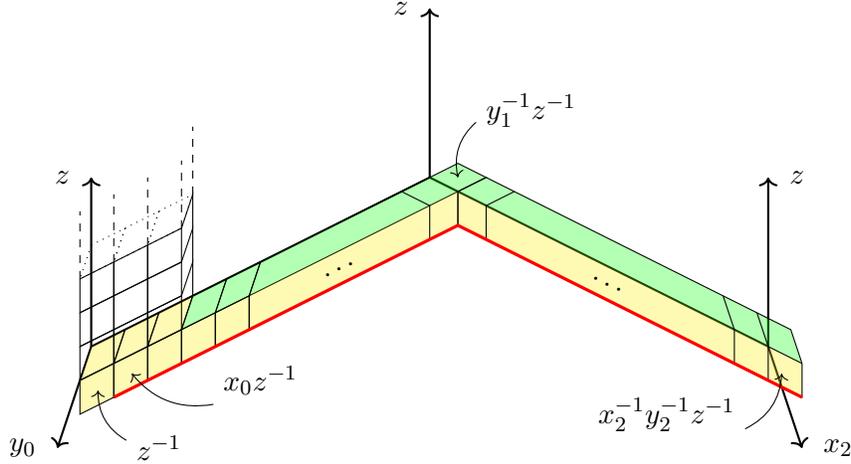
\begin{figure}[h]
  \centering
  \begin{tikzpicture}[rod/.style={fill=yellow,fill opacity=0.3}, scale=0.45]
    \draw[thick,->] (0,0) -- (-10,-5) -- (-11,-8) node[label=left:$y_0$]{};
    \draw[thick,->] (0,0) -- (10,-5) -- (11,-8) node[label=right:$x_2$]{};
    \draw[thick,->] (0,0) -- (0,5) node[label=left:$z$]{};
    \draw[thick,->] (-10,-5) -- (-10,0) node[label=left:$z$]{};
    \draw[thick,->] (10,-5) -- (10,0) node[label=right:$z$]{};

    \foreach \x/\y in {-10/-5,-9/-4.5,-8/-4} {
      \draw[rod] (\x,\y) -- ++(1,0.5) -- ++(-1/3,-1) -- ++(-1,-0.5) -- ++(1/3,1);
    };
    \foreach \x/\y in {-7/-3.5,-6/-3} {
      \draw[rod,fill=green] (\x,\y) -- ++(1,0.5) -- ++(-1/3,-1) -- ++(-1,-0.5) -- ++(1/3,1);
    };
    \foreach \x/\y in {-10/-2,-9/-1.5,-8/-1} {
      \draw[dotted] (\x,\y) -- ++(1,0.5) -- ++(-1/3,-1) -- ++(-1,-0.5) -- ++(1/3,1);
      \draw[dashed] (\x-1/3,\y-1) -- ++(0,2);
    };
    \draw[dashed] (-7-1/3,-0.5-1) -- ++(0,2);
    \draw[dashed] (-7,0.5-1) -- ++(0,2);
    \foreach \x in {0,1,2} {
      \foreach \y in {-4,-3,-2} {
        \draw (-10+\x-1/3,\y-1+\x/2) -- ++(1,0.5) -- ++(0,-1) -- ++(-1,-0.5) -- ++(0,1);
      };
    };
    \foreach \x/\y in {0,1,2,3,4} {
      \draw[rod] (-10+\x-1/3,-6+\x/2) -- ++(1,0.5) -- ++(0,-1) -- ++(-1,-0.5) -- ++(0,1);
    };
    \foreach \y in {-3.5,-2.5,-1.5} {
      \draw (-7,\y) -- ++(0,1) -- ++(-1/3,-1);
    };
    \draw[rod,fill=green] (-5,-2.5) -- ++(-1/3,-1) -- ++(7,3.5) -- ++(-5/6,5/12) -- ++(-6,-3);
    \draw[rod] (-5-1/3,-3.5) -- ++(0,-1) -- ++(6+1/6,3+1/12) -- ++(0,1);
    \draw[rod,fill=green] (5/6,-5/12) -- ++(5/6,5/12) -- ++(9,-4.5) -- ++(1/3,-1);
    \draw[rod] (5/6,-5/12) -- ++(0,-1) -- ++(10+1/6,-5-1/12) -- ++(0,1);
    \foreach \x in {3,4} {
      \draw (6+\x,-4-\x/2) -- ++(0,1) -- ++(-1/3,1);
    };
    \draw (-5/6,-5/12) -- ++(5/6,-5/12) -- ++(0,-1);
    \draw (3*5/6,-5/12) -- ++(-5/6,-5/12) -- ++(0,-1);
    \node[rotate=26] at (-2.6,-2.7) {$\cdots$};
    \node[rotate=-26] at (5.3,-3.2) {$\cdots$};
    \draw[very thick,red] (-9-1/3,-6.5) -- (5/6,-1-5/12) -- (11,-6.5);
    \draw (10,-5) -- (11,-5.5);

    \node (b1) at (-8,-8) {$z^{-1}$};
    \draw[->] (b1) to[bend left] (-9.8,-6.3);
    \node (b2) at (-5,-6) {$x_0 z^{-1}$};
    \draw[->] (b2) to[bend left] (-8.8,-5.8);
    \node (b3) at (3,2) {$y_1^{-1}z^{-1}$};
    \draw[->] (b3) to[bend right] (5/6,0);
    \node (b4) at (7,-7) {$x_2^{-1} y_2^{-1} z^{-1}$};
    \draw[->] (b4) to[bend right] (10.4,-5.8);
  \end{tikzpicture}
  \caption{A $\sT$-fixed point in $\PiPairs_{(3),\emptyset,\emptyset}(\cA_2)_{\text{nonsing }\infty}$}
  \label{fig:bs-rod-example}
\end{figure}

In Figure~\ref{fig:bs-rod-example}, the boxes colored yellow form a
rod $\cR$ of length $2$ and degree $\vec d = (1,-1)$ generated by the
box $\cube = z^{-1} \in U_0$. Alternatively, $\cR$ is the equivariant
line bundle $\cO_E(1,-1)$ with linearization
\begin{align*}
  \cR\big|_{p_0} &= z^{-1} \\
  \cR\big|_{p_1} &= y_1^{-1}z^{-1} = xy^{-2}z^{-1} \\
  \cR\big|_{p_2} &= x_2^{-1}y_2^{-1}z^{-1} = x^{-1}y^{-1}z^{-1}.
\end{align*}
To relate weights with degrees, it is helpful to recall that on a
$\bP^1$ whose coordinate around $0$ has weight $w$,
\[ \cO_{\bP^1}(d)\big|_\infty = w^d \cO_{\bP^1}(d)\big|_0. \]
For example, $\cR\big|_{p_2} = x_1^{-1} \cR\big|_{p_1}$ demonstrates
that $\cR$ has degree $-1$ on $E_2$.

\begin{definition}
  A box in a rod $\cR$ is {\it exposed} if it is $z$-torsion, i.e. it
  generates only a finite stack of additional boxes in the $z$
  direction. The exposed part of a rod forms possibly multiple
  disconnected rods.
  \begin{itemize}
  \item A {\it standard rightward rod} is a rod with degrees $\vec d =
    (0, 0, \ldots, 0, -1)$ and non-exposed generator on the left.
  \item A {\it standard leftward rod} is a rod with degrees $\vec d =
    (-1, 0, 0, \ldots, 0)$ and non-exposed generator on the right.
  \end{itemize}
  Standard rods are in some sense the minimal ones in $\cat{T}$; it is
  easiest to work only with standard rods for the combinatorial
  description of general $\cF$.
\end{definition}

In Figure~\ref{fig:bs-rod-example}, the exposed boxes in the rod $\cR$
are shaded green on one face, and form a rod $\cR' = \cO_E(-2,-1)
\subset \cR$ themselves (with appropriate linearization). The rod
$\cR'$ is generated by the box $\cube = x_0^3z^{-1}$. Note that $\cR$
also has the sub-rod generated by $\cube = x_0z^{-1}$. This sub-rod is
underlined in red in the figure, and is a standard rightward rod of
length $2$.

\subsubsection{}

All $0$-dimensional sheaves on $Y$ are in $\cat{T}$, so sheaves in
$\cat{T}^\perp$ are pure of dimension $1$. In particular, $\cF$ and
its subsheaf $\cO_C$ are pure. Consequently, $C$ is Cohen--Macaulay,
as in DT and PT theory. Thus $\cO_C$ contains {\it all} the boxes in
the infinite legs $L_a$, plus possibly some internal legs, and nothing
more. For consistency, we view all internal legs in $\cO_C$ as rods as
well.

While for stable pairs one can show that $C$ and the scheme-theoretic
support $C_{\cF} \coloneqq \supp(\cF)$ coincide \cite[Lemma
  1.6]{Pandharipande2009}, this is very much not true for $\pi$-stable
pairs, and in general
\[ C^{\text{red}} \neq C_{\cF}^{\text{red}}. \]
This is because $C_{\cF}$ receives contributions from $\cQ$, which may
contain rods that are not present in $\cO_C$.

\subsubsection{}
\label{sec:bs-local-model}

Consider the case when all external legs are trivial except for
$\lambda_a = \square$. Then we can give an explicit description of all
possible $\cF$. Importantly, the general case can be reduced to
understanding this special case.

\begin{definition}
  A {\it local model} for $\lambda_a = \square$ is an indecomposable
  sheaf $\cF \in \cat{Coh}_{\sT}(Y)$ which contains all boxes
  \[ \{\cube_k \coloneqq z^k \in U_a\}_{k \ge M} \]
  for some $M \in \bZ$ (possibly with multiplicity), such that $\cF$
  consists only of $\{\cube_k\}_{k \ge M}$ and finitely many rods
  generated by these boxes and:
  \begin{itemize}
  \item if $\cube_k$ generates a rod, it must be a standard left- or
    right-ward rod;
  \item if $\cube_k$ generates both a left- and right-ward rod, then
    $\cube_k$ has multiplicity $2$ and the $\cube_k$-weight space in
    $\cF(U_a)$ is
    \[ [z^k]\cF(U_a) = \bC v_L \oplus \bC v_R, \]
    where $v_L$ generates the left-ward rod and $v_R$ generates the
    right-ward rod;
  \item if $\ell \in \bZ$ is the largest integer such that
    $\cube_\ell$ has multiplicity $2$, then the multiplication-by-$z$
    map is
    \[ [z^\ell]\cF(U_a) \xrightarrow{\begin{psmallmatrix} 1 & 1 \end{psmallmatrix}} [z^{\ell+1}]\cF(U_a). \]
  \end{itemize}
  Let $\level(\cF) \coloneqq \ell$.
\end{definition}

Figure~\ref{fig:bs-local-model} is the box diagram for a general local
model $\cF$; that $\cF$ is a valid $\cO_Y$-module forces the
configuration of standard rods to be increasing in length as the
degree in $z$ increases, as shown. For visual clarity, we have
flattened the toric diagram and all boxes (cf.
Figure~\ref{fig:bs-rod-example}).

\begin{figure}[h]
  \centering
  \begin{tikzpicture}[scale=2.9]
    \foreach \x [count=\i] in {0,...,5} {
      \draw[->,thick] (\x,-1) -- (\x,1);
    };
    \node[above] at (0,1) {$L_0$};
    \node[above] at (1,1) {$L_{a-1}$};
    \node[above] at (2,1) {$L_a$};
    \node[above] at (3,1) {$L_{a+1}$};
    \node[above] at (4,1) {$L_{a+2}$};
    \node[above] at (5,1) {$L_{m-1}$};
    \foreach \x/\n in {0.1/1, 1.1/3, 2/6, 3/4, 4/2} {
      \foreach \y in {1,...,\n} {
        \draw (\x,-0.1*\y) rectangle ++(1,0.1);
        \foreach \cx in {0.1,0.2,0.3,0.7,0.8,0.9} {
          \draw (\x+\cx,-0.1*\y) -- (\x+\cx,-0.1*\y+0.1);
        };
        \node at (\x+0.5,-0.1*\y+0.05) {$\cdots$};
      };
    };
    \draw (2.1,0) -- (2.1,0.8);
    \foreach \y in {0.1,0.2,0.3,...,0.79} {
      \draw (2,\y) -- (2.1,\y);
    };
    \foreach \y in {-0.1,-0.2,-0.3} {
      \node at (2.05,\y+0.05) {\small $2$};
    }
    \fill[color=white] (0.35,-0.5) rectangle (0.65,0.5);
    \fill[color=white] (4.35,-0.5) rectangle (4.65,0.5);
    \node at (0.5,-0.05) {$\cdots$};
    \node at (4.5,-0.1) {$\cdots$};
  \end{tikzpicture}
  \caption{(Flattened) box diagram of a $\sT$-fixed pair with one non-zero leg $\lambda_a = \square$}
  \label{fig:bs-local-model}
\end{figure}
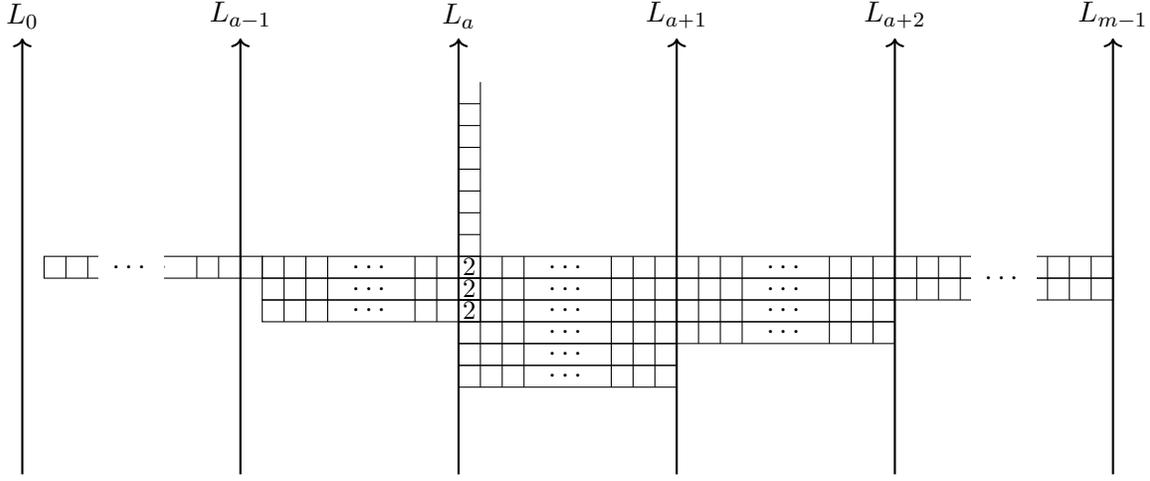

\begin{proposition}[Local model] \label{prop:bs-local-model}
  Let $\cF \in \PiPairs_{\vec\lambda}^{\sT}$ where $\vec\lambda =
  (\cdots, \emptyset, \emptyset, \lambda_a = \square, \emptyset,
  \emptyset, \cdots)$. Then $\cF$ is a local model for $\square \in
  \lambda_a$ with $\level(\cF) < 0$.
\end{proposition}

For readers familiar with PT theory, the allowed box configurations
for $\cF$ are the same as for the PT $1$-leg vertex for $\lambda =
\square$, except now boxes are allowed to emit standard
left-/right-ward rods in such a way that $\cF$ remains a valid
$\cO_Y$-module. When a box emits {\it both} left- and right-ward rods,
we will see that it necessarily gains a multiplicity.

\subsubsection{}

To begin the proof of Proposition~\ref{prop:bs-local-model}, it is
important to characterize the degrees/lengths of rods that can appear
in $\cQ$ and therefore in $\cF$. Let $\cQ \in \cat{T}$ be $\sT$-fixed,
and let $\cR$ be a rod in $\cQ$. The condition that $\cQ \in \cat{T}$
means that $H^1(E, \cR) = 0$.

\begin{lemma} \label{lem:bs-rod-degree-lower-bound}
  Let $\vec d(\cR) = (d_a, d_{a+1}, \ldots, d_b)$. Then $H^1(E, \cR) =
  0$ iff all $d_c \ge -1$ with equality for at most one $d_c$.
\end{lemma}

\begin{proof}
  The exceptional divisor $E$ is a nodal chain of its components $E_c
  \cong \bP^1$. Via normalization, the rod $\cR$ fits into an exact
  sequence
  \begin{equation}
    0 \to \cR \to \bigoplus_{c=a}^b \cR_c \to \cP \to 0
  \end{equation}
  where $\cR_c \cong \cO(d_c)$ is supported only on $E_c \cong \bP^1$
  and $\cP$ is supported only at the nodes $p_a$. The desired result
  follows from a standard argument using the associated long exact
  sequence.
\end{proof}

\subsubsection{}

The condition that $\cF \in \cat{T}^\perp$ also imposes a restriction
on the lengths of rods. This is because no subsheaf $\cE$ of an
exposed rod can be an element of $\cat{T}$. Otherwise $\cF$ would
contain the subsheaf generated by $\cE$, contradicting $\cF \in
\cat{T}^\perp$.

\begin{lemma} \label{lem:bs-rod-degree-upper-bound}
  Let $\cE$ be an exposed rod of length $\ell$. Then it is a subsheaf
  of the maximal exposed rod $\cE_{\text{max},\ell}$ of length $\ell$,
  where:
  \begin{itemize}
  \item if $\ell > 1$, then $\vec d(\cE_{\text{max},\ell}) = (-1, \,
    0, 0, \ldots, 0, 0, -1)$;
  \item if $\ell = 1$, then $\vec d(\cE_{\text{max},\ell}) = (-2)$.
  \end{itemize}
\end{lemma}
  
\begin{proof}
  Follows from the above discussion and
  Lemma~\ref{lem:bs-rod-degree-lower-bound} describing elements of
  $\cat{T}$. The maximal exposed rods $\cE_{\text{max},\ell}$ those
  which are maximal under inclusion, i.e. such that increasing the
  degree on any component creates a sub-sheaf in $\cat{T}$.
\end{proof}

When a sheaf $\cF$ is pure of dimension $1$, violating this lemma is
the only way it can fail to be in $\cat{T}^\perp$. Note that some care
must be taken when there are rods consisting of boxes with
multiplicity. The following example shows that, in general, there
cannot be two rods which start at the same box.

\begin{example} \label{ex:bs-exposed-rod-multiplicity}
  Figure~\ref{fig:bs-exposed-rod-unstable} contains an exposed rod of
  degree $-1$, not just exposed rods of degree $-2$. One way to see
  this exposed rod is to write down the vector space at each box along
  with the relevant module maps. In this case, the $\cO_Y$-module
  structure is uniquely determined by the box configuration: one can
  always rescale the $\bC^2$ in the rod so that the map $\bC^2 \to
  \bC$ is as specified. Then the exposed rod of degree $-1$ is
  generated by the vector $\begin{psmallmatrix} 1
    \\ -1 \end{psmallmatrix}$ at each $\bC^2$.
\end{example}

\begin{figure}[h]
  \centering
  \begin{subfigure}[b]{0.4\textwidth}
    \centering
    \begin{tikzpicture}[scale=0.45]
      \draw[thick,->] (0,0) -- (0,4);
      \draw[thick,->] (10,0) -- (10,4);
      \draw[thick] (0,0) -- (10,0);
      \draw[thick,->] (0,0) -- (-1,-1);
      \draw[thick,->] (10,0) -- (11,-1);
      
      \draw (-0.5,-1.5) rectangle (0.5,2.5);
      \draw (-0.5,-1.5) rectangle (9.5,-0.5);
      \draw (1,0) -- (1,3);
      \foreach \x/\y in {1/0, 2/0, 3/0, 1/1, 1/2, 1/3, 0/3} {
        \draw (\x,\y) -- (\x-0.5,\y-0.5);
      };
      \foreach \x/\y in {7/0, 8/0, 9/0} {
        \draw (\x,\y) -- (\x+0.5,\y-0.5);
      };
      \foreach \x in {1.5,2.5,7.5,8.5} {
        \draw (\x,-0.5) -- (\x,-1.5);
      };
      \foreach \x/\y in {0.5/0.5, 0.5/1.5, 1/3} {
        \draw (\x,\y) -- (\x-1,\y);
      };
      \foreach \x in {0,1,2,8,9} {
        \node at (\x,-1) {$2$};
      }
      \node at (5,-1) {$\cdots$};
    \end{tikzpicture}
    \caption{Box diagram}
  \end{subfigure}%
  \hspace{1em}%
  \begin{subfigure}[b]{0.5\textwidth}
    \centering
    \[ \begin{tikzcd}[row sep=small, column sep=small, ampersand replacement=\&]
      \vdots \\
      \bC \ar{u} \\
      \bC^2 \ar{u}{\begin{psmallmatrix} 1 & 1 \end{psmallmatrix}} \ar{r} \& \bC^2 \ar{r} \& \cdots \& \bC^2 \ar{l} \& \bC^2 \ar{l}
    \end{tikzcd}. \]
    \vspace{-4mm}
    \caption{$\cO_Y$-module structure}
  \end{subfigure}
  \caption{Non-$\pi$-stable: contains a degree $-1$ exposed rod}
  \label{fig:bs-exposed-rod-unstable}
\end{figure}
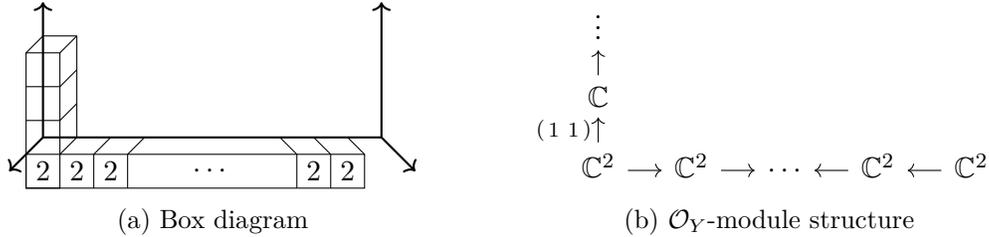

\subsubsection{}
\label{sec:bs-boxes-rods-assembly}

The next step is to understand the manner in which $\cF$ is built from
boxes and rods. The following lemma shows that, aside from the
presence of rods, $\cF(U_a)$ is essentially a PT $1$-leg vertex. Let
$\cF^\rb \subset \cF$ be the subsheaf consisting only of boxes which
are not part of rods, and similarly for $\cQ^\rb$. For short, let
$\cF_a$ denote $\cF(U_a)$ and similarly for other sheaves.

\begin{lemma} \label{lem:bs-boxes-module-structure}
  In the chart $U_a$, the module $\cF^\rb(U_a)$ is a $\sT$-invariant
  sub-module of the localization
  \[ M_a \coloneqq (L_a)_z = \bC[x_a, y_a, z, z^{-1}]/I_{\lambda_a} \]
  such that $z^n L_a \subset \cF^\rb(U_a)$ for $n \gg 0$.
\end{lemma}

\begin{proof}
  The key observation is that $\cF^{\rb}_a$ is $z$-torsion-free. This
  is because boxes in $\cF^{\rb}_a$ are necessarily torsion in the
  $x_a$ and $y_a$ directions, so if a box were also $z$-torsion then
  $\cF$ would fail to be pure of dimension $1$.

  The only way for a box $\cube \in \cF_a$ to be $z$-torsion-free is
  to have $z^n \cdot \cube \in L_a$ for $n \gg 0$. Hence the
  composition
  \[ \cF^\rb_a \hookrightarrow \cF_a \to \Hom\left(\Hom(\cF_a, L_a), L_a\right) \]
  is generically an isomorphism, and therefore an inclusion by the
  purity of $\cF^{\rb}_a$. To compute $\Hom\left(\Hom(\cF_a, L_a),
  L_a\right)$, apply $\Hom(-, L_a)$ to $0 \to (\cO_C)_a \to \cF_a \to
  \cQ_a \to 0$. Since everything in $\cQ_a$ is $z$-torsion, this
  yields an inclusion
  \[ 0 \to \Hom(\cF_a, L_a) \to \Hom((\cO_C)_a, L_a) = L_a. \]
  So $\Hom(\cF_a, L_a)$ is an ideal $I_Z \subset L_a$. The subscheme
  $Z$ it cuts out must be zero-dimensional since $\supp \cQ^\rb_a$ is.
  Thus $\Hom(I_Z, L_a)$ is a sub-module of $M_a$ satisfying the
  specified criterion.
\end{proof}

\begin{remark}
  Let $\cF^z$ be the quotient of $\cF$ consisting of boxes which are
  $z$-torsion-free. Then there is a short exact sequence
  \[ 0 \to \cF^\rb \to \cF^z \to \cF^{z,\rr} \to 0, \]
  where $\cF^{z,\rr}$ consists of $z$-torsion-free boxes which are
  part of rods. Every rod must contain some boxes which are
  $z$-torsion-free, otherwise the sub-sheaf of $\cF$ generated by the
  pre-image of the rod is also an element in $\cat{T}$, contradicting
  $\cF \in \cat{T}^\perp$.

  As with boxes in $\cF^b$, boxes $\cube \in \cF^{z,\rr}_a$ must have
  $z^n \cdot \cube \in L_a$ for $n \gg 0$. However, $\cF^z_a$ is not
  necessarily a sub-module of $M_a$; the proof of
  Lemma~\ref{lem:bs-boxes-module-structure} fails because $\cF^z$ is
  not a subsheaf of $\cF$. This therefore allows for boxes with
  multiplicity to occur in $\cF^{z,\rr}$.
\end{remark}

\subsubsection{}
\label{sec:bs-local-model-proof}

We can now give the proof of Proposition~\ref{prop:bs-local-model},
which describes all possible local models $\cF$.

\begin{proof}[of Proposition~\ref{prop:bs-local-model}]
  By Lemmas~\ref{lem:bs-rod-degree-lower-bound} and
  \ref{lem:bs-rod-degree-upper-bound}, the only rods that can appear
  in a local model are standard rods. For example, the simplest case
  of a right-ward rod of length $1$ must have degree $\ge -1$ while
  its exposed part must have degree $\le -2$. Therefore it must be
  exactly of degree $-1$, i.e. a standard rod.

  By the discussion of section~\ref{sec:bs-boxes-rods-assembly}, these
  standard rods must be generated by boxes $\cube$ which are
  $z$-torsion-free, which therefore eventually generate $L_a$. Since
  $\cF$ is an $\cO_Y$-module, once a standard rod appears, all rods
  which are stacked on top of it in the $z$ direction must have
  non-decreasing length. Finally, note that a single box cannot
  generate both leftward and rightward rods; such a rod would violate
  Lemma~\ref{lem:bs-rod-degree-lower-bound} and is not in $\cat{T}$.
  Hence if a leftward and a rightward standard rod start at the same
  box, that box has multiplicity $2$. Taking all these constraints
  into account, $\cF$ must be as described. The restriction on
  $\level(\cF)$ arises because there must be an inclusion $\cO_C
  \hookrightarrow \cF$, where here $C = \{p_a\} \times \bC_z$.
\end{proof}

\subsubsection{}
\label{sec:bs-local-model-filtration}

The description in Proposition~\ref{prop:bs-local-model} of the local
model is not specific to the $\square_0 \coloneqq x_a^0 y_a^0$ column
in $L_a$. In general, a local model for $\square = x_a^i y_a^j \in
\lambda_a$ is given by multiplying a local model for $\square_0$ by
$\square$.

\begin{proposition} \label{prop:bs-local-model-filtration}
  Let $[\cO_Y \to \cF]$ be a $\sT$-fixed $\pi$-stable pair with legs
  $\vec\lambda$. Then there is a (not necessarily unique) filtration
  of length $n = |\vec\lambda|$
  \[ \emptyset = \cF_n \subset \cF_{n-1} \subset \cdots \subset \cF_0 = \cF \]
  whose associated graded pieces are local models for each $\square
  \in \lambda_a$ (for $0 \le a < m$).
\end{proposition}

\begin{proof}
  Pick the smallest $a$ such that $\lambda_a$ is non-empty. Pick
  $\square \in \lambda_a$ not supporting any other squares in
  $\lambda_a$. Let $\cG$ be the {\it largest} local model at $\square$
  which is still a subsheaf of $\cF$. By largest, we mean with
  standard leftward and rightward rods of maximal possible length
  generated at every box $z^k \cdot \square$. The claim is that the
  resulting quotient $\cF' \coloneqq \cF/\cG$ is still $\pi$-stable.
  Then $\cF'$ contains strictly fewer columns in its infinite legs,
  and induction with Proposition~\ref{prop:bs-local-model} as the base
  case finishes the proof.

  Let $0 \to \cO_{C'} \to \cF' \to \cQ' \to 0$ be the resulting pair.
  There are two ways $\cF'$ can fail to be $\pi$-stable: $\cQ' \notin
  \cat{T}$ or $\cF' \notin \cat{T}^\perp$. Clearly $\cQ' \in \cat{T}$
  iff $\cQ \in \cat{T}$, since all standard rods satisfy
  Lemma~\ref{lem:bs-rod-degree-lower-bound}. So suppose $\cF' \notin
  \cat{T}^\perp$. Then $\Hom(\cT, \cF') \neq 0$ for some $\cT \in
  \cat{T}$. If $\dim \supp \cT = 0$, then the image of $\cT$ in $\cF'$
  must consist of boxes extending some standard rod that was removed
  as part of $\cG$. Then there is an exposed rod in $\cF$ violating
  Lemma~\ref{lem:bs-rod-degree-upper-bound}, contradicting $\cF \in
  \cat{T}^\perp$.

  The only remaining possibility is that $\dim \supp \cT = 1$. Namely,
  removing $\cG$ created an exposed rod $\cR'$ in $\cF'$ whose
  pre-image in $\cF$ is {\it not} an exposed rod. Let $\cube$ denote
  the box in $\cR'$ lying in the column corresponding to $\square$. By
  Lemma~\ref{lem:bs-rod-degree-upper-bound}, making $\cube$ into a
  $z$-torsion-free box turns $\cR'$ into a standard rod $\cR$ in
  $\cF$, which should have been removed as part of $\cG$. The only way
  for it to still be part of $\cF'$ is if it overlaps with a standard
  rod already in $\cG$, in the form of boxes with multiplicity in
  $\cR$. As in Example~\ref{ex:bs-exposed-rod-multiplicity}, these
  multiplicity $> 1$ boxes therefore contain an exposed rod violating
  Lemma~\ref{lem:bs-rod-degree-upper-bound}, again contradicting $\cF
  \in \cat{T}^{\perp}$.
\end{proof}

\subsubsection{}

The converse to Proposition~\ref{prop:bs-local-model-filtration} is
also true, and important for us when we want to build $\pi$-stable
pairs out of specific local models.

\begin{proposition} \label{prop:bs-local-model-converse}
  Suppose $\cF \in \cat{Coh}(Y)$ admits an inclusion $\cO_C
  \hookrightarrow \cF$ where $C$ has external legs $\vec\lambda$, and
  that it also admits a filtration
  \[ \emptyset = \cF_n \subset \cF_{n-1} \subset \cdots \subset \cF_0 = \cF \]
  of length $n = |\vec\lambda|$ whose associated graded pieces are
  local models for each $\square \in \lambda_a$ (for $0 \le a < m$).
  Then the induced
  \[ s\colon \cO_Y \twoheadrightarrow \cO_C \hookrightarrow \cF \]
  is a $\pi$-stable pair.
\end{proposition}

\begin{proof}
  We engineered the notion of a local model to satisfy the constraints
  of Lemma~\ref{lem:bs-rod-degree-lower-bound} and
  Lemma~\ref{lem:bs-rod-degree-upper-bound} for $\cQ$ and $\cF$ to be
  in $\cat{T}$ and $\cat{T}^\perp$ respectively.
\end{proof}

\subsection{T-fixed loci}
\label{sec:bs-fixed-loci-structure}

\subsubsection{}

In general, the $\sT$-fixed loci in $\PiPairs$ can have arbitrarily
large dimension. Since the virtual tangent space \eqref{eq:bs-Tvir}
for $\PiPairs$ depends only on the box configuration, and there is no
ambiguity in forming a local model from a given box configuration,
these fixed loci must arise from the ambiguity in reconstructing the
$\cO_Y$-module $\cF$ from the local models. More precisely, note that
for fixed $a$, the local models for each $\square \in \lambda_a$ form
a $\cO_Y$-module $\cF_a$ uniquely (cf. the proof of
Lemma~\ref{lem:bs-quasimap-module-condition}), essentially because the
$\cO_Y$-module maps between boxes in different local models are
induced from those of the leg $L_a$. By
Proposition~\ref{prop:bs-local-model-filtration}, the ambiguity in
reconstructing $\cF$ is therefore solely in how to attach different
$\cF_a$ to each other. In other words, given fixed $\cF_a$ for $0 \le
a < m$, the space of all possible $\cF$ is controlled by the groups
$\Ext^1(\cF_a, \cF_b)$ for $0 \le a, b < m$.

\subsubsection{}

In particular, all non-trivial extensions between $\cF_a$ and $\cF_b$
with $a < b$ arise when there are two boxes $\cube \in \cF_a$ and
$\cube \in \cF_b$ with the same $\sT$-weight, and either $\cube \in
\cF_a$ generates a standard rightward rod or $\cube \in \cF_b$
generates a standard leftward rod. The simplest example is as follows.

\begin{example} \label{ex:bs-nontrivial-fixed-loci}
  Consider the vertex for $(\square, \square) \in \cA_1^{\sT}$. One
  $\sT$-fixed component for this vertex has the following box
  configuration.
  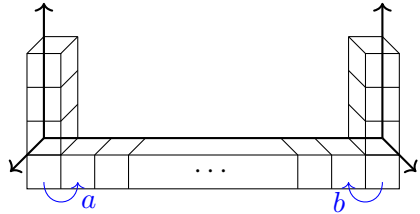
\begin{figure}[h]
    \centering
    \begin{tikzpicture}[scale=0.45]
      \draw[thick,->] (0,0) -- (0,4);
      \draw[thick,->] (10,0) -- (10,4);
      \draw[thick] (0,0) -- (10,0);
      \draw[thick,->] (0,0) -- (-1,-1);
      \draw[thick,->] (10,0) -- (11,-1);

      \draw (-0.5,-0.5) rectangle (0.5,2.5);
      \draw (9.5,-0.5) rectangle (10.5,2.5);
      \draw (-0.5,-1.5) rectangle (10.5,-0.5);
      \draw (1,0) -- (1,3);
      \draw (9,0) -- (9,3);
      \foreach \x/\y in {1/0, 2/0, 3/0, 1/1, 1/2, 1/3, 0/3} {
        \draw (\x,\y) -- (\x-0.5,\y-0.5);
      };
      \foreach \x/\y in {7/0, 8/0, 9/0, 9/1, 9/2, 9/3, 10/3} {
        \draw (\x,\y) -- (\x+0.5,\y-0.5);
      };
      \foreach \x in {0.5,1.5,2.5,7.5,8.5,9.5} {
        \draw (\x,-0.5) -- (\x,-1.5);
      };
      \foreach \x/\y in {0.5/0.5, 0.5/1.5, 1/3, 10.5/0.5, 10.5/1.5, 10/3} {
        \draw (\x,\y) -- (\x-1,\y);
      };
      \node at (5,-1) {$\cdots$};
      \draw[blue,->] (0,-1.3) to[looseness=2,in=-90,out=-90] node[label=right:$a$]{} (1,-1.3);
      \draw[blue,->] (10,-1.3) to[looseness=2,in=-90,out=-90] node[label=left:$b$]{} (9,-1.3);
    \end{tikzpicture}
    \caption{Box configuration for a positive-dimensional fixed locus}
  \end{figure}
  Here the $a$ and $b$ arrows represent $\cO_Y$-module maps. Note that
  every point $[a : b] \in \bP^1$ yields a distinct $\cO_Y$-module
  $\cF$. The points $0 \in \bP^1$ and $\infty \in \bP^1$ correspond to
  the direct sum of two local models: one with a single box and
  another with a standard leftward/rightward rod. The point $[1 : 1]$
  corresponds to $z^{-1} \cO_C$, where $C$ is the CM curve with the
  two specified external legs and one internal leg.
\end{example}

\subsubsection{}

The existence of positive-dimensional fixed loci should not be
surprising, since already in PT theory fixed loci can be arbitrary
products of $\bP^1$'s \cite[Theorem 1]{Pandharipande2009a}. Note
however that it is heavily dependent on the underlying threefold
geometry being of the form $\tot(\cO_{\bP^1} \oplus \cO_{\bP^1}(d))$,
i.e. having a trivial $z$-direction. When this is not the case, such
as for the conifold $\tot(\cO(-1) \oplus \cO(-1))$, the fixed locus
should consist of only isolated fixed points.

That the fixed loci are always proper will follow from the properness
of the corresponding quasimap fixed loci and the isomorphism of fixed
loci constructed in the BS/quasimaps correspondence. It seems difficult to
prove properness directly in $\PiPairs$.

\section{Two Hilbert schemes}

\subsection{As Nakajima quiver varieties}
\label{sec:two-hilbert-schemes}

\subsubsection{}

Let $Q$ be a quiver with vertex set $I$. The associated Nakajima
quiver variety $\cM_{\vec\theta}$ depends on a GIT stability parameter
${\vec\theta} \in \bR^I$. We briefly review the construction.

\begin{definition}[{\cite[Section 2]{Nakajima1994}}, {\cite[Section 5]{Ginzburg2012}}]
  For two $I$-graded vector spaces $V = \bigoplus_{i \in I} V_i$ and
  $W = \bigoplus_{i \in I} W_i$, let
  \begin{equation} \label{eq:quiver-rep-vector-spaces}
    \cat{Rep}_Q^{\text{framed}} \coloneqq \bigoplus_{\text{edge } i \to j} \Hom(V_i, V_j) \oplus \bigoplus_{i \in I} \Hom(W_i, V_i)
  \end{equation}
  be the space of {\it framed} representations of the quiver $Q$,
  where $W$ is the framing. Then $T^*\cat{Rep}_Q^{\text{framed}}$ has
  a standard symplectic form, and the standard action by $G_V
  \coloneqq \prod_{i \in I} \GL(V_i)$ is Hamiltonian with moment map
  which we denote $\mu$. The GIT quotient
  \[ \cM_{\vec\theta} \coloneqq \mu^{-1}(0) \sslash_{\vec\theta} G_V \]
  is called the {\it algebraic} symplectic reduction. The resulting
  Nakajima quiver variety has components $\cM_{\vec\theta}(\vec v,
  \vec w)$ indexed by the dimension vectors $\vec v$ and $\vec w$ of
  $V$ and $W$.
\end{definition}

When ${\vec\theta}$ is generic, $\cM_{\vec\theta}$ is smooth and is a
fine moduli space of quiver representations. It therefore has a
universal family
\[ \rV = \bigoplus_{i \in I} \rV_i. \]
(The same goes for $\rW = \bigoplus_{i \in I} \rW_i$, which is a
trivial bundle and less interesting.)

\subsubsection{}
\label{sec:quiver-variety-stability-conditions}

By the general theory of GIT quotients, $\cM_{\vec\theta}(\vec v, \vec w)$
is actually independent of ${\vec\theta}$ within chambers in $\bR^I$ cut out
by a certain hyperplane arrangement. For Nakajima quiver varieties,
these chambers are exactly the Weyl chambers of the (generalized)
Kac--Moody algebra associated to $Q$. 

The varieties $\cM_{\vec\theta}$ and $\cM_{{\vec\theta}'}$ for ${\vec\theta}$ and
${\vec\theta}'$ in different stability chambers are related as follows. Let
$(B, i) \in \cat{Rep}_Q^{\text{framed}}$ and $(\bar B, j)$ be in the
cotangent fiber. Points in $\cM_{\vec\theta}$ are represented by tuples $(B,
\bar B, i, j)$ satisfying the moment map equation $\mu = 0$ and the
stability condition imposed by ${\vec\theta}$. There is an $S^1$-action
given by
\begin{equation} \label{eq:S1-hyperkahler-rotation}
  t \cdot (B, \bar B, i, j) \coloneqq (t B, t^{-1} \bar B, i, j)
\end{equation}
which respects the equation $\mu = 0$, but may change stability
chambers.

\begin{theorem}[{\cite[4.1.3]{Nagao2009}}] \label{thm:hyperkahler-rotation-diffeomorphism}
  For generic ${\vec\theta}, {\vec\theta}'$, the varieties $\cM_{\vec\theta}$ and
  $\cM_{{\vec\theta}'}$ are $S^1$-equivariantly diffeomorphic.
\end{theorem}

\begin{remark}
  For algebraic symplectic reductions, it is straightforward to verify
  that both ${\vec\theta}$ and $-{\vec\theta}$ yield isomorphic spaces. Due to
  this, our stability chambers/criteria may differ from convention by
  a sign.
\end{remark}

\subsubsection{}

The quiver $Q$ we are interested in arises from the McKay
correspondence, which associates to every surface singularity of type
ADE an affine Dynkin diagram of the corresponding type.

\begin{definition}
  If $\Gamma \subset \SL(2, \bC)$ defines the singularity, let $R_0 =
  \bC, R_1, \ldots, R_{m-1}$ denote the irreducible representations of
  $\Gamma$, and $\bC^2$ be the natural representation of $\Gamma$.
  Decompose
  \[ R_i \otimes \bC^2 \cong \bigoplus R_j^{\oplus a_{ij}}. \]
  Then the matrix $A \coloneqq (a_{ij})$ is the adjacency matrix for
  the Dynkin diagram.
\end{definition}

For $\cA_{m-1}$, the resulting quiver is the affine (or cyclic) type A
quiver in Figure~\ref{fig:cyclic-type-A-quiver}. After framing and
doubling, we view $T^*\cat{Rep}_Q^{\text{framed}}$ as the moduli of
representations of the quiver shown in
Figure~\ref{fig:cyclic-type-A-quiver-doubled-framed}. The dimension
vectors
\[ \vec v^0 \coloneqq (\dim R_0, \dim R_1, \ldots, \dim R_{m-1}), \quad \vec w^0 \coloneqq (1, 0, 0, \ldots, 0) \]
will be especially relevant. For $\cA_{m-1}$, clearly $\dim R_i = 1$
for all $i$. In Lie-theoretic language, $\vec v^0$ is the affine root
of the associated affine Lie algebra.

\begin{figure}[h]
  \centering
  \begin{subfigure}[b]{0.45\textwidth}
    \centering
    \begin{tikzpicture}[decoration={markings,mark=at position 0.5 with {\arrow{>}}}]
      \foreach \i in {0,1,...,5} {
        \node[vertex] (v\i) at (60*\i:1) {};
      }
      \foreach \i in {0,1,...,5} {
        \pgfmathtruncatemacro{\ii}{int(mod(\i+1,6))}
        \draw[postaction={decorate}] (v\i) to[bend right] (v\ii);
      }
    \end{tikzpicture}
    \caption{Affine type A quiver}
    \label{fig:cyclic-type-A-quiver}
  \end{subfigure}%
  ~
  \centering
  \begin{subfigure}[b]{0.45\textwidth}
    \centering
    \begin{tikzpicture}[decoration={markings,mark=at position 0.5 with {\arrow{>}}}]
      \node[vertex,shape=rectangle,minimum size=6pt] (w) at (2,0) {};
      \foreach \i in {0,1,...,5} {
        \node[vertex] (v\i) at (60*\i:1) {};
      }
      \draw[postaction={decorate}] (w) to[bend right] (v0);
      \draw[postaction={decorate}] (v0) to[bend right] (w);
      \foreach \i in {0,1,...,5} {
        \pgfmathtruncatemacro{\ii}{int(mod(\i+1,6))}
        \draw[postaction={decorate}] (v\i) to[bend right] (v\ii);
        \draw[postaction={decorate}] (v\ii) to[bend right] (v\i);
      }
    \end{tikzpicture}
    \caption{Doubled framed affine type A quiver}
    \label{fig:cyclic-type-A-quiver-doubled-framed}
  \end{subfigure}
  \caption{Quivers associated to the $A_5$ surface singularity}
\end{figure}
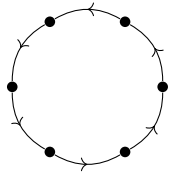
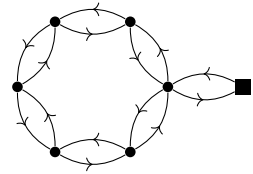

\subsubsection{}

For affine ADE quivers, there are two stability chambers we are
interested in. The first is
\begin{equation} \label{eq:orbifold-stability-chamber}
  C_+ \coloneqq \{{\vec\theta} \in \bR^m \; | \; \theta_i > 0 \text{ for all } 0 \le i < m\}.
\end{equation}
The following theorem rephrases a well-known generalization of the
description of $\Hilb(\bC^2)$ as a Nakajima quiver variety.

\begin{theorem}[{\cite[Theorem 4.4]{Nakajima1999}}] \label{thm:Hilb-orbifold-stability-chamber}
  For any integer $n \ge 0$ and any ${\vec\theta} \in C_+$, there is an
  isomorphism
  \[ \cM_{\vec\theta}(n \vec v^0, \vec w^0) \cong \Hilb^n([\bC^2/\Gamma]). \]
\end{theorem}

Here, the Hilbert scheme of $n$ points on the stack $[\bC^2/\Gamma]$
is equivalently the Hilbert scheme of collections of $nm$ points on
$\bC^2$ which are $\Gamma$-fixed. For this reason, it is sometimes
called the $\Gamma$-equivariant Hilbert scheme. In this language, a
$\Gamma$-invariant ideal $I \subset \bC[x,y]$ cutting out the $nm$
points corresponds to the quiver data
\begin{align} \label{eq:orbifold-quiver-data}
  \begin{split}
    V_k &= \Hom_\Gamma(R_k, \bC[x,y]/I) \\
    W_0 &= \bC.
  \end{split}
\end{align}
The quiver maps are given by $i(1) \coloneqq [1] \in \bC[x,y]/I$ and
$j \coloneqq 0$, and the $B, \bar B\colon V_i \to V_j$ are
multiplication by $x$ or $y$, as per the McKay correspondence. For the
affine type A quiver defining $\cA_{m-1}$, all counterclockwise arrows
are multiplication by $x$ and all clockwise arrows are multiplication
by $y$.

\subsubsection{}

The second stability chamber we want is given in terms of the
level-$0$ hyperplane
\[ D_\delta \coloneqq \{{\vec\theta} \in \bR^m \; | \; \vec v^0 \cdot {\vec\theta} = 0\}. \]
On $D_\delta$, there is a chamber structure defined by hyperplanes
$D_\alpha \coloneqq \{\alpha \cdot {\vec\theta} = 0\}$ for finite roots
$\alpha$; in particular,
\[ C \coloneqq \{{\vec\theta} \in D_\delta \; | \; \theta_i > 0 \text{ for all } 1 \le i < m\} \]
is a chamber. Let $C_-(m)$ be the unique chamber in $\bR^m$ lying on
the positive side $\{\vec v^0 \cdot {\vec\theta} > 0\}$ of $D_\delta$ with
$C$ as its face. For example, in type A
\begin{equation} \label{eq:regular-stability-chamber}
  (-m+1+\epsilon, 1, 1, \ldots, 1) \in C_-(m)
\end{equation}
for sufficiently small $\epsilon > 0$.

\begin{theorem}[{\cite[Theorem 4.2]{Nakajima2007}, \cite[Theorem 4.9]{Kuznetsov2007}}] \label{thm:Hilb-regular-stability-chamber}
  For any integer $n \ge 0$ and ${\vec\theta} \in C_-(m)$, there is an
  isomorphism
  \[ \cM_{\vec\theta}(n \vec v^0, \vec w^0) \cong \Hilb^n(S) \]
  where $S \to \bC^2/\Gamma$ is the minimal resolution.
\end{theorem}

Let $Z \subset \cA_{m-1}$ be a length-$n$ subscheme. Then the
corresponding quiver data is
\begin{align} \label{eq:regular-quiver-data}
  \begin{split}
    V_k &= H^0(\cA_{m-1}, \rV_k \otimes \cO_Z) \\
    W_0 &= H^0(\cA_{m-1}, \cO_{\cA_{m-1}}) = \bC,
  \end{split}
\end{align}
where $\rV = \bigoplus \rV_k$ is the universal family for
$\Hilb^1(\cA_{m-1}) = \cA_{m-1}$. The quiver maps between the $V_k$
are induced from the quiver maps of the universal family. For the
framing, $j = 0\colon V_0 \to W_0$ always while $i\colon W_0 \to V_0$
is more subtle. If $(V_0)_0 \subset V_0$ denotes the subspace of
trivial $\sT'$-weight, then for $\Hilb([\bC^2/\Gamma])$ we have $\dim
(V_0)_0 = 1$ with generator $i(1)$. But for $\Hilb(\cA_{m-1})$, it
turns out that $\dim (V_0)_0 > 1$ in general (see
section~\ref{sec:quiver-data-as-boxes}), in which case the framing is
given by the diagonal map $i(1) = (1, 1, \ldots, 1) \in (V_0)_0$.

\subsection{T-equivariant geometry}
\label{sec:hilbert-schemes-equivariant-geometry}

\subsubsection{}

Both $\Hilb([\bC^2/\Gamma])$ and $\Hilb(\cA_{m-1})$ inherit the
standard action of the torus $\sT'$ on $\bC^2$, from
section~\ref{sec:tori-definitions}. K-theoretic computations on these
spaces require understanding $\sT'$-fixed points and their tangent
spaces. In particular, for quasimap theory, it is productive to work
with both Hilbert schemes as moduli of quiver representations, and
therefore it is necessary to determine the quiver data $V = \bigoplus
V_k$ at each fixed point.

\subsubsection{}

The $\Gamma$-equivariant Hilbert scheme of $n$ points is a subscheme
\[ \Hilb^n([\bC^2/\Gamma]) \subset \Hilb^{nm}(\bC^2). \]
Hence $\sT'$-fixed points in $\Hilb^n([\bC^2/\Gamma])$ can be
identified with certain Young diagrams of size $nm$.

\begin{definition}
  The {\it color} or {\it content} of a square $\square = x^i y^j \in
  \lambda$ is
  \[ c(\square) \coloneqq i(\square) - j(\square) \in \bZ/m = \Gamma \]
  and records its $\Gamma$-weight. A {\it uniformly colored} Young
  diagram $\lambda$ has the same number of squares of each color.
\end{definition}

Let $\lambda \in \Hilb^n([\bC^2/\Gamma])^{\sT'}$ be a fixed point. The
$\Gamma$-invariance of $\lambda$ implies it is uniformly colored. From
section~\ref{eq:orbifold-quiver-data}, $V_k$ is therefore the
character of all color-$k$ boxes in $\lambda$. The quiver data for
$\lambda$ will be denoted
\[ V(\lambda) = \bigoplus_k V_k(\lambda). \]

\subsubsection{}

Similarly, a $\sT'$-fixed point $Z \in \Hilb^n(\cA_{m-1})$ is a
collection of $m$ Young diagrams $\lambda_0, \ldots, \lambda_{m-1}$,
one for each $\sT'$-fixed point $p_i \in \cA_{m-1}$, so that
\begin{equation} \label{eq:structure-sheaf-regular-Hilb}
  \cO_Z\big|_{U_a} = \bC[x_a, y_a] / I(\lambda_a).
\end{equation}
The ideal $I(\lambda_a)$ has generators $\{x_a^{i(\square)}
y_a^{j(\square)}\}_{\square \in \lambda_a}$. Since $Z$ has length $n$
we require $\sum_{a=0}^{m-1} |\lambda_a| = n$, but otherwise the
$\lambda_a$ can be arbitrary Young diagrams. From now on, we will
abuse notation and write
\[ Z = \vec\lambda \coloneqq (\lambda_0, \ldots, \lambda_{m-1}) \]
to denote the fixed point. The quiver data for $\vec\lambda$ will be
denoted
\[ V(\vec\lambda) = \bigoplus_a V_a(\vec\lambda). \]
To describe $V_a(\vec\lambda)$ combinatorially, using
\eqref{eq:regular-quiver-data}, it remains to determine the universal
family $\rV$ on $\cA_{m-1}$.

\subsubsection{}
\label{sec:universal-family-Am}

Since $\cA_{m-1}$ is toric, an easy way to determine (equivariantly)
the universal line bundles $\rV_k$ on $\cA_{m-1}$ is to identify their
weights at the fixed points $p_a$. This can be done from first
principles by explicitly determining the corresponding $\sT'$-fixed
quiver representations.

\begin{proposition}
  For $0 \le a, k < m$,
  \begin{equation} \label{eq:universal-family-Am}
    \rV_k\big|_{p_a} = \begin{cases} y^{m-k} & a < k \\ x^k & a \ge k. \end{cases}
  \end{equation}
\end{proposition}

Non-equivariantly, this means $\rV_0 = \cO_{\cA_{m-1}}$ and $\rV_k =
\cO_{E_k}(1)$ for $1 \le k < m$. This is in agreement with \cite[Lemma
  2.1]{Kapranov2000}, where it is shown that $\rV_k = \cO_{E_k}(1)
\oplus \cO_{E_k}^{\dim R_k - 1}$ for $k \neq 0$ in the general ADE
case.

\begin{proof}
  The dimension vector is $\vec v^0 = (1, 1, \ldots, 1)$, so all
  vector spaces $V_i$ are one-dimensional. In this setting, the key
  observation is that non-trivial cycles in the quiver, namely any
  composition of arrows which starts and ends at the same vertex $v$,
  must be zero. Such a non-zero cycle would have a non-trivial
  $\sT'$-weight which cannot be canceled by the $G_V$ freedom at $v$,
  and therefore cannot be $\sT'$-fixed. Hence only trees can be
  $\sT'$-fixed, and it suffices to figure out which maps $B_{ij}$ or
  $\bar B_{ij}$ are non-zero. (The $G_V$ freedom can be used to scale
  all maps in non-zero trees to $1$.)

  GIT stability for framed quiver representations has a reformulation
  in terms of sub-object stability, by work of King and Crawley-Bovey
  \cite[Section 3.2]{Ginzburg2012}. In our setting, we need to extend
  the stability condition by $\theta_\infty \coloneqq -{\vec\theta}
  \cdot \dim V$ to include the extra framing vertex $W_0$. Then the
  reformulated stability criterion is that
  \begin{equation} \label{eq:king-stability-criterion}
    ({\vec\theta}, \theta_\infty) \cdot (\dim V', \dim W_0') > 0
  \end{equation}
  for non-trivial sub-representations $(V', W_0')$ (see the proof of
  \cite[Lemma 7.2.10]{Okounkov2017}). For us, using
  \eqref{eq:regular-stability-chamber}, this extended stability
  condition is
  \[ ({\vec\theta}, \theta_\infty) = (-m+1+\epsilon, 1, 1, \ldots, 1, -\epsilon), \]
  and $\dim V'_0 \in \{0, 1\}$. It follows that $V_0$ must generate
  all other $V_k$. So the only $\sT'$-fixed quiver representations are
  those in Figure~\ref{fig:quiver-fixed-points}, with characters
  \begin{equation} \label{eq:m-hook}
    \hk_a \coloneqq y^{m-a} + y^{m-a-1} + \cdots + y + 1 + x + \cdots + x^a
  \end{equation}
  for $0 \le a < m$. Each of these characters is the restriction of
  the universal family $\rV$ to a fixed point $p_a$; the line bundle
  $\rV_k$ is the summand with $\bZ/m$-weight $k$.
\end{proof}

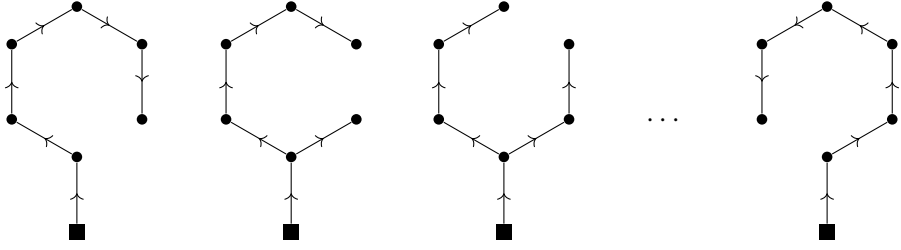
\begin{figure}[h]
  \centering
  \begin{tabular}{ccccc}
    \begin{tikzpicture}[decoration={markings,mark=at position 0.5 with {\arrow{>}}},rotate=-90]
      \node[vertex,shape=rectangle,minimum size=6pt] (w) at (2,0) {};
      \node[vertex] (v0) at (0:1) {};
      \node[vertex] (v1) at (60:1) {};
      \node[vertex] (v2) at (120:1) {};
      \node[vertex] (v3) at (180:1) {};
      \node[vertex] (v4) at (240:1) {};
      \node[vertex] (v5) at (300:1) {};
      \draw[postaction={decorate}] (w) -- (v0);
      \draw[postaction={decorate}] (v0) -- (v5);
      \draw[postaction={decorate}] (v5) -- (v4);
      \draw[postaction={decorate}] (v4) -- (v3);
      \draw[postaction={decorate}] (v3) -- (v2);
      \draw[postaction={decorate}] (v2) -- (v1);
    \end{tikzpicture} \hspace{1em}
    &
    \begin{tikzpicture}[decoration={markings,mark=at position 0.5 with {\arrow{>}}},rotate=-90]
      \node[vertex,shape=rectangle,minimum size=6pt] (w) at (2,0) {};
      \node[vertex] (v0) at (0:1) {};
      \node[vertex] (v1) at (60:1) {};
      \node[vertex] (v2) at (120:1) {};
      \node[vertex] (v3) at (180:1) {};
      \node[vertex] (v4) at (240:1) {};
      \node[vertex] (v5) at (300:1) {};
      \draw[postaction={decorate}] (w) -- (v0);
      \draw[postaction={decorate}] (v0) -- (v1);
      \draw[postaction={decorate}] (v0) -- (v5);
      \draw[postaction={decorate}] (v5) -- (v4);
      \draw[postaction={decorate}] (v4) -- (v3);
      \draw[postaction={decorate}] (v3) -- (v2);
    \end{tikzpicture} \hspace{1em}
    &
    \begin{tikzpicture}[decoration={markings,mark=at position 0.5 with {\arrow{>}}},rotate=-90]
      \node[vertex,shape=rectangle,minimum size=6pt] (w) at (2,0) {};
      \node[vertex] (v0) at (0:1) {};
      \node[vertex] (v1) at (60:1) {};
      \node[vertex] (v2) at (120:1) {};
      \node[vertex] (v3) at (180:1) {};
      \node[vertex] (v4) at (240:1) {};
      \node[vertex] (v5) at (300:1) {};
      \draw[postaction={decorate}] (w) -- (v0);
      \draw[postaction={decorate}] (v0) -- (v1);
      \draw[postaction={decorate}] (v1) -- (v2);
      \draw[postaction={decorate}] (v0) -- (v5);
      \draw[postaction={decorate}] (v5) -- (v4);
      \draw[postaction={decorate}] (v4) -- (v3);
    \end{tikzpicture} \hspace{1em} 
    &
    $\cdots$ \hspace{1em}
    &
    \begin{tikzpicture}[decoration={markings,mark=at position 0.5 with {\arrow{>}}},rotate=-90]
      \node[vertex,shape=rectangle,minimum size=6pt] (w) at (2,0) {};
      \node[vertex] (v0) at (0:1) {};
      \node[vertex] (v1) at (60:1) {};
      \node[vertex] (v2) at (120:1) {};
      \node[vertex] (v3) at (180:1) {};
      \node[vertex] (v4) at (240:1) {};
      \node[vertex] (v5) at (300:1) {};
      \draw[postaction={decorate}] (w) -- (v0);
      \draw[postaction={decorate}] (v0) -- (v1);
      \draw[postaction={decorate}] (v1) -- (v2);
      \draw[postaction={decorate}] (v2) -- (v3);
      \draw[postaction={decorate}] (v3) -- (v4);
      \draw[postaction={decorate}] (v4) -- (v5);
    \end{tikzpicture}
  \end{tabular}
  \caption{$\sT'$-fixed quiver representations for $\cA_5$ (only non-zero maps shown)}
  \label{fig:quiver-fixed-points}
\end{figure}

\subsubsection{}
\label{sec:quiver-data-as-boxes}

Plugging \eqref{eq:structure-sheaf-regular-Hilb} and
\eqref{eq:universal-family-Am} into \eqref{eq:regular-quiver-data}
yields the following description for $V = \bigoplus_k V_k$. Given a
fixed point $\vec\lambda = (\lambda_0, \ldots, \lambda_{m-1}) \in
\Hilb(\cA_{m-1}, n)$,
\begin{equation} \label{eq:orbifold-box-diagram-fixed-points}
  V(\vec\lambda) = \sum_{a=0}^{m-1} V^{(a)}(\vec\lambda), \quad V^{(a)}(\vec\lambda) \coloneqq \hk_a \sum_{\square \in \lambda_a} x_a^{i(\square)} y_a^{j(\square)}.
\end{equation}
In terms of diagrams, this means to take each $\lambda_a$ and draw it
using $(x_a, y_a)$ coordinates, but with each box replaced with the
$m$-hook $\hk_a$ from \eqref{eq:m-hook}. Then $V$ is the sum of all
the resulting diagrams.
Figure~\ref{fig:quiver-data-fixed-point-example} illustrates an
example. Note that $V$ may have multiple boxes with the {\it same}
$\sT'$-weight, which is not a phenomenon that occurs for
$\Hilb([\bC^2/\Gamma])$. We indicate this on diagrams by labeling
boxes with their multiplicities whenever the multiplicity is greater
than $1$.
\begin{figure}[h]
  \centering
  \begin{tikzpicture}[scale=0.5]
    \begin{scope}
      \foreach \x/\y in {0/0, 0/1, -2/1, -1/1, 1/0, 2/0} {
        \draw (\y,\x) rectangle ++(1,1);
      }
      \draw[->,thick] (0,0) -- (4,0) node[label=right:$x$]{};
      \draw[->,thick] (0,0) -- (0,4) node[label=above:$y$]{};
    \end{scope}
    \node at (5.5,2) {$+$};
    \begin{scope}[xshift=7cm]
      \foreach \x/\y in {0/0, 0/1, 0/2} {
        \draw (\y,\x) rectangle ++(1,1);
      }
      \draw[->,thick] (0,0) -- (4,0) node[label=right:$x$]{};
      \draw[->,thick] (0,0) -- (0,4) node[label=above:$y$]{};
    \end{scope}
    \node at (13,2) {$=$};
    \begin{scope}[xshift=15cm]
      \foreach \x/\y/\c in {0/0/2, 0/1/2, 0/2/1, -2/1/1, -1/1/1, 1/0/1, 2/0/1} {
        \draw (\y,\x) rectangle ++(1,1);
        \ifnum \c > 1 {
          \node at (0.5+\y, 0.5+\x) {\c};
        } \else {}
        \fi;
      }
      \draw[->,thick] (0,0) -- (4,0) node[label=right:$x$]{};
      \draw[->,thick] (0,0) -- (0,4) node[label=above:$y$]{};
    \end{scope}
  \end{tikzpicture}
  \caption{Quiver data for ({\tiny\protect\yng(2)}, $\emptyset$, {\tiny\protect\yng(1)}) $\in \Hilb(\cA_2)$}
  \label{fig:quiver-data-fixed-point-example}
\end{figure}
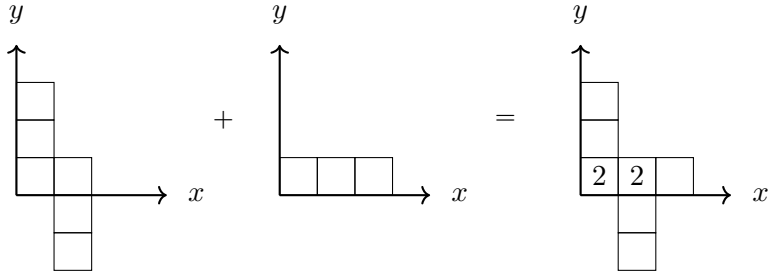

As a quiver representation, $V(\vec\lambda)$ is clearly also a sum of
the quiver representations $V^{(a)}(\vec\lambda)$. Each of these
quiver representations has a non-zero multiplication by $x$ (resp.
$y$) map whenever a box has a neighboring box to its right (resp.
above it). In other words, each box in $V^{(a)}(\vec\lambda)$
generates all boxes above and to the right of it.

\begin{example}
  The quiver representation corresponding to
  Figure~\ref{fig:quiver-data-fixed-point-example} is
  \[ \begin{tikzcd}[ampersand replacement=\&]
    \& \bC \ar{d}{{\begin{psmallmatrix} 1 \\ 1 \\ 0 \end{psmallmatrix}}} \\
    \& \bC \oplus \bC \oplus \bC xy^{-2} \ar[bend left=5]{ddr}{{\begin{psmallmatrix} & & y \\ & 0 &  \\ y & & \end{psmallmatrix}}} \ar[swap,bend right=5]{ddl}{{\begin{psmallmatrix} x \\ & x \\ & & 0 \end{psmallmatrix}}} \\ \\
    \bC x \oplus \bC x \oplus \bC y^2 \ar[swap,bend right=5]{uur}{0} \ar[swap,bend right=3]{rr}{{\begin{psmallmatrix} 0 \\ & x \\ & & 0 \end{psmallmatrix}}} \& \& \bC xy^{-1} \oplus \bC x^2 \oplus \bC y \ar[bend left=5]{uul}{0} \ar[swap,bend right=3]{ll}{{\begin{psmallmatrix} y \\ & 0 \\ & & y \end{psmallmatrix}}}.
  \end{tikzcd} \]
\end{example}

\subsubsection{}

In fact $\sT'$-fixed points for $\Hilb([\bC^2/\Gamma])$ and
$\Hilb(\cA_{m-1})$ have a strong relationship.
Theorem~\ref{thm:hyperkahler-rotation-diffeomorphism} implies there is
a bijection between $\sT'$-fixed points of these two spaces which is
$\sA'$-equivariant, i.e. the $\sA'$-characters of the corresponding
quiver data must be equal. This bijection was described
non-equivariantly in \cite[Theorem 5.3]{Kuznetsov2007} and
$\sA'$-equivariantly in \cite[Theorem 4.5]{Nagao2009}, but now we can
also see how the fully equivariant $\sT'$-characters of fixed points
compare.

Recall that an $m$-colored partition $\lambda$ can also be described
by its {\it $m$-core} $c(\lambda)$ and {\it $m$-quotient}
$(q_0(\lambda), \ldots, q_{m-1}(\lambda))$ (see \cite[Section
  2.7]{James1981} for an introduction). Uniformly $m$-colored
partitions are precisely those with trivial $m$-core, so let
\begin{align*}
  \vec q\colon \left\{\text{uniformly } m\text{-colored partitions}\right\} &\xrightarrow{\sim} \left\{\text{all partitions}\right\}^{\times m} \\
  \lambda &\mapsto (q_0(\lambda), \ldots, q_{m-1}(\lambda))
\end{align*}
be the bijection.

\begin{lemma}[{\cite[Proposition 2.5.3]{Nagao2009}}] \label{lem:orbifold-fixed-point-bijection}
  The bijection $\vec q$ is exactly the bijection
  \[ \Hilb([\bC^2/\Gamma])^{\sT'} \cong \Hilb(\cA_{m-1})^{\sT'} \]
  under our identification of fixed points with Young diagrams, and
  \begin{equation} \label{eq:A-equivariant-fixed-point-bijection}
    V_k(\lambda) \equiv V_k(\vec q(\lambda)) \bmod{\hbar}
  \end{equation}
\end{lemma}

Importantly, \eqref{eq:A-equivariant-fixed-point-bijection} is only
true mod $\hbar$, i.e. only in $K_{\sA'}(\pt)$, and is {\it not} true
in $K_{\sT'}(\pt)$. This means $V(\lambda)$ and $V(\vec q(\lambda))$
may differ (and can only differ) by shifting boxes along the diagonals
of same color. More precisely, to get the Young diagram represented by
$V(\lambda)$ from $V(\vec q(\lambda))$, we shift boxes along their
diagonals until we get a valid Young diagram, as in
Figure~\ref{fig:T-equivariant-fixed-point-correspondence}.
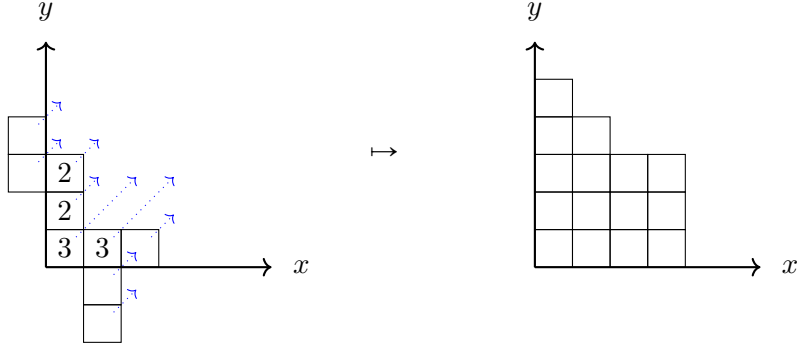
\begin{figure}[h]
  \centering
  \begin{tikzpicture}[scale=0.5]
    \begin{scope}
      \foreach \x/\y/\c in {0/0/3, 0/1/3, 0/2/1, -1/1/1, -2/1/1, 1/0/2, 2/0/2, 2/-1/1, 3/-1/1} {
        \draw (\y,\x) rectangle ++(1,1);
        \ifnum \c > 1 {
          \node at (0.5+\y, 0.5+\x) {\c};
          \draw[->,blue,dotted] (0.8+\y,0.8+\x) -- ++(\c-1.4,\c-1.4);
        } \else {
          \draw[->,blue,dotted] (0.8+\y,0.8+\x) -- ++(\c-0.4,\c-0.4);
        } \fi;
      }
      \draw[->,thick] (0,0) -- (6,0) node[label=right:$x$]{};
      \draw[->,thick] (0,0) -- (0,6) node[label=above:$y$]{};
    \end{scope}
    \node at (9,3) {$\mapsto$};
    \begin{scope}[xshift=13cm]
      \foreach \x/\n in {0/4, 1/4, 2/4, 3/2, 4/1} {
        \foreach \y in {1,...,\n} {
          \draw (\y,\x) rectangle ++(-1,1);
        }
      }
      \draw[->,thick] (0,0) -- (6,0) node[label=right:$x$]{};
      \draw[->,thick] (0,0) -- (0,6) node[label=above:$y$]{};
    \end{scope}
  \end{tikzpicture}
  \caption{$\sT'$-equivariant correspondence for ({\tiny\protect\yng(2)}, {\tiny\protect\yng(1,1)}, {\tiny\protect\yng(1)}) $\in \Hilb(\cA_2)$}
  \label{fig:T-equivariant-fixed-point-correspondence}
\end{figure}

As for obtaining $V(\vec q(\lambda))$ from $V(\lambda)$, the easiest
way is somewhat indirect: compute $\vec q(\lambda)$ directly from
$V(\lambda)$ and then apply
\eqref{eq:orbifold-box-diagram-fixed-points}. It is not clear if there
is a nice combinatorial description of which boxes in $V(\lambda)$
need to be shifted (and by how much) to obtain $V(\vec q(\lambda))$
directly.

\subsubsection{}

There is a universal formula for the tangent space of both
$\Hilb([\bC^2/\Gamma])$ and $\Hilb(\cA_{m-1})$ at fixed points. This
is because they are both open subsets of the stack
$[\mu^{-1}(0)/G_V]$, whose tangent bundle in K-theory can be written
as follows. If $V$ is the quiver data of a fixed point, then
\begin{equation} \label{eq:tangent-space-quiver-variety}
  T_V\Hilb = \overbrace{T^*\left(V_0 + \frac{1}{x}\sum_{i=0}^{m-1} \Hom(V_i, V_{i+1})\right)}^{\text{deformations of quiver rep}} - \overbrace{\left(1 + \hbar\right) \sum_{i=0}^{m-1} \Hom(V_i, V_i)}^{\text{moment map and quotient}}.
\end{equation}
Note that $T^*V_i = V_i + \hbar V_i^\vee$. The $x$ factor in
front of $\Hom(V_i, V_{i+1})$ is due to the multiplication-by-$x$
maps.

\subsection{The quasimap vertex}
\label{sec:quasimap-theory}

\subsubsection{}

To consider maps from a curve $C$ to a Nakajima quiver variety, we
upgrade from quiver representations into vector spaces to quiver
representations into vector bundles on $C$. Such objects are called
{\it quiver bundles} \cite{Gothen2005}.

\begin{definition}
  For two vector bundles $\cV = \bigoplus \cV_i$ and $\cW = \bigoplus
  \cW_i$ on $C$, \eqref{eq:quiver-rep-vector-spaces} can be upgraded
  to
  \[ \cRep_Q^{\text{framed}} \coloneqq \bigoplus_{\text{edge } i \to j} \cHom(\cV_i, \cV_j) \oplus \bigoplus_{i \in I} \cHom(\cW_i, \cV_i). \]
  A section $s$ of $T^*\cRep_Q^{\text{framed}}$ satisfying $\mu(s) =
  0$ is a {\it quasimap} in the sense of \cite{Ciocan-Fontanine2014};
  it is equivalent to a map
  \[ f\colon C \to [\mu^{-1}(0)/G_V]. \]
  Since we want maps to the open locus $\cM_{\vec\theta} \subset
  [\mu^{-1}(0)/G_V]$, we say a quasimap is {\it singular} at $p \in C$
  if $f(p) \notin \cM_{\vec\theta}$. If $f$ is singular only at
  finitely many points on $C$, it is a {\it stable} quasimap.
\end{definition}

\subsubsection{}
\label{sec:quasimap-vertex}

Let $\QMaps(\cM_{\vec\theta})$ denote the moduli of stable quasimaps to the
Nakajima quiver variety $\cM_{\vec\theta}$ for $C = \bP^1$. Consider the
open locus
\[ \QMaps(\cM_{\vec\theta})_{\text{nonsing } \infty} \subset \QMaps(\cM_{\vec\theta}) \]
consisting of all quasimaps which are non-singular at $\infty \in
\bP^1$. On this locus there is a well-defined evaluation map
\[ \ev_\infty\colon \QMaps(\cM_{\vec\theta})_{\text{nonsing }\infty} \to \cM_{\vec\theta} \]
sending $f$ to $\ev_\infty(f) \coloneqq f(\infty)$.

\begin{definition}[{\cite[Section 7.2]{Okounkov2017}}]
  The {\it quasimap vertex} is the series
  \[ \cat{V}_{\cat{QMaps}}(\vec \fz) \coloneqq \sum_{\vec d} \vec \fz^{\vec d} \ev_{\infty,*}\left(\QMaps^{\vec d}, \hat\cO^{\vir}\right) \in K_{\sT}(\cM_{\vec\theta})_{\text{localized}}[[\vec \fz^{\pm 1}]]. \]
  The variables $\vec\fz \coloneqq (\fz_0, \ldots, \fz_{m-1})$ record
  the degree $\deg f \coloneqq (\deg \cV_i) \in \bZ^m$ of the
  quasimap, which indexes connected components $\QMaps^{\vec d}
  \subset \QMaps(\cM_{\vec\theta})_{\text{nonsing }\infty}$.
\end{definition}

\begin{remark}
  Quasimaps to $\Hilb([\bC^2/\Gamma])$ are studied modulo
  $\sA'$-equivariance in \cite{Dinkins2019}, where a formula for the
  quasimap vertex is given for the general $A_\infty$ quiver (which is
  a limiting case of affine $A_n$ quivers). The extension to full
  $\sT$-equivariance is straightforward. However, note that while the
  stability chamber \eqref{eq:orbifold-stability-chamber} defining
  $\Hilb([\bC^2/\Gamma])$ is valid on the $A_\infty$ quiver, the
  stability chamber \eqref{eq:regular-stability-chamber} defining
  $\Hilb(\cA_{m-1})$ is not. It truly depends on $m$ and therefore
  cannot be investigated using the $A_\infty$ quiver.
\end{remark}

\subsubsection{}

The domain $C = \bP^1$ for quasimaps has an action by the
$\bC^\times_z$ in $\sT$ in \eqref{eq:full-torus-weights}, and
corresponds to the $C$ in $Y = S \times C$. For $p \in
\cM_{\vec\theta}$, let
\[ \QMaps_p \coloneqq \ev_\infty^{-1}(p). \]
It has components $\QMaps_p^{\vec d}$ consisting of quasimaps of
degree $\vec d$. It is known \cite[Corollary 7.2.5]{Okounkov2017} that
the summands of
\[ \QMaps(\cM_{\vec\theta})^{\sT}_{\text{nonsing }\infty} = \bigoplus_{p \in \cM_{\vec\theta}^{\sT}} \bigoplus_{\vec d} \left(\QMaps_p^{\vec d}\right)^{\bC^\times_z} \]
are proper, and therefore the $p$-th component
\begin{equation} \label{eq:quasimap-vertex-definition}
  \sV^p_{\QMaps}(\vec\fz) = \sum_{\vec d} \chi\left(\QMaps_p^{\vec d}, \hat\cO^{\vir}\right) \vec\fz^{\vec d}
\end{equation}
of the quasimap vertex is well-defined via localization.

\subsubsection{}
\label{sec:quasimaps-fixed-loci}

As with fixed points on the Nakajima quiver variety
$\cM_{\vec\theta}$, we conflate $\sT$-fixed quasimaps with their
quiver data. Fix a point $V \in \cM_{\vec\theta}^{\sT'}$, and let $f
\in \QMaps_V^{\vec d}$ be a $\sT$-fixed quasimap defined by its quiver
data $\cV$.

\begin{lemma}[{\cite[Section 7.1]{Okounkov2017}}]
  The map $f$ is constant, and $\cV$ consists of line bundles
  $\cO_{\bP^1}(d)$ linearized with weights $z^{-d}, 1$ at $0, \infty
  \in \bP^1$ respectively.
\end{lemma}

\begin{definition}
  For a square $\vsquare \in V$, let $d_{\vsquare}$ be the degree of
  its corresponding line bundle
  \[ \cO_{\bP^1}(d_{\vsquare}) \subset \cV. \]
  If $\vsquare$ has multiplicity, e.g. for $\Hilb(\cA_{m-1})$, then we
  write $d_{\vsquare}^{(a)}$ for the degree coming from $\vsquare \in
  V^{(a)}(\vec \lambda)$.
\end{definition}

To represent $\cV$ diagrammatically, we write (lists of) degrees in
each box of the diagram for $V$, as in
Figure~\ref{fig:quasimaps-nontrivial-fixed-loci}. When there are
multiple degrees in a box, we disregard the ordering of the list of
degrees.

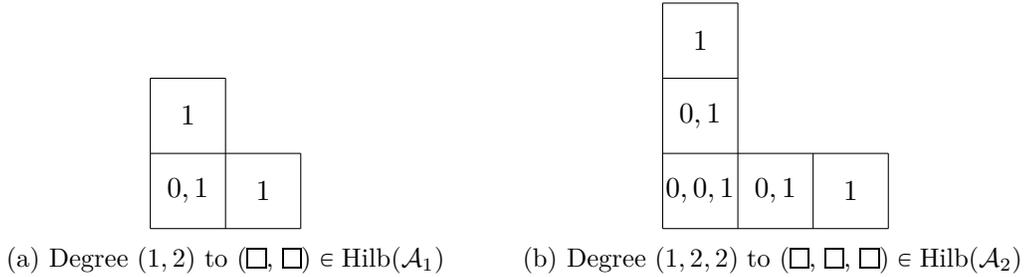
\begin{figure}[h]
  \centering
  \begin{subfigure}[b]{0.45\textwidth}
    \centering
    \begin{tikzpicture}
      \draw (0,0) -- (0,2);
      \draw (1,0) -- (1,2);
      \draw (2,0) -- (2,1);
      \draw (0,0) -- (2,0);
      \draw (0,1) -- (2,1);
      \draw (0,2) -- (1,2);
      \node at (0.5,0.5) {$0,1$};
      \node at (1.5,0.5) {$1$};
      \node at (0.5,1.5) {$1$};
    \end{tikzpicture}
    \caption{Degree $(1,2)$ to ({\tiny\protect\yng(1)}, {\tiny\protect\yng(1)}) $\in \Hilb(\cA_1)$}
    \label{fig:quasimaps-nontrivial-fixed-loci-1}
  \end{subfigure}%
  ~
  \begin{subfigure}[b]{0.45\textwidth}
    \centering
    \begin{tikzpicture}
      \draw (0,0) -- (0,3);
      \draw (1,0) -- (1,3);
      \draw (2,0) -- (2,1);
      \draw (3,0) -- (3,1);
      \draw (0,0) -- (3,0);
      \draw (0,1) -- (3,1);
      \draw (0,2) -- (1,2);
      \draw (0,3) -- (1,3);
      \node at (0.5,0.5) {$0,0,1$};
      \node at (1.5,0.5) {$0,1$};
      \node at (2.5,0.5) {$1$};
      \node at (0.5,1.5) {$0,1$};
      \node at (0.5,2.5) {$1$};
    \end{tikzpicture}
    \caption{Degree $(1,2,2)$ to ({\tiny\protect\yng(1)}, {\tiny\protect\yng(1)}, {\tiny\protect\yng(1)}) $\in \Hilb(\cA_2)$}
    \label{fig:quasimaps-nontrivial-fixed-loci-2}
  \end{subfigure}
  \caption{Some $\sT$-fixed components of quasimap moduli space}
  \label{fig:quasimaps-nontrivial-fixed-loci}
\end{figure}

\subsubsection{}

There is a universal formula for the {\it virtual} tangent space of
$\QMaps$ at $\cV$. It is completely analogous to
\eqref{eq:tangent-space-quiver-variety}, but requires an extra
pushforward over the domain $\bP^1$:
\begin{equation} \label{eq:quasimaps-Tvir}
  T^{\vir}_{\cV}\QMaps = H^\bullet\left(T^*\left(\cV_0 + \frac{1}{x}\sum_{i=0}^{m-1} \Hom(\cV_i, \cV_{i+1})\right) - \left(1 + \hbar\right) \sum_{i=0}^{m-1} \Hom(\cV_i, \cV_i)\right).
\end{equation}

Usually for the quasimap vertex, one normalizes by subtracting the
contribution of $T_V\cM_{\vec\theta}$ from $T^{\vir}_{\cV}$. This is
equivalent to replacing $H^\bullet$ in \eqref{eq:quasimaps-Tvir} with
\begin{equation} \label{eq:quasimap-Tvir-normalization}
  H^\bullet_\sim(\cO(n)) \coloneqq H^\bullet(\cO(n)) - 1
\end{equation}
for non-equivariant line bundles $\cO(n)$ on $\bP^1$. The contribution
$T_V\cM_{\vec\theta}$ is the freedom of the point $V$ to move around
in $\cM_\theta$, and removing it is analogous to the redistribution
for the BS vertex discussed in \ref{sec:bs-vertex-Tvir}. As we do not
perform the redistribution for the BS vertex, we also do not normalize
the quasimap vertex.

\subsubsection{}

From here on, we focus on the case of quasimaps to $\Hilb(\cA_{m-1})$.
While the moduli of quasimaps to $\Hilb([\bC^2/\Gamma])$ has isolated
$\sT$-fixed points, the moduli of quasimaps to $\Hilb(\cA_{m-1})$ can
have arbitrarily high-dimensional $\sT$-fixed loci. Using the formula
\eqref{eq:bs-Tvir} for $T^{\vir}$, one can read off a combinatorial
formula for the virtual dimension of $\sT$-fixed loci.

\begin{lemma} \label{lem:quasimap-Tvir-fixed-term}
  The $\sT$-fixed term in $T^{\vir}_{\cV}$ is:
  \begin{enumerate}
  \item the total number of pairs $d_\vsquare^{(a)} >
    d_{\vsquare'}^{(b)}$ where $\vsquare'$ is either $\vsquare$ or $xy
    \cdot \vsquare$, minus
  \item the total number of pairs $d_\vsquare^{(a)} >
    d_{\vsquare'}^{(b)}$ where $\vsquare'$ is either $x \cdot
    \vsquare$ or $y \cdot \vsquare$, minus
  \item the total number of negative degrees at $\vsquare = (0,0)$.
  \end{enumerate}
\end{lemma}

\begin{proof}
  The $\sT$-fixed term is unchanged if we normalize
  $T^{\vir}_{\cV}\QMaps$ by using $H^\bullet_\sim$ as defined in
  \eqref{eq:quasimap-Tvir-normalization} instead of the usual
  $H^\bullet$, because the discrepancy $T_V\cM_{\vec\theta}$ has no
  $\sT$-fixed weight. Using the formula \eqref{eq:quasimaps-Tvir} for
  $T^{\vir}_{\cV}$, the $\sT$-fixed term is the total number of line
  bundles $\cO_{\bP^1}(n)$ with $n < 0$ in the expression
  \[ \cV_0 + \frac{\cV_0^\vee}{xy} + \frac{1}{x} \sum_a \cV_a^\vee \cV_{a+1} + \frac{1}{y} \sum_a \cV_a \cV_{a+1}^\vee - \left(1 + \frac{1}{xy}\right) \sum_a \cV_a^\vee \cV_a, \]
  counted with {\it opposite} sign. Namely, a line bundle $\pm
  \cO_{\bP^1}(n)$ contributes $\mp 1$ to the total.
\end{proof}

For example, the fixed component of
Figure~\ref{fig:quasimaps-nontrivial-fixed-loci-2} has virtual
dimension $4 - 2 = 2$ in $\QMaps(\Hilb(\cA_2))$. Conjecturally, the
fixed loci have sufficiently nice geometry for their virtual
dimensions to be their actual dimensions.

\begin{remark}
  Since $\pi$-stable pairs have no non-trivial automorphisms
  \cite[Lemma 23]{Bryan2016}, when the virtual dimension is negative
  there must be non-trivial obstructions and therefore the degrees do
  not define a valid quasimap fixed component. Hence
  Lemma~\ref{lem:quasimap-Tvir-fixed-term} yields a combinatorial
  criterion for how to label squares in $V(\vec\lambda)$ to produce a
  valid quasimap.
\end{remark}

\subsubsection{}

To understand the geometry of $\sT$-fixed loci in
$\QMaps(\Hilb(\cA_{m-1}))$, a different perspective on $\sT$-fixed
quasimaps is necessary.

\begin{definition}
  Let $\cV \in \QMaps(\Hilb(\cA_{m-1}))^{\sT}$. Associated to $\cV =
  \bigoplus_k \cV_k$ are the vector spaces
  \[ \bV_k \coloneqq \bigoplus_{n \in \bZ} \bV_k[n] \coloneqq H^0\left(\cV_k\big|_{\bP^1 \setminus \{\infty\}}\right), \]
  where $\bV_k[n]$ is the $\bC_z^\times$ weight space of weight $n$.
  Multiplication by $z$ induces embeddings
  \begin{equation} \label{eq:quasimaps-as-flags}
    \bV_k[n] \hookrightarrow \bV_k[n+1] \hookrightarrow \cdots \hookrightarrow \bV_k[\infty] = V_k
\end{equation}
  compatible with quiver maps. Let $\bV^\bullet = \bigoplus_k
  \bV_k^\bullet$ denote the resulting flag of quiver representations.
\end{definition}

A $\sT$-fixed quasimap to $V$ is therefore the data of an infinite
flag $\bV^\bullet$ of (stable) quiver sub-representations of $V$,
starting with the zero quiver representation at $n = -\infty$ and
ending with $V$ at $n = \infty$, along with the data of the framing
morphism $\cW \hookrightarrow \cV$. Note that since $\cW =
\cO_{\bP^1}$, its associated flag is
\[ \bW_k[n] = \begin{cases} W_k & n \ge 0 \\ 0 & n < 0 \end{cases} \]
and there is a map $\bW^\bullet \hookrightarrow \bV^\bullet$.
Rephrasing, the only obstruction to a flag $\bV^\bullet$ of quiver
sub-representations of $V$ being a valid quasimap is that it admits a
framing $\bW^\bullet \hookrightarrow \bV^\bullet$.

\begin{proposition} \label{prop:quasimap-as-framed-quiver-bundle}
  A quiver bundle $\cV$ defined by a flag $\bV^\bullet$ of quiver
  representations of $V$ is a stable quasimap iff it admits a framing
  $\bW^\bullet \hookrightarrow \bV^\bullet$.
\end{proposition}

\subsubsection{}

Using $\bV^\bullet$, we now explain the ambiguity in defining a
single quasimap from diagrams like in
Figure~\ref{fig:quasimaps-nontrivial-fixed-loci}. The problem is that,
even after specifying the dimension vector for a quiver
sub-representation, there is still non-trivial moduli for the quiver
sub-representation whenever a square has multiplicity. This occurs in
{\it negative} $z$-degree, where, crucially, the sub-representations
are {\it not} framed. Note that this will only occur for
$\QMaps(\Hilb(\cA_{m-1}))$; the $\sT$-fixed loci in
$\QMaps(\Hilb([\bC^2/\Gamma]))$ are isolated points.

\begin{example} \label{ex:quasimaps-nontrivial-fixed-loci}
  Consider the $\sT$-fixed component in
  Figure~\ref{fig:quasimaps-nontrivial-fixed-loci-1}. Disregarding
  $z$-weights, it corresponds to the flag
  \[ \begin{tikzcd}[ampersand replacement=\&]
    \bC x \oplus \bC y \ar[shift right=2]{d} \ar[shift left=1]{d} \ar[hookrightarrow]{r} \& \bC x \oplus \bC y \ar[shift right=2]{d} \ar[shift left=1]{d} \ar[hookrightarrow]{r} \& \cdots \\
    \bC \ar[shift right=2]{u} \ar[shift left=1]{u} \ar[hookrightarrow]{r}{\iota = \begin{psmallmatrix} a \\ b \end{psmallmatrix}} \& \bC^2 \ar[shift right=2]{u} \ar[shift left=1]{u} \ar[hookrightarrow]{r} \& \cdots \\
    {} \& \bC \ar[hookrightarrow]{r} \ar[swap]{u}{\begin{psmallmatrix} 1 \\ 1 \end{psmallmatrix}} \& \cdots
  \end{tikzcd}. \]
  There is a $\bP^1$ worth of freedom for the inclusion $\iota\colon
  \bC \to \bC^2$, and then everything else is uniquely specified.
  Hence this $\sT$-fixed component is a $\bP^1$. 
\end{example}

This example is the quasimap analogue of
Example~\ref{ex:bs-nontrivial-fixed-loci}. Note that if the $\bV[-1]$
slice were also framed, the framing would force $\iota
= \begin{psmallmatrix} 1 \\ 1 \end{psmallmatrix}$, removing the degree
of freedom and turning the fixed $\bP^1$ into a fixed point.

\section{BS/quasimaps correspondence}
\label{sec:bs-quasimaps}

\subsection{The correspondence}

\subsubsection{}

The main result of this paper is that, after a suitable change of
variables, the series in the BS $1$-leg vertex and the quasimap vertex
for the $\cA_{m-1}$ geometry are equal on the nose.

\begin{theorem}[BS/quasimaps correspondence] \label{thm:bs-quasimap-correspondence}
  For $S = \cA_{m-1}$,
  \[ \sV_{\PiPairs(S \times \bP^1)}(Q, \vec A) = \sV_{\QMaps(\Hilb(S))}(\vec\fz) \]
  after the change of variables
  \begin{align*}
    Q &\leftrightarrow \fz_0\fz_1 \cdots \fz_{m-1} \\
    A_i &\leftrightarrow \fz_i \qquad \forall 1 \le i < m.
  \end{align*}
\end{theorem}

As with many results about $\cA_{m-1}$, this theorem continues to hold
when $m = 1$. In this case, the geometry is $\cA_0 \coloneqq \bC^2$,
viewed as a trivial resolution of $\bC^2$ with no singularities. Then
BS theory on $\cA_0$ is equivalent to PT theory, in which case this
theorem shows that
\[ \sV_{\Pairs(\bC^2)}(Q) = \sV_{\QMaps(\Hilb(\bC^2))}(\fz) \]
where $Q = \fz$. In fact, in this case, this arises from a known
isomorphism $\Pairs(\bC^2) \simeq \QMaps(\Hilb(\bC^2))$ of moduli
spaces \cite[Exercise 4.3.2]{Okounkov2017}.

\subsubsection{}
\label{sec:bs-quasimaps-bijection}

The proof of the correspondence goes via $\sT$-equivariant
localization: it suffices to give an isomorphism of $\sT$-fixed loci
for $\PiPairs$ and $\QMaps$ and show that this isomorphism respects
virtual tangent spaces. We first state this isomorphism as a bijection
between {\it components} of $\sT$-fixed loci, and then clarify why it
is an isomorphism. Fix $\vec\lambda \in \Hilb(\cA_{m-1})^{\sT}$.
\begin{itemize}
\item Recall from section~\ref{sec:bs-local-model-filtration} that
  $\sT$-fixed (components of) $\pi$-stable pairs to $\vec\lambda$ are
  specified by local models for each $\square \in \lambda_a$. In
  particular, the local model for $\square \in \lambda_a$ is defined
  by numbers
  \[ \vec e^\square \coloneqq (e_0^\square, \ldots, e_{m-1}^\square). \]
  Here $e_0^\square$ is the smallest integer $k \in \bZ$ such that the
  box $z^k \cdot \square$ is not involved in a rod, and otherwise
  $e_a^\square$ is the integer such that $e_a^\square - e_0^\square$
  is the number of standard rods in the local model with non-zero
  support over the exceptional component $E_a$.

\item Recall from section~\ref{sec:quasimaps-fixed-loci} if
  $V(\vec\lambda) = \sum_a V^{(a)}(\vec\lambda)$ is the quiver data
  describing $\vec\lambda$, then $\sT$-fixed (components of) quasimaps
  to $\vec\lambda$ are specified by labeling the squares in each
  $V^{(a)}(\vec\lambda)$ with degrees. More specifically, each
  $\square \in \lambda_a$ has an associated sequence of degrees
  \[ \vec d^\square \coloneqq (d_0^{\square}, \ldots, d_{m-1}^{\square}), \]
  namely where $d_a^\square$ is the label of the color-$a$ square in
  the hook with character $\hk_a x_a^{i(\square)} y_a^{j(\square)}
  \subset V^{(a)}(\vec\lambda)$.
\end{itemize}
The desired bijection is simply that
\begin{equation} \label{eq:bs-quasimaps-bijection}
  \vec e^\square = \vec d^\square \quad \forall \, \square \in \lambda_a, \; \forall 0 \le a < m.
\end{equation}
That this is indeed a bijection of fixed components is the content of
Proposition~\ref{prop:bs-quasimap-stability-correspondence}, obtained
by manually matching stability conditions in BS and quasimap theories.
One can easily check that this bijection yields the specified change
of variables $(Q, \vec A) \leftrightarrow \vec\fz$.

\subsubsection{}

The outline of the proof is as follows. In
section~\ref{sec:bs-quasimaps-Tvir-correspondence}, the bijection
$\vec d^{\square} \mapsto \vec e^{\square}$ is realized geometrically
as an equivalence of categories
\[ \tilde\Phi_{\sT}\colon D^b\cat{Coh}_{\sT}(\bC^2 \times \bP^1) \to D^b\cat{Coh}_{\sT/\Gamma}(\cA_{m-1} \times \bP^1), \]
constructed using a $\sT$-equivariant version of the derived McKay
equivalence. Consequently, the bijection of fixed components is
refined to an {\it isomorphism} of fixed components, discussed in
section~\ref{sec:bs-quasimap-correspondence-refinement}. Then, to
match vertices, it suffices to match virtual tangent spaces of fixed
components. Although not all $\{\vec d^\square\}_{\square \in
  \vec\lambda}$ form stable quasimaps, and not all $\{\vec
e^\square\}_{\square \in \vec\lambda}$ form valid $\cO_Y$-modules for
$\pi$-stable pairs, in fact
Proposition~\ref{prop:bs-quasimaps-Tvir-correspondence} shows that the
formulas \eqref{eq:bs-Tvir} and \eqref{eq:quasimaps-Tvir} for their
virtual tangent spaces agree for {\it any} $\vec d^\square = \vec
e^\square$ via $\tilde\Psi_{\sT}$. .

For the equality of BS and quasimap vertices, it suffices to show that
$\tilde\Phi_{\sT}$ sends stable $\sT$-fixed quasimaps to valid
$\sT$-fixed $\pi$-stable pairs. This is done in
section~\ref{sec:stability-conditions}. However, it is in fact true
that $\tilde\Phi_{\sT}$ sends {\it all} stable quasimaps to valid
$\pi$-stable pairs: Proposition~\ref{prop:bs-qmaps-moduli-isomorphism}
shows that it provides a $\sT$-equivariant isomorphism of moduli
spaces, not just of $\sT$-fixed loci.

\subsection{Equivariant derived McKay}
\label{sec:bs-quasimaps-Tvir-correspondence}

\subsubsection{}

When $Y \to \fY$ is a crepant resolution of a sufficiently nice
orbifold singularity, it is generally expected that $D^b\cat{Coh}(Y)$
and $D^b\cat{Coh}(\fY)$ are equivalent. Historically, this was first
proved for ADE surfaces $S$, which are crepant resolutions of the
surface singularities $[\bC^2/\Gamma]$. The equivalence of categories
is given by a Fourier--Mukai transform, whose kernel is the incidence
correspondence
\[ \begin{tikzcd}
  & \Sigma \subset S \times \bC^2 \ar[swap]{dr}{q} \ar{dl}{p} \\
  S && \bC^2
\end{tikzcd} \]
given by viewing points on $S$ as $\Gamma$-orbits of $m$ points in
$\bC^2$. Later this was generalized to crepant resolutions of
threefold singularities $[M/\Gamma]$ for affine $M$
\cite{Bridgeland2001}.

\begin{theorem}[{\cite[Theorem 1.1]{Bridgeland2001}}] \label{thm:bkr}
  The Fourier--Mukai transform
  \[ \Psi(-) \coloneqq Rq_*\left(\cO_\Sigma \otimes p^*(- \otimes R_0)\right)\colon D^b\cat{Coh}(S) \to D^b\cat{Coh}_\Gamma(\bC^2) \]
  is an equivalence of categories.
\end{theorem}

Although in our setting $Y = \cA_{m-1} \times C$ is a threefold, it is
easier to understand $\Psi$ for the surface case and then use the
equivalence relatively over $C$. At the level of equivariant K-theory,
$\Psi$ induces an isomorphism
\[ \Psi\colon K(S) \cong K_\Gamma(\bC^2) \]
known earlier to Gonzalez-Sprinberg and Verdier
\cite{Gonzalez-Sprinberg1983}, which provides a geometric explanation
for the classical McKay correspondence. As such, the equivalence
$\Psi$ is known as the {\it derived} McKay correspondence. We use
$\Psi$ to denote both the derived equivalence and the K-theoretic
isomorphism.

\subsubsection{}

The first step in the proof of the BS/quasimaps correspondence is to match
the virtual tangent sheaves $T^{\vir}$ on both sides, which is the
content of Proposition~\ref{prop:bs-quasimaps-Tvir-correspondence}.
This requires us to compute that the bijection of
section~\ref{sec:bs-quasimaps-bijection} between $\sT$-fixed
$\pi$-stable pairs and quasimaps is just an application of the
equivalence $\Psi$. The isomorphism of K-theories then yields a
comparison of fibers of $T^{\vir}$.

Since the $T^{\vir}$ are $\sT$-equivariant objects, it is important to
extend the equivalence $\Psi$ to the {\it fully} $\sT$-equivariant
\[ \Psi_{\sT}\colon D^b\cat{Coh}_{\sT/\Gamma}(S) \to D^b\cat{Coh}_\sT(\bC^2) \]
which we call the {\it equivariant} derived McKay equivalence. It is
straightforward but slightly tedious to check that all machinery used
in the proof \cite[Section 6]{Bridgeland2001} of Theorem~\ref{thm:bkr}
holds in the $\sT$-equivariant setting as well. Note that this is not
automatic; a similar correspondence with Fourier--Mukai kernel
$\cO_\Sigma^\vee$ instead of $\cO_\Sigma$ is defined in
\cite{Kapranov2000} and shown to be an equivalence $D^b\cat{Coh}(S)
\simeq D^b\cat{Coh}_\Gamma(\bC^2)$, but it does {\it not} extend to a
$\sT$-equivariant equivalence.

\subsubsection{}

By the general theory of Fourier--Mukai transforms, the inverse of
$\Psi_{\sT}$ must be its left (and right) adjoint
\[ \Phi_{\sT}(-) \coloneqq \left[Rp_*(\cP \otimes q^*(-))\right]^\Gamma\colon D^b\cat{Coh}_\sT(\bC^2) \to D^b\cat{Coh}_{\sT/\Gamma}(S) \]
where, $\sT$-equivariantly,
\[ \cP \coloneqq \cO_\Sigma^\vee \otimes q^*(\omega_{\bC^2})[2] = xy \,\cO_\Sigma^\vee[2]. \]
Let $\iota\colon \{0\} \to \bC^2$ be the inclusion. The crucial
computation identifying quasimap data with $\pi$-stable pair data is
the image of skyscrapers $\iota_*w \coloneqq w \cdot \iota_*\cO_0$
under $\Phi_{\sT}$, for weights $w \in K_T(\pt)$. Note that
$\Phi_{\sT}$ is $\sT/\Gamma$-linear, namely
\[ \Phi_{\sT}(w_0 \cE) = w_0 \Phi_{\sT}(\cE) \quad \forall \, w_0 \in K_{\sT/\Gamma}(\pt). \]
Hence, for $S = \cA_{m-1}$, it suffices to compute $\Phi(\iota_* x^k)$
for $0 \le k < m$. There are many ways to do so, e.g. using the same
process as in \cite[Section 2]{Kapranov2000}, but since at the end of
the day we work only in K-theory, it is instructive and fairly
straightforward to do the calculation there and then lift the result
up to $D^b\cat{Coh}$ using some interesting general theory.

\begin{lemma} \label{lem:mckay-Phi-skyscrapers}
  For $S = \cA_{m-1}$,
  \[ \Phi_{\sT}(\iota_*x^k) = \begin{cases} \cO_E(-2)[1] & k = 0 \\ \cO_{E_k}(-1) & \text{otherwise}, \end{cases} \]
  where the $\sT$-linearizations are given by
  \begin{align*}
    \cO_E(-2)\big|_{p_0} &= x_0 \\
    \cO_{E_k}(-1)\big|_{p_k} &= x_k.
  \end{align*}
\end{lemma}

\begin{proof}
  A priori $\Phi_{\sT}(\iota_*x^k)$ is a complex of sheaves, but
  \cite[Theorem 1.1]{Cautis2009} shows it is actually some shift of a
  coherent sheaf. Computing $\Phi_{\sT}(\iota_*x^k)$ in equivariant
  K-theory identifies the sheaf, and then cohomological considerations
  identify the shift. Apply the equivariant Koszul resolution of
  $\cO_0$ simultaneously to all fibers of $\cO_\Sigma$ and dualize to
  get the equivariant resolution
  \[ xy p^*\rV^\vee \to x p^*\rV^\vee \oplus y p^*\rV^\vee \to p^*\rV^\vee \]
  of $xy\cO_\Sigma^\vee$. Here $\rV = \bigoplus_{k=0}^{m-1} \rV_k$ is
  the universal family on $\cA_{m-1}$ described in
  section~\ref{sec:universal-family-Am}; equivalently, $\rV =
  p_*\cO_\Sigma$. By push-pull, it follows that
  \begin{equation} \label{eq:resolution-of-mckay-Phi}
    \Phi_{\sT}(\iota_*x^k) = x^k \cdot \left[xy \rV^\vee_k \to x \rV^\vee_{k+1} \oplus y \rV^\vee_{k-1} \to \rV^\vee_k\right] \in D^b\cat{Coh}_{\sT/\Gamma}(\cA_{m-1}),
  \end{equation}
  where indices should be taken modulo $m$. It is now easy to identify
  the resulting sheaf by computing its K-theoretic weight in each
  chart or at each fixed point, using \eqref{eq:universal-family-Am}.
  For example, in an appropriate localization of
  $K_{\sT/\Gamma}(U_a)$, the modules
  \[ \Gamma(U_a, \Phi_{\sT}(\iota_* 1)) = \frac{1}{(1-x_a)(1-y_a)} \cdot \begin{cases} xy - x/y^{m-1} & a=0 \\ xy - 1 & 1 \le a < m-1 \\ xy - y/x^{m-1} & a = m-1 \end{cases} \]
  glue to form $-\cO_E(-2)$ with the desired linearization. (It is
  helpful to visualize this by drawing the appropriate box
  configuration in each chart.) Finally, \cite[Section 5]{Cautis2009}
  actually shows that \eqref{eq:resolution-of-mckay-Phi} has
  cohomological support in degrees $[-2, 0]$ when $k = 0$ and in
  degrees $[-1,0]$ when $k \neq 0$, which uniquely determines the
  cohomological shift of the resulting sheaves for both cases.
\end{proof}

\subsubsection{}

It remains to lift $\Phi_{\sT}$ to the threefold setting of $Y =
\cA_{m-1} \times \bP^1$. Let $X \coloneqq \bC^2 \times \bP^1$. Note
that performing $\Phi_{\sT}$ on fibers of $X$ relative to $\bP^1$
induces an equivalence
\[ \tilde\Phi_\sT\colon D^b\cat{Coh}_{\sT}(X) \to D^b\cat{Coh}_{\sT/\Gamma}(Y) \]
Let $\iota\colon \{0\} \times \bP^1 \hookrightarrow X$ be the
inclusion. Given a $\sT$-fixed quasimap $\cV \in \QMaps_V$, view it as
a flag of quiver sub-representations $\bV[n] \subset V$ as in
\eqref{eq:quasimaps-as-flags}, where the flag inclusions are
multiplication-by-$z$ maps. Then, by definition,
$\tilde\Phi_{\sT}(\iota_*\cV)$ is constructed from the pieces
$\Phi_{\sT}(\iota_*\bV[n])$ with $z$-weight $n$, linked together by
the induced multiplication-by-$z$ maps.

\begin{lemma} \label{lem:bs-quasimap-correspondence-is-mckay}
  The bijection of section~\ref{sec:bs-quasimaps-bijection} is exactly
  $\tilde\Phi_{\sT}$ in K-theory.
\end{lemma}

\begin{proof}
  From the discussion above and the $T/\Gamma$-linearity of $\tilde
  \Phi_{\sT}$, it suffices to compute the image of the hook $\hk_a$
  labeled with the degrees $\vec d^\square$ and show that the result
  is the local model with data $\vec e^\square = \vec d^\square$. Here
  $\hk_a$ is the $m$-hook corresponding to the $a$-th leg, from
  \eqref{eq:m-hook}. Using the calculation of
  Lemma~\ref{lem:mckay-Phi-skyscrapers}:
  \begin{itemize}
  \item $\Phi_{\sT}(\iota_*\hk_a)$ is a single box $\cube \in U_a$ of
    weight $1$;
  \item $\Phi_{\sT}(\iota_*(x^{a-k}+ \cdots + x^a))$ is a standard
    leftward rod of length $k$ from $\cube$;
  \item $\Phi_{\sT}(\iota_*(y^{a-k}+ \cdots + y^a))$ is a standard
    rightward rod of length $k$ from $\cube$. Here we are using that
    $\Phi_{\sT}(\iota_*y^{m-j}) = y_j \Phi_{\sT}(\iota_*x^j)$.
  \end{itemize}
  Attaching these pieces together with the natural
  multiplication-by-$z$ maps yields the desired local model.
\end{proof}

\subsubsection{}

Let $I^\bullet = [\cO_Y \to \cF]$ be a $\sT$-fixed $\pi$-stable pair
with corresponding $\sT$-fixed quasimap $\cV$. In this notation,
Lemma~\ref{lem:bs-quasimap-correspondence-is-mckay} shows that
\[ \cF = \tilde\Phi_{\sT}(\iota_*\cV) \in K_{\sT/\Gamma}(Y). \]
It is easy to verify that $\tilde\Phi_{\sT}(\cO_X) = \cO_Y$. Since
$I^\bullet = \cO_Y - \cF$, it follows that
\[ I^\bullet = \tilde\Phi_{\sT}(I_{\cV}) \]
where $I_{\cV} \coloneqq \cO_X - i_*\cV$.

\begin{proposition} \label{prop:bs-quasimaps-Tvir-correspondence}
  In $K_{\sT/\Gamma}(\pt)$,
  \[ T^{\vir}_{[I^\bullet]}\PiPairs = T^{\vir}_{\cV}\QMaps. \]
\end{proposition}

\begin{proof}
  By the local-to-global Ext spectral sequence,
  \[ T^{\vir}_{[I^\bullet]} = H^\bullet_Y\left(\cExt_Y^\bullet(\cO_Y, \cO_Y) - \cExt_Y^\bullet(I^\bullet, I^\bullet)\right). \]
  Using that $\tilde\Phi_{\sT}$ is an equivalence of categories,
  \[ T^{\vir}_{[I^\bullet]} = H^\bullet_X\left(\cExt_X^\bullet(\cO_X, \cO_X) - \cExt_X^\bullet(I_{\cV}, I_{\cV})\right)^\Gamma. \]
  The quantity in the rhs that $H^\bullet_X$ is applied to is a sheafy
  version of $T_V\Hilb(\bC^2)$, for which the usual formula applies
  (see e.g. \cite[Section 4.7]{Maulik2006}):
  \begin{equation} \label{eq:relative-tangent-space-hilbert-scheme-C2}
    \cExt_X^\bullet(\cO_X, \cO_X) - \cExt_X^\bullet(I_{\cV}, I_{\cV}) = \iota_*\left(\cV + \frac{\cV^\vee}{xy} - \cV \cV^\vee \frac{(1-x)(1-y)}{xy}\right).
  \end{equation}
  Use $H^\bullet_X \circ \iota_* = H^\bullet_{\bP^1}$ to simplify.
  Since $(-)^\Gamma$ commutes with $H^\bullet_{\bP^1}$, we can take
  $\Gamma$-invariants first. But $\Gamma$-invariants of
  \eqref{eq:relative-tangent-space-hilbert-scheme-C2} yields exactly a
  sheafy version of $T_V\Hilb(\cA_{m-1})$. Then taking
  $H^\bullet_{\bP^1}$ gives $T^{\vir}_{\cV}\QMaps$, as in
  \eqref{eq:quasimaps-Tvir}.
\end{proof}

\subsubsection{}
\label{sec:bs-quasimap-correspondence-refinement}

Although the description of $\tilde\Phi_{\sT}$ so far is as a {\it
  bijection} between $\sT$-fixed components of $\pi$-stable pairs and
quasimaps, it naturally refines into an {\it isomorphism} of fixed
components as follows. The description of $\Phi_{\sT}$ in
Lemma~\ref{lem:mckay-Phi-skyscrapers} should be viewed as a
correspondence between individual squares in each $\bV[n]$ of a
quasimap and individual standard rods in $\cF$ of a $\pi$-stable pair.
The remaining data on either side is how these individual pieces
attach or glue to each other, either via quiver maps or $\cO_Y$-module
maps. Hence it remains to match the gluing data on both sides.

\begin{example}
  On $\cA_1$, remembering the quiver maps from color-$0$ squares (to
  non-color-$0$ squares) as $\tilde\Phi_{\sT}$ is applied yields, for
  example,
  \begin{equation} \label{eq:bs-quasimap-quiver-maps}
    \begin{tikzpicture}[scale=0.5,baseline={([yshift=0.5ex]current bounding box.center)}]
      \draw (0,0) rectangle (1,2);
      \draw (0,1) -- (2,1) -- (2,0) -- (1,0);
      \draw[blue,->] (0.5,0.2) to[looseness=2,in=-90,out=-90] node[label=right:$a$]{} (1.5,0.2);
      \draw[blue,->] (0.2,0.5) to[looseness=2,in=180,out=180] node[label=left:$b$]{} (0.2,1.5);
    \end{tikzpicture}
    \quad \leadsto \quad
    \left[\cO_E(-2) \xrightarrow{\begin{psmallmatrix} \textcolor{blue}{a} \\ \textcolor{blue}{b} \end{psmallmatrix}} \cO_E(-1) \oplus \cO_E(-1) y/x\right]
  \end{equation}
  where the linearizations on the sheaves are the ones in
  Lemma~\ref{lem:mckay-Phi-skyscrapers}. If both $a$ and $b$ are
  non-zero, then the resulting complex is quasi-isomorphic to the
  $\cO_Y$-module $\cO_E$, whereas if only one is zero, then the result
  is (non-equivariantly) $\cO_{\pt} \oplus \cO_E(-1)$. This matches
  the degrees of freedom (for a $\bP^1$ fixed component) in
  Example~\ref{ex:bs-nontrivial-fixed-loci} and
  Example~\ref{ex:quasimaps-nontrivial-fixed-loci}.
\end{example}

The above example is representative of the general case, at any
$\cA_1$ piece inside $\cA_{m-1}$, for how quiver maps from color-$0$
squares correspond to $\cO_Y$-module maps involving certain
boxes/rods.

For quiver maps from non-color-$0$ squares, the $\cO_Y$-module maps
that correspond to each quiver map are described implicitly in the
proof of Lemma~\ref{lem:bs-quasimap-module-condition}. Split each
$V^{(a)}(\vec\lambda)$ into its hooks. The quiver maps between squares
from the same hook correspond to $\cO_Y$-module maps within each local
model prescribing how standard rods are attached to each other. The
quiver maps between squares from different hooks correspond to maps
from white to gray boxes in Figures
\ref{fig:bs-local-model-adjacent-inequalities} and
\ref{fig:bs-local-model-diagonal-inequalities}.

\begin{remark}
  After localization to fixed components, the vertex on either side is
  computed by a further integration $\chi(F, \hat\cO^{\vir})$ over
  each fixed component $F$. The isomorphism of fixed components along
  with the equality of $T^{\vir}$ given by
  Proposition~\ref{prop:bs-quasimaps-Tvir-correspondence} implies the
  equality of $\chi(F, \hat\cO^{\vir})$ in BS and quasimap theories.
\end{remark}

\subsection{Stability conditions}
\label{sec:stability-conditions}

\subsubsection{}

It remains to show that $F \subset \PiPairs_{\vec\lambda}^{\sT}$ iff
$\tilde\Phi_{\sT}(F) \subset \QMaps_{\vec\lambda}^{\sT}$. The idea is
to identify two combinatorial conditions on the degrees
\[ \{d_\vsquare \; | \; \vsquare \in V(\vec\lambda)\}, \]
and then to show that on the BS side they are equivalent to the
corresponding $\cF$ forming a $\pi$-stable pair, and on the quasimap
side they are equivalent to the corresponding $\cV$ forming a stable
quasimap. 

\begin{proposition} \label{prop:bs-quasimap-stability-correspondence}
  The degrees
  \[ \{d_\vsquare \; | \; \vsquare \in V(\vec\lambda)\} \]
  define a $\pi$-stable pair $[\cO_Y \to \cF] \in \PiPairs_V$ iff they
  define a stable quasimap $\cV \in \QMaps_V$.
\end{proposition}

\begin{proof}
  It will follow from Lemma~\ref{lem:bs-quasimap-module-condition}
  that $\cF$ is a valid $\cO_Y$-module iff $\cV$ is a valid quiver
  bundle. Proposition~\ref{prop:bs-local-model-converse} shows that a
  sheaf $\cF$ built from local models is a $\pi$-stable pair iff it
  admits an inclusion $\cO_C \hookrightarrow \cF$. Similarly,
  Proposition~\ref{prop:quasimap-as-framed-quiver-bundle} shows that a
  quiver bundle $\cV$ is a stable quasimap iff it admits a framing
  $\cW \hookrightarrow \cV$. Then, to conclude,
  Lemma~\ref{lem:bs-quasimaps-framing} will show that the quiver
  bundle $\cV$ admits an inclusion from $\cW = \cO_{\bP^1}$ iff $\cF$
  admits an inclusion from $\cO_C$.
\end{proof}

\subsubsection{}

Let $\vec\lambda \in \Hilb(\cA_{m-1})^{\sT'}$, and consider the
labeling of $V(\vec\lambda) = \sum_a V^{(a)}(\vec\lambda)$ by the
degrees $\{d_{\vsquare}\}$. By viewing quiver bundles as flags of
quiver representations or otherwise, it is clear that in order for
this labeling to produce a valid quiver bundle $\cV$, the degrees of
boxes in each $V^{(a)}$ must satisfy the property
\begin{equation} \label{eq:quasimap-degree-condition}
  d_\vsquare \le d_{x \cdot \vsquare}, d_{y \cdot \vsquare} \quad \forall \;\vsquare \in V^{(a)}
\end{equation}
whenever the squares $x \cdot \vsquare$ and $y \cdot \vsquare$ exist
in $V^{(a)}$. There is an ambiguity in identifying which degrees are
for which $V^{(a)}$ at squares with multiplicity, but in order for the
resulting $\cV$ to form a valid quiver bundle there must be at least
one identification satisfying property
\eqref{eq:quasimap-degree-condition}. This is analogous to picking the
filtration for BS local models in
Proposition~\ref{prop:bs-local-model-filtration}. Hence, in what
follows, it suffices to work with a {\it single} $V^{(a)}$.
Equivalently, we assume that $\vec\lambda$ is empty except at
$\lambda_a \neq \emptyset$, so that $V = V^{(a)}$.

\begin{lemma} \label{lem:bs-quasimap-module-condition}
  The degrees
  \[ \{d_\vsquare \; | \; \vsquare \in V^{(a)}\} \]
  define a valid $\cO_Y$-module $\cF^{(a)}$ iff they satisfy the
  condition \eqref{eq:quasimap-degree-condition}.
\end{lemma}

\begin{proof}
  It is easier to group the squares in $V^{(a)}$ into the hooks
  associated to each $\square \in \lambda_a$, and work with the
  degrees $\{\vec d^\square\}_{\square \in \lambda_a}$ instead of the
  degrees $\{d_\vsquare\}_{\vsquare \in V^{(a)}}$. Then to check
  condition \eqref{eq:quasimap-degree-condition} using $\{\vec
  d^\square\}_{\square \in \lambda_a}$, it is equivalent to check it
  holds for the squares in each individual hook and then check that it
  holds for adjacent squares coming from two different hooks. For
  example, consider the hook in $V^{(a)}$ associated to $\square \in
  \lambda_a$, whose squares have degrees $\vec d^\square$. From the
  discussion of section~\ref{sec:bs-local-model}, it is clear that the
  local model given by the single hook forms a valid $\cO_Y$-module
  iff the degrees of the hook satisfy
  \eqref{eq:quasimap-degree-condition}.

  Now consider two different hooks given by $\square$ and $\square'
  \coloneqq x_a \cdot \square$, whenever both are in $\lambda_a$. The
  corresponding hooks in $V^{(a)}$ are adjacent at exactly one
  location, and the desired inequality is
  \[ d^{\square'}_{a+1} \ge d^{\square}_a. \]
  The relevant portion of the local models is shown in
  Figure~\ref{fig:bs-local-model-adjacent-inequalities-right}, where
  boxes in the local model for $\square'$ are colored gray. Suppose
  $\cube \coloneqq z^k \cdot \square$ generates a standard leftward
  rod. Since the $\square'$ column exists in $L_a$, it must also
  generate a box $\cube\,' \coloneqq x_a \cdot \cube$. This box
  $\cube\,'$ must be part of the local model associated to $\square'$.
  Moreover, $\cube\,'$ cannot generate a standard leftward rod;
  otherwise there is no non-zero $\cO_Y$-module map $x_a \colon \cube
  \to \cube\,'$. Hence $\cube\,'$ generates a standard rightward rod,
  and the desired inequality follows. Note that an analogous argument
  can be given for $\square$ and $\square' \coloneqq y_a \cdot
  \square$, as shown in
  Figure~\ref{fig:bs-local-model-adjacent-inequalities-up}, where the
  desired inequality is $d^{\square'}_a \ge d^{\square}_{a+1}$.

  \begin{figure}[h]
    \centering
    \begin{subfigure}[t]{0.45\textwidth}
      \centering
      \begin{tikzpicture}[scale=0.6]
        \foreach \y in {0,-1,...,-4} {
          \foreach \x in {0,1,2} {
            \draw[fill=gray!40] (\x,\y) -- ++(1,0) -- ++(0,-1) -- ++(-1,0) -- ++(0,1);
          }
          \draw[fill=gray!40] (3.5,\y) -- ++(-0.5,0) -- ++(0,-1) -- ++(0.5,0);
        }
        \draw[fill=gray!40] (4,0.5) -- ++(-0.5,0) -- ++(-0.5,-0.5) -- ++(0.5,0);
        \draw[fill=gray!40] (0,1) -- ++(1,0) -- ++(0,-1) -- ++(-1,0) -- ++(0,1);
        \foreach \x in {1,2} {
          \draw[fill=gray!40] (\x,0) -- ++(1,0) -- ++(0.5,0.5) -- ++(-1,0) -- ++(-0.5,-0.5);
        }
        \draw[fill=gray!40] (1,0) -- ++(0,1) -- ++(0.5,0.5) -- ++(0,-1) -- ++(-0.5,-0.5);
        \draw[fill=gray!40] (0,1.5) -- (0,1) -- (1,1) -- (1,1.5);
        \draw[fill=gray!40] (1,1.5) -- (1,1) -- (1.5,1.5) -- (1.5,2);
        \draw (-1,0) -- (-2.25,-1.25);
        \foreach \y in {0,-1,-2,-3} {
          \draw (3.5,\y) -- (0,\y) -- (-1.25,\y-1.25);
        };
        \foreach \x in {0,-0.5,-1} {
          \draw (\x-1,\x) -- (\x,\x) -- (\x,\x-3);
        };
        \draw (-1,1.5) -- (-1,0);
        \draw (-1,1) -- (0,1);
        \node at (-2,-2.55) {$z^{d^{\square}_a-2}$};
        \node at (-2,-3.55) {$z^{d^{\square}_a-1}$};
        \node at (-2,-4.55) {$z^{d^{\square}_a}$};
        \node at (4.3,-3) {$z^{d^{\square}_a}$};
        \node at (4.3,-5) {$z^{d^{\square'}_{a+1}}$};
      \end{tikzpicture}
      \caption{Multiplication by $x_a$}
      \label{fig:bs-local-model-adjacent-inequalities-right}
    \end{subfigure}%
    ~
    \begin{subfigure}[t]{0.45\textwidth}
      \centering
      \begin{tikzpicture}[scale=0.6]
        \foreach \y in {0,-1,...,-4} {
          \foreach \x in {0,-0.5,-1} {
            \draw[fill=gray!40] (\x,\y+\x) -- ++(-0.5,-0.5) -- ++(0,-1) -- ++(0.5,0.5) -- ++(0,1);
          }
          \draw[fill=gray!40] (-1.75,\y-1.75) -- ++(0.25,0.25) -- ++(0,-1) -- ++(-0.25,-0.25);
        }
        \draw[fill=gray!40] (0,1) -- ++(-0.5,-0.5) -- ++(0,-1) -- ++(0.5,0.5) -- ++(0,1);
        \draw[fill=gray!40] (-0.5,0.5) -- ++(-1,0) -- ++(0,-1) -- ++(1,0) -- ++(0,1);

        \foreach \x/\y in {-0.5/-0.5,-1/-1} {
          \draw[fill=gray!40] (\x,\y) -- ++(-1,0) -- ++(-0.5,-0.5) -- ++(1,0) -- ++(0.5,0.5);
        }
        \draw[fill=gray!40] (-1.75,-1.75) -- ++(0.25,0.25) -- ++(-1,0) -- ++(-0.25,-0.25);
        \draw[fill=gray!40] (0,1.5) -- (0,1) -- (-0.5,0.5) -- (-0.5,1);
        \draw[fill=gray!40] (-0.5,1) -- (-0.5,0.5) -- (-1.5,0.5) -- (-1.5,1);
        \draw (3,0.5) -- (0.5,0.5) -- (0,0);
        \foreach \x in {1,2} {
          \draw (\x+0.5,0.5) -- (\x,0) -- (\x,-3);
        };
        \foreach \y in {0,-1,-2,-3} {
          \draw (0,\y) -- ++(2.5,0);
        }
        \draw (0.5,0.5) -- ++(0,1.5);
        \draw (0,1) -- ++(0.5,0.5);
        \node at (-2.4,-4.85) {$z^{d^{\square}_{a+1}}$};
        \node at (-2.4,-6.85) {$z^{d^{\square'}_a}$};
        \node at (3.6,-1) {$z^{d^{\square}_{a+1}-2}$};
        \node at (3.6,-2) {$z^{d^{\square}_{a+1}-1}$};
        \node at (3.6,-3) {$z^{d^{\square}_{a+1}}$};
      \end{tikzpicture}
      \caption{Multiplication by $y_a$}
      \label{fig:bs-local-model-adjacent-inequalities-up}
    \end{subfigure}
    \caption{Portions of local models associated to adjacent columns}
    \label{fig:bs-local-model-adjacent-inequalities}
  \end{figure}
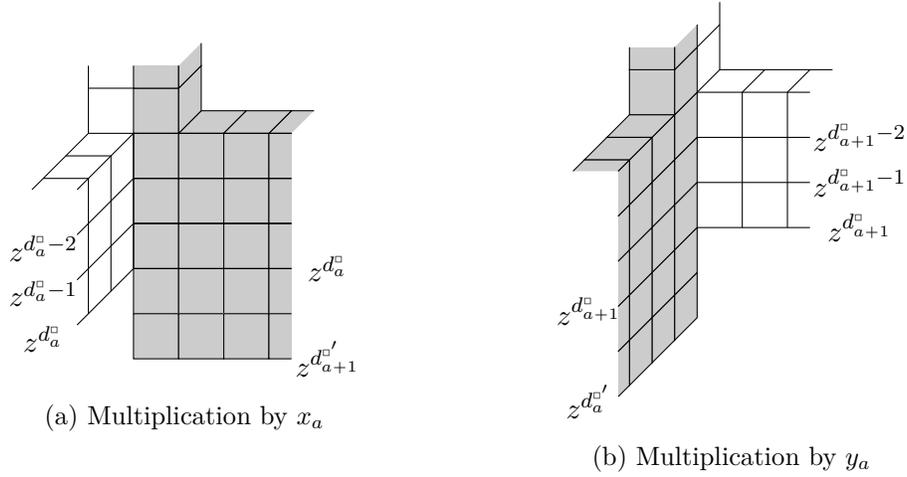
  
  Equivalently, using the notation of
  section~\ref{sec:bs-boxes-rods-assembly}, the inequality follows
  from the non-zero $\cO_Y$-module maps of $\cF^z$ between adjacent
  columns. These multiplication-by-$x_a$ or $y_a$ maps are non-zero
  because they are non-zero in $L_a$, and all boxes in $\cF^z$
  eventually generate $L_a$.

  Finally, consider $\square$ and $\square' \coloneqq xy \cdot
  \square$. This is the only remaining case where the corresponding
  hooks have adjacent squares in $V^{(a)}$. Again, non-zero
  $\cO_Y$-module maps in $\cF^z$ yield the inequalities
  \[ d^{\square'}_0 \ge d^{\square}_{m-1}, d^{\square}_1. \]
  (When $a = 0$ or $a = m-1$, only the appropriate one applies.) It
  remains to explain the inequalities
  \begin{subequations}
    \begin{align}
      d^{\square'}_b &\ge d^{\square}_{b-1} \quad \forall\; a < b < m-1 \label{eq:quasimap-diagonal-inequality-right} \\
      d^{\square'}_b &\ge d^{\square}_{b+1} \quad \forall\; 0 < b \le a. \label{eq:quasimap-diagonal-inequality-up}
    \end{align}
  \end{subequations}
  These arise from the behavior of standard rods in the $\cO \oplus
  \cO(-2)$ geometry. Namely, if both the boxes $\cube \coloneqq z^k
  \cdot \square$ and $xy \cdot \cube$ generate standard rightward
  rods $\cR$ and $\cR'$ of length $\ell$ and $\ell'$, it must be that
  $\ell' > \ell$. This is because the right-most box of $\cR$, of
  weight $y_{a+\ell} \cdot \cube$, must generate in the $x_{a+\ell}$
  direction another box of weight
  \[ x_{a+\ell} y_{a+\ell} \cdot \cube  = xy \cdot \cube \]
  which must belong to $\cR'$. If $\cR'$ were of length $\ell$, its
  right-most box would be $y_{a+\ell} xy \cdot \cube$, a
  contradiction. This is shown in
  Figure~\ref{fig:bs-local-model-diagonal-inequalities-right}, where
  $\cR'$ is in gray. Hence $\ell' > \ell$, which yields
  \eqref{eq:quasimap-diagonal-inequality-right}. The analogous
  argument holds for standard leftward rods, shown in
  Figure~\ref{fig:bs-local-model-diagonal-inequalities-up}, and yields
  \eqref{eq:quasimap-diagonal-inequality-up}.

  \begin{figure}[h]
    \centering
    \begin{subfigure}[t]{0.45\textwidth}
      \centering
      \begin{tikzpicture}[scale=0.6,xscale=-1]
        \foreach \x/\y in {-1/0, 0/0, 1/0, 2/0, 3/0, -1.5/-0.5, -2/-1} {
          \draw[fill=gray!40] (\x,\y) -- ++(1,0) -- ++(0.5,0.5) -- ++(-1,0) -- ++(-0.5,-0.5);
        }
        \foreach \x/\y in {-1/0,0/0,1/0,2/0,3/0} {
          \draw (\x+0.5,\y+0.5) -- ++(1,0) -- ++(0.5,0.5) -- ++(-1,0) -- ++(-0.5,-0.5);
        }
        \foreach \x in {0,1,...,3} {
          \draw[fill=gray!40] (\x,0) -- ++(1,0) -- ++(0,-1) -- ++(-1,0) -- ++(0,1);
        }
        \foreach \x/\y in {-1/-1,-0.5/-0.5} {
          \draw[fill=gray!40] (\x,\y) -- ++(0,-1) -- ++(0.5,0.5) -- ++(0,1) -- ++(-0.5,-0.5);
        }
        \draw[fill=gray!40] (-2.25,-1.25) -- ++(0.25,0.25) -- ++(1,0) -- ++(-0.25,-0.25);
        \draw[fill=gray!40] (-1.25,-1.25) -- ++(0.25,0.25) -- ++(0,-1) -- ++(-0.25,-0.25);
        \draw[fill=gray!40] (4.5,0) -- ++(-0.5,0) -- ++(0,-1) -- ++(0.5,0);
        \draw[fill=gray!40] (4.5,0) -- ++(-0.5,0) -- ++(0.5,0.5) -- ++(0.5,0);
        \draw (5,0.5) -- ++(-0.5,0) -- ++(0.5,0.5) -- ++(0.5,0);
      \end{tikzpicture}
      \caption{Multiplication by $x_b$}
      \label{fig:bs-local-model-diagonal-inequalities-right}
    \end{subfigure}%
    ~
    \begin{subfigure}[t]{0.45\textwidth}
      \centering
      \begin{tikzpicture}[scale=0.6]
        \foreach \x/\y in {-1/0, 0/0, 1/0, 2/0, 3/0, -1.5/-0.5, -2/-1} {
          \draw[fill=gray!40] (\x,\y) -- ++(1,0) -- ++(0.5,0.5) -- ++(-1,0) -- ++(-0.5,-0.5);
        }
        \foreach \x/\y in {-1/0,0/0,1/0,2/0,3/0} {
          \draw (\x+0.5,\y+0.5) -- ++(1,0) -- ++(0.5,0.5) -- ++(-1,0) -- ++(-0.5,-0.5);
        }
        \foreach \x in {0,1,...,3} {
          \draw[fill=gray!40] (\x,0) -- ++(1,0) -- ++(0,-1) -- ++(-1,0) -- ++(0,1);
        }
        \foreach \x/\y in {-1/-1,-0.5/-0.5} {
          \draw[fill=gray!40] (\x,\y) -- ++(0,-1) -- ++(0.5,0.5) -- ++(0,1) -- ++(-0.5,-0.5);
        }
        \draw[fill=gray!40] (-2.25,-1.25) -- ++(0.25,0.25) -- ++(1,0) -- ++(-0.25,-0.25);
        \draw[fill=gray!40] (-1.25,-1.25) -- ++(0.25,0.25) -- ++(0,-1) -- ++(-0.25,-0.25);
        \draw[fill=gray!40] (4.5,0) -- ++(-0.5,0) -- ++(0,-1) -- ++(0.5,0);
        \draw[fill=gray!40] (4.5,0) -- ++(-0.5,0) -- ++(0.5,0.5) -- ++(0.5,0);
        \draw (5,0.5) -- ++(-0.5,0) -- ++(0.5,0.5) -- ++(0.5,0);
      \end{tikzpicture}
      \caption{Multiplication by $y_b$}
      \label{fig:bs-local-model-diagonal-inequalities-up}
    \end{subfigure}
    \caption{Portions of adjacent local models}
    \label{fig:bs-local-model-diagonal-inequalities}
  \end{figure}
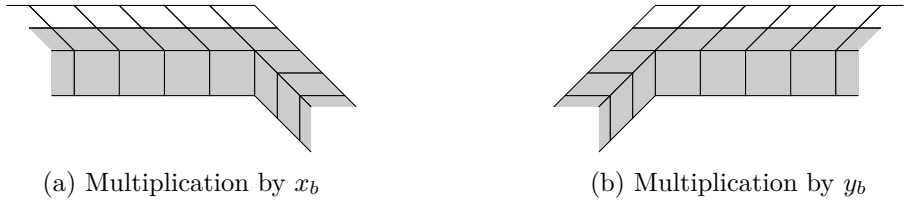

  We have just shown that every $\cO_Y$-module $\cF^{(a)}$ with one
  non-trivial external leg formed from local models yield quasimap
  degrees $\{d_\vsquare\}_{\vsquare \in V^{(a)}}$ satisfying
  \eqref{eq:quasimap-degree-condition} by verifying the necessary
  inequalities. Conversely, it is clear that these inequalities are
  sufficient to create a valid $\cO_Y$-module from local models.
\end{proof}

\subsubsection{}

Within $\cF = \tilde\Phi_{\sT}(\iota_*\cV)$ must be a unique subsheaf
$\cO_C$ which is the structure sheaf of a CM curve containing the
infinite legs $L_a$. To finish the proof of
Proposition~\ref{prop:bs-quasimap-stability-correspondence}, we need
to identify the conditions on $\cV$ in order for there to be an
inclusion $\cO_C \hookrightarrow \cF$. 

\begin{lemma} \label{lem:bs-quasimaps-framing}
  Let $\cF = \tilde\Phi_{\sT}(\iota_*\cV)$ be an $\cO_Y$-module. Then
  $\cF$ admits an inclusion $\cO_C \hookrightarrow \cF$ iff $\cV$
  admits a framing $\cW \hookrightarrow \cV$.
\end{lemma}

\begin{proof}
  Consider the subsheaf $\cF_C \subset \cF/\inner{xy, z}\cF$ generated
  by boxes $\cube = (0,0,0) \in U_a$ whenever they exist. Then (the
  pre-image of) $\cF_C$ must generate $\cO_C \subset \cF$. Hence $\cF$
  has a subsheaf of the form $\cO_C$ iff $\cF_C$ is the direct sum of
  $\cO_{E_{ab}}$ (with the trivial linearization) where $E_{ab}
  \coloneqq E_a \cup E_{a+1} \cdots \cup E_b$ are connected components
  of $\supp \cF_C$. In other words, $\cF_C$ is built from rods of
  degree $(\ldots, -\infty, 0, 0, \ldots, 0, -\infty, \ldots)$; these
  are internal legs.

  In the corresponding quasimap $\cV$, the standard rods in $\cF_C$
  are specified by the hooks $\hk^\square$ for $\square = (0,0) \in
  \lambda_a$. Note that in \eqref{eq:bs-quasimap-quiver-maps}, the
  leftward and rightward standard rods can attach to form $\cO_E$ iff
  the maps $a$ and $b$ satisfy $[a : b] = [1 : 1] \in \bP^1$. This is
  also discussed in Example~\ref{ex:bs-nontrivial-fixed-loci}. The
  generalization to $\cA_{m-1}$ is as follows. Let $\bV_C[n]$ be the
  vector space at $\square = (0,0)$ in $\bV[n]$, and consider the
  sub-flag
  \[ \bV_C[0] \subset \bV_C[1] \subset \cdots \subset \bV_C[\infty] \]
  of the flag $\bV^\bullet/\inner{x,y}\bV^\bullet$. This is the
  quasimap analogue of $\cF_C$. Then the leftward and rightward
  standard rods generated by boxes $\cube = (0,0,0) = U_a$ are
  attached to form $\cO_{E_{ab}}$ iff this flag preserves the
  inclusion of the vector $(1, 1, \ldots, 1) \in \bV_C[\infty]$. But
  this is precisely the data of a framing $\cW \hookrightarrow \cV$.
\end{proof}

If $\dim \bV_C[\infty] = n$, then $\cF$ has $n$ non-trivial (external)
legs given by $n$ different hooks. The basic idea is that as we
descend in the sub-flag $\bV_C^\bullet$, standard rods in two hooks
$\hk^{(0,0)}_a$ and $\hk^{(0,0)}_b$ ``glue'' via a non-trivial step in
the flag induced by
\[ \bC \xhookrightarrow{1 \mapsto \begin{psmallmatrix} 1 \\ 1 \end{psmallmatrix}} \bC_a \oplus \bC_b \]
where $\bC_a$ and $\bC_b$ are the subspaces in $\bV_C[\infty]$
corresponding to $\square = (0,0)$ in the two hooks. So $\dim
\bV_C[\infty] - \dim \bV_C[0]$ is the number of exceptional curves
$E_a$ in $\supp \cO_C$.

\subsubsection{}

Although the equality of BS and quasimap vertices only requires
matching stability conditions for $\sT$-fixed loci, it is in fact true
that stability conditions match {\it globally} for the entire moduli
spaces of quasimaps and $\pi$-stable pairs. The following proof was
suggested by A. Okounkov.

\begin{proposition} \label{prop:bs-qmaps-moduli-isomorphism}
  There is a $\sT$-equivariant isomorphism
  \[ \PiPairs_{\text{nonsing }\infty}(S \times \bP^1) \simeq \QMaps_{\text{nonsing }\infty}(\Hilb(S)). \]
\end{proposition}

\begin{proof}
  Let $Y = S \times \bP^1$ and $X = \bC^2 \times \bP^1$. Consider the
  image
  \[ \sP \coloneqq \tilde\Psi_{\sT}(\PiPairs_{\text{nonsing }\infty}(Y)) \subset D^b\cat{Coh}_{\sT}(X). \]
  The $\sT$-action on both moduli spaces means they are unions of
  attracting subschemes for components of their $\sT$-fixed loci.
  Since $\tilde\Psi_{\sT}$ is an isomorphism on the $\sT$-fixed locus:
  \begin{itemize}
  \item $\sP$ contains the $\sT$-fixed locus $\QMaps_{\text{nonsing
    }\infty}(\Hilb(S))^{\sT}$;
  \item by the upper semi-continuity of sheaf cohomology, elements of
    $\sP$ continue to be supported in a single cohomological degree
    away from $\sT$-fixed loci as well, so $\sP \subset
    \cat{Coh}_{\sT}(X)$.
  \end{itemize}
  In addition, the equivalence $\Psi_{\sT}$ sends sheaves supported on
  the exceptional divisor $E \subset \cA_{m-1}$ to sheaves supported
  at $0 \in \bC^2$, so fibers over $\bP^1$ of elements in $\sP$ are
  supported at $0 \in \bC^2$. Therefore elements of $\sP$ are (not
  necessarily stable) quasimaps, mapping at $\infty \in \bP^1$ to the
  same point as for the original $\pi$-stable pair. But stability is
  an open and $\sT$-invariant condition, and therefore the
  destabilizing locus is contained in the $\sT$-fixed locus. We
  already know the $\sT$-fixed locus of $\sP$ is stable. Hence $\sP
  \subset \QMaps_{\text{nonsing }\infty}(\Hilb(S))$. An analogous
  argument for $\tilde\Phi_{\sT}$ shows it is an isomorphism.
\end{proof}

\section{3d mirror symmetry}

\subsection{Via quasimap theory}
\label{sec:3d-mirror-symmetry}

\subsubsection{}

The (classical and quantum) geometry of certain Nakajima quiver
varieties $X$ is intimately related to the geometry of a {\it mirror}
Nakajima quiver variety $\check X$. In physics, this relation comes
from {\it 3d mirror symmetry}, also called {\it symplectic duality},
which is an S-duality between certain 3d $\cN = 4$ supersymmetric
gauge theories. In our setting, these gauge theories are associated to
a quiver and its 3d mirror, which can be constructed following
\cite[Section 3.3]{Boer1997}. Specifically, affine type A quivers are
3d mirror to other affine type A quivers, but dimension vectors for
their associated Nakajima quiver varieties are exchanged as in
Figure~\ref{fig:3d-mirror-symmetry}. We say the affine type A quiver
is (conjecturally) {\it self-mirror}.

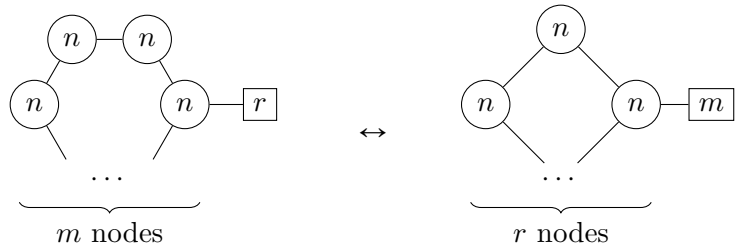
\begin{figure}[h]
  \centering
  \begin{tikzpicture}
    \begin{scope}
      \node[draw] (w) at (2,0) {$r$};
      \foreach \i in {0,1,2,3} {
        \node[draw,circle] (v\i) at (60*\i:1) {$n$};
      }
      \node (v4) at (60*4:1) {};
      \node (v5) at (60*5:1) {};
      \draw (w) to (v0);
      \foreach \i in {0,1,2,3,5} {
        \pgfmathtruncatemacro{\ii}{int(mod(\i+1,6))}
        \draw (v\i) to (v\ii);
      }
      \node at (60*4.5:1) {$\cdots$};
      \draw[decorate,decoration={brace,amplitude=4,mirror}] (-1.2,-1.3) -- node[label=below:$m\text{ nodes}$]{} (1.2,-1.3);
    \end{scope}
    \node at (3.5,-0.4) {$\leftrightarrow$};
    \begin{scope}[xshift=6cm]
      \node[draw] (w) at (2,0) {$m$};
      \foreach \i in {0,1,2} {
        \node[draw,circle] (v\i) at (90*\i:1) {$n$};
      }
      \node (v3) at (90*3:1) {$\cdots$};
      \draw (w) to (v0);
      \foreach \i in {0,1,2,3} {
        \pgfmathtruncatemacro{\ii}{int(mod(\i+1,4))}
        \draw (v\i) to (v\ii);
      }
      \draw[decorate,decoration={brace,amplitude=4,mirror}] (-1.2,-1.3) -- node[label=below:$r\text{ nodes}$]{} (1.2,-1.3);
    \end{scope}
  \end{tikzpicture}
  \caption{3d mirror symmetry of affine type A quivers}
  \label{fig:3d-mirror-symmetry}
\end{figure}

To identify the resulting Nakajima quiver varieties for the two
chambers $C_+$ and $C_-(m)$ that we are interested in, one can extend
the isomorphisms of Theorems~\ref{thm:Hilb-orbifold-stability-chamber}
and \ref{thm:Hilb-regular-stability-chamber} to isomorphisms
\begin{align*}
  \cM_{\vec\theta \in C_+}(n\vec v^0, r\vec w^0) &\cong \cM_r^n([\bC^2/\Gamma])\\ 
  \cM_{\vec\theta \in C_-(m)}(n\vec v^0, r\vec w^0) &\cong \cM_r^n(S)
\end{align*}
where $\cM_r^n(S)$ is the moduli of rank-$r$ instantons on $S$ of
instanton number $n$. To be precise, an instanton on the surface $S$
is a torsion-free sheaf $\cE$ on the compactified resolution $\bar S
\to \bP^2/\Gamma$ along with a choice of framing $\cE\big|_{\bP} \cong
\cO_{\bP}^{\oplus r}$ on the line at infinity $\bP \subset \bP^2$. The
same applies to $\cM_r^n([\bC^2/\Gamma])$. Note that $\cM_1^n(-) =
\Hilb^n(-)$.

\subsubsection{}

We will now describe some standard mathematical expectations
\ref{item:3d-mirror-assumption-1}, \ref{item:3d-mirror-assumption-2},
and \ref{item:3d-mirror-assumption-3} for what 3d mirror symmetry
entails (cf. \cite[Section 1.6]{Dinkins2019a}, \cite[Definition
  1.1]{Rimanyi2019}). These expectations were first explicitly
described in \cite{Okounkov2018}.

Firstly, 3d mirror symmetry swaps the equivariant and K\"ahler
variables of $X$ with those of $\check X$. More precisely, let $\sK
\coloneqq \Pic(X) \otimes_{\bZ} \bC^\times$, and let $\sT = \sA \times
\bC^\times_{\hbar}$ act on $X$, where $\sA$ is the sub-torus
preserving the symplectic form. Similarly define $\check\sK$,
$\check\sT$, and $\check\sA$ for $\check X$. Then the data of 3d
mirror symmetry should include
\begin{itemize}[leftmargin=1.5cm,rightmargin=1cm,align=left,labelwidth=1.4cm,labelsep=0pt]
\item[\namedlabel{item:3d-mirror-assumption-1}{($\star$)}] an
  isomorphism
  \[ \kappa\colon \check\sA \times \check\sK \times \bC^\times_{\check\hbar} \times \bC^\times_q \xrightarrow{\sim} \sA \times \sK \times \bC^\times_{\hbar} \times \bC^\times_q \]
  which identifies $\check\sA \cong \sK$, $\check\sK \cong \sA$ and
  $\check\hbar = q/\hbar$, and a bijection of fixed points
  \[ b\colon X^{\sT} \cong (\check X)^{\check \sT}. \]
\end{itemize}
In the case of $X = \cM_r(\cA_{m-1})$, our notation for the
equivariant and K\"ahler variables will be as follows.
\begin{itemize}
\item (Equivariant variables) Let $\sT_W \subset G_W \coloneqq \prod
  \GL(W_i)$ be the {\it framing} torus, whenever the framing of the
  Nakajima quiver variety has dimension $> 1$. Let $u_1, \ldots, u_r$
  denote its weights. Then $\sT \coloneqq \sT' \oplus \sT_W$ where
  $\sT'$ is the usual torus with weights $x, y$ acting on $\cA_{m-1}$
  (from section~\ref{sec:tori-definitions}). Recall that $\hbar^{-1} =
  xy$, and let $t \coloneqq x/y$. Then the equivariant variables are
  the coordinates on $\sA$:
  \[ t, \frac{u_1}{u_2}, \frac{u_2}{u_3}, \ldots, \frac{u_{r-1}}{u_r}. \]
\item (K\"ahler variables) Recall that $\Pic(X) = \bZ \cO_{\Hilb}(1)
  \oplus \Pic(\cA_{m-1})$. In this order, denote its generators by
  \[ \fz_\delta, \fz_1, \ldots, \fz_{m-1}. \]
  These are the K\"ahler variables. In the notation of
  section~\ref{sec:quasimap-vertex}, $\fz_\delta = \fz_0 \fz_1 \cdots
  \fz_{m-1}$.
\end{itemize}
If $w$ is a variable for $X$, let $\check w$ denote the analogous
variable for $\check X$. Then the isomorphism $\kappa$ is given by
\begin{equation} \label{eq:3d-mirror-variables}
  \check\hbar \leftrightarrow q/\hbar, \qquad
  \begin{aligned}
    \check t &\leftrightarrow \fz_\delta \\
    \check \fz_\delta &\leftrightarrow t
  \end{aligned}, \qquad
  \begin{aligned}
    \check u_i/\check u_{i+1} &\leftrightarrow \fz_i \\
    \check \fz_i &\leftrightarrow u_i/u_{i+1}
  \end{aligned}.
\end{equation}
A fixed point in $X$ is an $m$-tuple of $r$-tuples of partitions
$\vec\lambda = (\lambda_{a,b})_{\substack{0 \le a < m\\0 \le b < r}}$,
which we view as an $m \times r$ matrix. Then the bijection $b$ of
fixed points is the transpose map $\vec\lambda \mapsto \vec\lambda^t$.

\subsubsection{}

Let $\bar\kappa$ denote the isomorphism $\kappa$ restricted to
$\check\sA \times \check\sK$. Then $d\bar\kappa$ induces isomorphisms
\begin{align*}
  \check \Lie \sA &\cong \Pic(X) \otimes_{\bZ} \bC \\
  \check\Pic(X) \otimes_{\bZ} \bC &\cong \Lie \sA.
\end{align*}
For Nakajima quiver varieties, $\Pic(X) \otimes_{\bZ} \bC$ is exactly
the (complexified) space of GIT stability conditions. A choice of
stability chamber $C_{\eff} \subset \Pic(X) \otimes_{\bZ} \bC$
identifies $X$ among all its flops $\cM_{\vec\theta}$, and is
equivalent to the choice of effective cone. Let $\check C_{\eff}$ be
the effective cone of the mirror $\check X$. Given $X$ with a choice
of attracting chamber $\fC \subset \Lie \sA$, 3d mirror symmetry
should give
\begin{itemize}[leftmargin=1.5cm,rightmargin=1cm,align=left,labelwidth=1.4cm,labelsep=0pt]
\item[\namedlabel{item:3d-mirror-assumption-2}{($\star\star$)}]
  identifications $\check C_{\eff} \leftrightarrow \fC$ and
  $\check \fC \leftrightarrow C_{\eff}$ via $d\bar\kappa$.
\end{itemize}

\begin{example} \label{ex:mirror-attracting-chambers}
  For the stability condition $\vec\theta = (-m+1+\epsilon, 1, 1,
  \ldots, 1)$ defining $\Hilb(\cA_{m-1})$, the corresponding point in
  $\Lie \check\sA$ is given by $t = \sum \theta_i = \epsilon$ and $u_i
  - u_j = 1$. (Here we conflate the K-theoretic weights $t, u_i \in
  \check\sA$ with their logarithms in $\Lie \check\sA$.) The chamber
  containing this point is therefore
  \[ \check\fC_- \coloneqq \{u_1 \gg u_2 \gg \cdots \gg u_m > t > 0\}. \]
  Similarly, the mirror chamber for $\Hilb([\bC^2/\Gamma])$ is
  \[ \check\fC_+ \coloneqq \{t > u_1 \gg u_2 \gg \cdots \gg u_m > 0\}. \]
\end{example}

\subsubsection{}

In an enumerative context, 3d mirror symmetry can be studied using
quasimap theory following \cite{Aganagic2016} and \cite{Okounkov2018},
using a key ingredient $\Stab_{\fC}^\Ell$ called the {\it elliptic
  stable envelope}. Since the two geometries $X$ and $\check X$ are
mirror, one expects the quasimap vertices $\sV_{\QMaps}(X)$ and
$\sV_{\QMaps}(\check X)$ to be related. The relationship is best
investigated using the $q$-difference equations
\footnote{Here, the variable $q$ is identified with the weight of the
  torus acting on the quasimap domain $C$; we called this $z$
  previously, in the context of the threefold $Y = \cA_{m-1} \times
  C$.} that they satisfy in the equivariant and K\"ahler variables,
first derived in \cite[Section 8]{Okounkov2017}. Let
$\tilde\sV_{\QMaps}$ be the normalization of $\sV_{\QMaps}$ so that
these $q$-difference equations are {\it scalar}. Whatever the correct
mathematical definition is for $X$ and $\check X$ to be 3d mirror, we
assume it requires that
\begin{itemize}[leftmargin=1.5cm,rightmargin=1cm,align=left,labelwidth=1.4cm,labelsep=0pt]
  \item[\namedlabel{item:3d-mirror-assumption-3}{($\star\star\star$)}]
    $\tilde\sV_{\QMaps}(X)$ and $\tilde\sV_{\QMaps}(\check X)$ satisfy
    the {\it same} scalar $q$-difference equations,
\end{itemize}
up to the change of variables \eqref{eq:3d-mirror-variables} which
swaps the $q$-difference operators for equivariant and K\"ahler
variables. Then $\tilde\sV(X)$ and $\tilde\sV(\check X)$ are
distinguished among all solutions to the $q$-difference equations as
the ones holomorphic in $\fz$ and $\check\fz$ respectively, in a
suitable neighborhood of $0$ in $C_{\eff}$ and $\check C_{\eff}$. One
can ask for the change of basis matrix transforming the basis of
solutions holomorphic in $\fz$ to the basis holomorphic in
$\check\fz$. The main result of \cite{Aganagic2016} is a geometric
realization of this matrix as a certain normalization
$\tilde\Stab_{\fC}^\Ell$ of the elliptic stable envelope.

\begin{theorem}[{\cite[Theorem 5]{Aganagic2016}}]
  Suppose $X$ is 3d mirror to $(\check X, \check\fC)$ in the sense
  that it satisfies \ref{item:3d-mirror-assumption-1},
  \ref{item:3d-mirror-assumption-2},
  \ref{item:3d-mirror-assumption-3}. Then
  \begin{equation} \label{eq:mirror-quasimap-vertex}
    \kappa^*\tilde\sV_{\QMaps}(X) = \tilde\Stab_{\check\fC}^{\Ell}\tilde\sV_{\QMaps}(\check X)
  \end{equation}
  where $\kappa^*$ denotes the change of variables
  \eqref{eq:3d-mirror-variables}.
\end{theorem}

This theorem inspires the tentative mathematical definition of 3d
mirror symmetry given in \cite[Definition 1.1]{Rimanyi2019}, where it
is shown that the {\it non-affine} type A quiver is self-mirror. Their
definition should be equivalent to our three assumptions
\ref{item:3d-mirror-assumption-1}, \ref{item:3d-mirror-assumption-2}
and \ref{item:3d-mirror-assumption-3}.

\subsubsection{}

The precise definition of $\tilde\Stab_{\fC}^\Ell$ is unimportant for
us, because we will mostly use Theorem~\ref{thm:3d-mirror-symmetry} in
the case where all equivariant variables of $X$ vanish. Namely, if
$\vec\sigma\colon \bC^\times \to \sA$ is a generic cocharacter with
$d\vec\sigma \in \fC$, then we are interested in the limit
\[ \sV_{\QMaps}^p(\vec 0_{\fC}, \vec\fz) \coloneqq \lim_{w \to 0} \sV_{\QMaps}^p(\vec\sigma(w), \vec\fz) \in \bZ(q, \hbar)[[\vec\fz^\pm]]. \]
Equivalently, all K\"ahler variables of $\check X$ vanish, and
$\tilde\sV_{\QMaps}^p(\check X)$ is truncated to its constant term. The
elliptic stable envelope becomes a diagonal matrix with very explicit
entries. The overall effect is that \eqref{eq:mirror-quasimap-vertex}
becomes very simple.

\begin{corollary} \label{thm:3d-mirror-symmetry}
  For $p \in X^{\sT}$, let
  \[ T_pX = T^{> 0,\fC}_p + T^{< 0,\fC}_p \]
  be a decomposition into attracting and repelling directions with
  respect to $\fC$. Then
  \begin{equation} \label{eq:mirror-quasimap-vertex-non-equivariant}
    \kappa^*\sV_{\QMaps}^{p}(\vec 0_{\fC}, \vec\fz) = \prod_{w \in T^{<0,\check\fC}_{b(p)}\check X} \frac{(q\check \hbar^{-1} w^{-1}; q)_\infty}{(w^{-1}; q)_\infty}
  \end{equation}
  where $(w; q)_\infty \coloneqq \prod_{n \ge 0} (1 - q^n w)$ is the
  $q$-analogue of the Gamma function.
\end{corollary}

This relation between the quasimap vertex of $X$ and the
$\check\sT$-equivariant geometry of $\check X$ was stated and checked
in \cite{Dinkins2019a, Dinkins2019} for cotangent bundles to flag
varieties and $\Hilb([\bC^2/\Gamma])$.

\subsubsection{}

The remainder of this paper will explore the consequences of 3d mirror
symmetry in the form of Corollary~\ref{thm:3d-mirror-symmetry},
between $X = \Hilb(\cA_{m-1})$ (and its flop $\Hilb([\bC^2/\Gamma])$)
and its mirror $\check X = \cM_m(\bC^2)$. Specifically, via the
BS/quasimaps correspondence, the formula
\eqref{eq:mirror-quasimap-vertex-non-equivariant} yields formulas for
certain DT/PT/BS vertices of $Y = \cA_{m-1} \times C$ that can be
explicitly checked. Alternatively, the resulting statements for these
vertices can be viewed as non-trivial evidence for the {\it affine}
type A quiver being self-mirror.

\subsection{The Calabi--Yau limit}
\label{sec:3d-mirror-symmetry-CY}

\subsubsection{}

For equivariant K-theoretic objects on a threefold $Y$, setting $xyz =
1$ is known as the {\it Calabi--Yau (CY) limit}. The evident
self-duality of the tangent-obstruction theory \eqref{eq:bs-Tvir}
ensures that in this limit, the {\it equivariant} pushforward $\chi(M,
\hat\cO_M^{\vir})$ becomes the {\it topological} Euler characteristic
$\chi(M)$ up to a sign (see \cite[Section 3.1]{Nekrasov2016} for
details). When $Y$ is a crepant resolution, it is known that
\begin{equation} \label{eq:CY-pairs-vertex}
  \sV_{\CY}^{\vec\lambda}(Y) \coloneqq \sV^{\vec\lambda}_{\PiPairs}(Y)\bigg|_{xyz=1} = \frac{\sV^{\vec\lambda}_{\Pairs}(Y)}{\sV_{\Pairs_{\text{exc}}}(Y)}\bigg|_{xyz=1} = \frac{\sV^{\vec\lambda}_{\cat{DT}}(Y)}{\sV_{\cat{DT}_{\text{exc}}}(Y)}\bigg|_{xyz=1} \in \bZ((Q))[[\vec A]]
\end{equation}
by \cite{Bryan2016} and \cite{Toda2010} respectively, i.e. the
DT/PT/BS correspondences hold.

It is possible to compute an explicit formula for
$\sV_{\CY}^{\vec\lambda}(Y)$. Such an explicit formula is known in the
physics literature under the guise of {\it geometric engineering} of
(Nekrasov partition functions for) 4d $\cN=2$ gauge theories with
prescribed matter contents. Then we can verify
\eqref{eq:mirror-quasimap-vertex-non-equivariant} (in the CY limit)
manually via the BS/quasimaps correspondence. Thus the BS/quasimaps
correspondence allows us to view this geometric engineering, in
certain cases, as a consequence of 3d mirror symmetry.

\subsubsection{}

The CY limit $xyz=1$ becomes $q/\hbar = 1$ in quasimaps language. Let
$\sV_{\QMaps,\CY}$ denote the quasimap vertex in the CY limit. On the
mirror side, the CY limit becomes $\check\hbar = 1$. Then there is
massive cancellation on the rhs of
\eqref{eq:mirror-quasimap-vertex-non-equivariant}, which yields (cf.
\cite{Okounkov2018})
\begin{equation} \label{eq:quasimap-vertex-CY-limit}
  \sV_{\QMaps,\CY}^p(X) = S^\bullet\left(T_{b(p)}^{<0,\check\fC} \check X\right)^\vee.
\end{equation}
Here $S^\bullet$ can be viewed as the symmetric algebra functor on
(virtual) vector spaces, and is also known as the {\it plethystic
  exponential}. On $K_{\sT}(\pt)$, it can be defined as
\[ S^\bullet(w) \coloneqq \frac{1}{1 - w}, \quad S^\bullet(w + w') = S^\bullet(w) S^\bullet(w') \]
where $w, w' \in K_{\sT}(\pt)$ are monomials.

\begin{proposition} \label{prop:cy-vertex-formula}
  Let $\check\fC = \{u_1 \gg u_2 \gg \cdots \gg u_m \gg t > 0\}$. Then
  \[ \sV_{\CY}^{\vec\lambda}(Y) = S^\bullet \left(T^{< 0, \check\fC}_{\vec\lambda}\cM_m(\bC^2)\right)^\vee\bigg|_{(x,y)=(Q,Q^{-1}), \, u_a/u_b=A_{ab}}. \]
\end{proposition}

This is exactly 3d mirror symmetry in the form of
Corollary~\ref{thm:3d-mirror-symmetry} for $X = \Hilb(\cA_{m-1})$ in
the CY limit: Example~\ref{ex:mirror-attracting-chambers} shows that
$\check\fC$ is mirror to $C_{\eff}(X)$, and the change of variables is
exactly the composition of $\kappa$ with the BS/quasimaps change of
variables $Q \leftrightarrow \fz_\delta$ and $A_{ab} \leftrightarrow
\fz_{ab}$.

\subsubsection{}

The proof of Proposition~\ref{prop:cy-vertex-formula} is by direct
calculation. Whenever $Y$ is built from the local CY pieces $\tot(\cO
\oplus \cO(-2))$ and $\tot(\cO(-1) \oplus \cO(-1))$, a convenient
formalism for computing $\sV_{\CY}(Y)$ is given in \cite{Iqbal2006}.
This certainly holds for $Y = \cA_{m-1} \times \bC$. These local CY
pieces are special because it is possible to obtain a closed-form
formula for their vertices, using the formula \cite[Formula
  3.23]{Okounkov2006} for the topological vertex. Namely, let
\[ s_\lambda \coloneqq s_\lambda(Q^{-\rho}) = s_\lambda(Q^{1/2}, Q^{3/2}, Q^{5/2}, \ldots) \]
be a principal $Q$-specialization of the Schur function, and let
\[ [\lambda\mu]_A \coloneqq \prod_{i,j \ge 1} \frac{1}{1 - A Q^{-\rho_i - \lambda_i} Q^{-\rho_j - \mu_j}} \]
be the prefactor arising from applying (skew) Cauchy identities to
certain (skew) Schur functions in the topological vertex. Then, using
\cite{Iqbal2006} or otherwise,
\begin{equation} \label{eq:vertex-CY-limit}
  \sV_{\CY}^{\vec\lambda}(Y) = Q^{\cdots} \frac{\prod_a s_{\lambda_a} \prod_{a<b} [\lambda_a^t\lambda_b]_{A_{ab}}}{\prod_{a<b} [\emptyset\emptyset]_{A_{ab}}}
\end{equation}
where $A_{ab} \coloneqq A_a A_{a+1} \cdots A_{b-1}$ for $a < b$, and
$Q^{\cdots}$ denotes some monomial in $Q$ which is unimportant to us.

\subsubsection{}

To begin relating $\sV_{\CY}^{\vec\lambda}(Y)$ to
$T^{<0}_{\vec\lambda}\cM_m(\bC^2)$, we need to introduce some notation
for partitions. Recall that the {\it arm-length} and {\it leg-length}
of a square $\square \in \lambda$ are
\begin{align*}
  a_\lambda(\square) &\coloneqq \lambda_{i(\square)} - j(\square) \\
  \ell_\lambda(\square) &\coloneqq (\lambda^t)_{j(\square)} - i(\square).
\end{align*}
Let $n(\lambda) \coloneqq \sum_k (k - 1/2)\lambda_k$. Then a slight
modification of Stanley's hook length formula yields
\[ s_\lambda = \frac{Q^{n(\lambda)}}{\prod_{\square \in \lambda} (1 - Q^{a_\lambda(\square) + \ell_\lambda(\square) + 1})}. \]
This formula for $s_\lambda$ can be written in terms of the geometry
of $\Hilb(\bC^2)$ as follows. There is an explicit combinatorial
formula
\begin{equation} \label{eq:hilb-tangent-character}
  \begin{split}
    T_\lambda\Hilb(\bC^2)
    &= \chi(\cO_{\bC^2}) - \chi(I_\lambda, I_\lambda) \\
    &= \sum_{\square \in \lambda} \left(x^{-a_\lambda(\square)-1} y^{\ell_\lambda(\square)} + x^{a_\lambda(\square)} y^{-\ell_\lambda(\square)-1}\right)
  \end{split}
\end{equation}
for the character of the tangent space at a fixed point $\lambda$. If
$\check\fC = \{t > 0\}$ is the choice of attracting chamber, then
\begin{equation} \label{eq:schur-mirror-formula}
  Q^{-n(\lambda)} s_\lambda = S^\bullet \left(T_\lambda^{< 0, \check\fC}\Hilb(\bC^2)\right)^\vee\bigg|_{(x,y)=(Q,Q^{-1})}.
\end{equation}

\subsubsection{}

The generalization of \eqref{eq:hilb-tangent-character} to
$\cM_m(\bC^2)$ is
\[ T_{\vec\lambda}\cM_m(\bC^2) = \sum_{a,b} \frac{u_b}{u_a} T_{\lambda_a,\lambda_b}, \]
where $T_{\lambda,\mu} \coloneqq \chi(\cO_{\bC^2}) - \chi(I_\lambda,
I_\mu)$. It is shown in the proof of \cite[Lemma 6]{Carlsson2012} that
\begin{align*}
  T_{\lambda,\mu}
  &= \sum_{i,j \ge 1} x^{\mu_i-j} y^{-\lambda^t_j+i-1} - \sum_{i,j \ge 1} x^{-j} y^{i-1} \\
  &= \sum_{\square \in \lambda} x^{-a_\lambda(\square)-1} y^{\ell_\mu(\square)} + \sum_{\square \in \mu} x^{a_\mu(\square)} y^{-\ell_\lambda(\square)-1}.
\end{align*}
It follows almost immediately that the remaining terms in
\eqref{eq:vertex-CY-limit} are
\begin{equation} \label{eq:instanton-Tab-mirror-formula}
  \frac{[\lambda_a\lambda_b]_{A_{ab}}}{[\emptyset\emptyset]_{A_{ab}}} = S^\bullet \left(A_{ab} T_{\lambda_a,\lambda_b}^\vee\right)\bigg|_{(x,y)=(Q,Q^{-1})}.
\end{equation}

\begin{proof}[of Proposition~\ref{prop:cy-vertex-formula}]
  Note that $T_\lambda \Hilb(\bC^2) = T_{\lambda,\lambda}$ and
  \[ T^{< 0,\check\fC}\cM_m(\bC^2) = \sum_a T_{\lambda_a,\lambda_a}^{<0,\check\fC} + \sum_{a<b} \frac{u_b}{u_a} T_{\lambda_a,\lambda_b}. \]
  Hence plugging \eqref{eq:schur-mirror-formula} and
  \eqref{eq:instanton-Tab-mirror-formula} into
  \eqref{eq:vertex-CY-limit} yields the desired equality up to a
  monomial $Q^{\cdots}$. This monomial must be trivial, because both
  sides are $1 + O(Q, \vec A)$.
\end{proof}

\subsubsection{}

It is worth mentioning that even without using the BS/quasimaps
correspondence, 3d mirror symmetry in the CY limit for quasimap
vertices already has many interesting consequences. For example, the
degree configurations of $\sT$-fixed quasimaps to $\lambda \in
\Hilb([\bC^2/\Gamma])$ have been widely studied under the name of
(colored) {\it reverse plane partitions (RPPs)} of shape $\lambda$.
These are ways to label the squares in $\lambda$ with non-negative
degrees $d_\square$ such that they are non-decreasing along rows and
columns. If $\deg_k$ is the sum of $d_\square$ for all squares of
color $c(\square) = k$, then
\[ \sV_{\QMaps,\CY}^{\lambda}(\Hilb([\bC^2/\Gamma])) = \sum_{\pi \in \text{RPP}(\lambda)} \vec\fz^{\deg \pi} \]
is a generating series for colored RPPs, since all fixed loci are
isolated points. Then 3d mirror symmetry yields a closed-form formula
for this series which matches with the known result in \cite[Theorem
  5.1]{Gansner1981}. More generally, \cite{Dinkins2019} proves a
formula for $\sV_{\QMaps}^{\lambda}(\Hilb([\bC^2/\Gamma])\big|_{t=0}$
to verify Corollary~\ref{thm:3d-mirror-symmetry} in full generality,
which can be viewed as a $(q,\hbar)$-deformation of that of
\cite{Gansner1981}.

A striking feature of these and related formulas is that they are {\it
  plethystic}, in the sense that the desired generating series can be
written as $S^\bullet(\cdots)$. This is unsurprising via the CY limit
of 3d mirror symmetry, which even explicitly identifies the terms in
$\cdots$ using the equivariant geometry of the mirror. 

\subsubsection{}

For $X = \Hilb(\cA_{m-1})$, the description of $\sT$-fixed quasimaps
to $\vec\lambda \in X^{\sT}$ in section~\ref{sec:quasimaps-fixed-loci}
suggests to generalize RPPs as follows. For general $\vec\lambda$, the
components of $\QMaps^{\vec\lambda}(X)$ are not isolated points, and
it is not so clear what the enumerative significance of
$\sV_{\QMaps,\CY}^{\vec\lambda}(X)$ is. But when $\vec\lambda$ is
empty except at a single $\lambda_a \neq \emptyset$, all fixed loci
are isolated points. The resulting degree configuration is a {\it
  generalized RPP (GRPP)} in the sense that each square is still
labeled by a degree $d_{\square}$ and degrees are still non-decreasing
along rows and columns, but:
\begin{itemize}
\item the underlying shape $V(\vec\lambda)$ is a skew diagram with
  squares {\it outside} the positive quadrant in general;
\item the degrees labeling squares may be negative, with the condition
  that $d_{(0,0)} \ge 0$.
\end{itemize}
Then Proposition~\ref{prop:cy-vertex-formula} can be viewed as a
plethystic formula for the generating series of colored GRPPs of shape
$V(\vec\lambda)$, which also admits a $(q,\hbar)$-deformation.

\subsection{Crepant resolution conjecture}
\label{sec:DT-CRC}

\subsubsection{}

Let $Y \to \fY$ be a crepant resolution of a CY3 orbifold. One can
study the orbifold DT/PT theory for $\fY$ and obtain an {\it orbifold
  topological vertex} $\sV_{\CY}^{\orb}(\fY)$ \cite{Bryan2012}. When
$\fY$ satisfies the hard Lefschetz condition \cite[Section
  4]{Fernandez2006}, the {\it DT crepant resolution conjecture (CRC)}
essentially asserts that
\begin{equation} \label{eq:DT-CRC-vertices}
  \sV_{\CY}(Y) \equiv \sV_{\CY}^{\orb}(\fY)
\end{equation}
where we write $\equiv$ to emphasize that the equality is {\it not} as
series, but rather as rational functions. That both sides are indeed
rational functions is not immediate; a proof is given in
\cite{Beentjes2018}, which also proves the DT CRC, whenever $\fY$ has
projective coarse moduli space, for {\it partition functions}
$\sZ_{\CY}(Y)_{\beta} \equiv \sZ_{\CY}^{\orb}(\fY)_{\beta}$
enumerating curves of suitable classes $\beta$.

\subsubsection{}
\label{sec:DT-CRC-full}

For $Y = \cA_{m-1} \times \bC$ and $\fY = [\bC^2/\Gamma] \times \bC$,
the BS/quasimaps correspondence and 3d mirror symmetry yield a version
of the DT CRC for fully equivariant and K-theoretic $1$-leg vertices.
In what follows, $\tilde\sV_{\Pairs}(\fY)$ is exactly the (normalized)
PT $1$-leg vertex for $\bC^3$, except with variables $\vec Q \coloneqq
(Q_0, \ldots, Q_{m-1})$ recording the number of boxes of each of the
$m$ colors.

\begin{conjecture-theorem}[Equivariant K-theoretic DT CRC]
  Using the bijection of fixed points of
  Lemma~\ref{lem:orbifold-fixed-point-bijection}, identify
  $\cat{Coh}_{\sT}(Y)$ with $\cat{Coh}_{\sT}(\fY)$. There is an
  equality of rational functions
  \[ \tilde\sV_{\PiPairs}(Y) \equiv \tilde\sR_{\check\fC_- \leftarrow \check\fC_+} \tilde\sV_{\Pairs}(\fY) \]
  where $\check\fC_\pm$ are the attracting chambers of
  Example~\ref{ex:mirror-attracting-chambers}, and
  \[ \tilde\sR_{\fC' \leftarrow \fC} \coloneqq \tilde\Stab_{\fC'}^{\Ell} \circ \left(\tilde\Stab_{\fC}^{\Ell}\right)^{-1} \]
  are (a certain normalization of) elliptic R-matrices.
\end{conjecture-theorem}

The conjectural aspect arises from whether $\Hilb(\cA_{m-1})$ is
actually 3d mirror to $\cM_m(\bC^2)$ in the sense of
\ref{item:3d-mirror-assumption-3}, so that
Theorem~\ref{thm:3d-mirror-symmetry} is applicable. (Assumptions
\ref{item:3d-mirror-assumption-1} and
\ref{item:3d-mirror-assumption-2} are easily verified to hold.) In the
CY limit this is verified by Proposition~\ref{prop:cy-vertex-formula},
and therefore what follows will constitute a proof. An alternate proof
by direct computation was given in \cite{Ross2017}.

\subsubsection{}

We first explain how to reinterpret the DT CRC as a consequence of 3d
mirror symmetry. Let
\begin{align*}
  X_- &\coloneqq \Hilb(\cA_{m-1}) \\
  X_+ &\coloneqq \Hilb([\bC^2/\Gamma]),
\end{align*}
which are flops of each other (across multiple walls). The BS/quasimap
correspondence provides equalities
\begin{align*}
  \tilde\sV_{\QMaps}(X_-) &= \tilde\sV_{\PiPairs}(\cA_{m-1} \times \bC) \\
  \tilde\sV_{\QMaps}(X_+) &= \tilde\sV_{\Pairs}([\bC^2/\Gamma] \times \bC)
\end{align*}
after the appropriate changes of variables. Hence the DT CRC is a
comparison of $\sV_{\QMaps}(X_+)$ and $\sV_{\QMaps}(X_-)$. On the 3d
mirror side, Theorem~\ref{thm:3d-mirror-symmetry} then yields
\[ \left(\tilde\Stab_{\check\fC_+}^{\Ell}\right)^{-1} \tilde\sV_{\QMaps}(X_+) = \kappa^* \tilde\sV_{\QMaps}(\cM_m(\bC^2)) = \left(\tilde\Stab_{\check\fC_-}^{\Ell}\right)^{-1} \tilde\sV_{\QMaps}(X_-), \]
as desired. That all relevant $1$-leg vertices are rational functions
follows from more general expectations for fully-equivariant quasimap
vertices, see e.g. \cite{Smirnov2016}.

\subsubsection{}

In the CY limit, using \eqref{eq:quasimap-vertex-CY-limit}, the DT CRC
becomes a comparison of
\[ \sV_{\pm} \coloneqq S^\bullet \left(T^{< 0, \check\fC_\pm}_{\vec\lambda}\cM_m(\bC^2)\right)^\vee\bigg|_{(x,y)=(Q,Q^{-1})} \]
for the two different attracting chambers $\check\fC_\pm$ of
Example~\ref{ex:mirror-attracting-chambers}. Note that tangent spaces
of symplectic spaces in general can be written as
\[ T_{\vec\lambda} = T^{< 0,\fC} + \hbar (T^{> 0,\fC})^\vee. \]
Hence a change of attracting chamber $\fC \leadsto \fC'$ can only
change $T^{<0}$ as
\begin{equation} \label{eq:Tn-change-of-chamber}
  T^{< 0, \fC'} = T^{< 0, \fC} - G + \hbar G^\vee
\end{equation}
for some subspace $G \subset T^{< 0,\fC}$. The change of
variables $(x,y) = (Q,Q^{-1})$ implies $\hbar = 1$. Using that
\[ S^\bullet(G) = (-1)^{\dim G} \det(G) S^\bullet(G^\vee), \]
it follows immediately that
\[ \sV_+ = (\text{monomial}) \cdot \sV_-, \]
The discrepancy of an overall monomial factor in comparison to
\eqref{eq:DT-CRC-vertices} comes from
\[ \tilde\sV_{\Pairs,\CY}(\fY) = \vec Q^{\cdots} \sV_{\CY}^{\orb}(\fY) \]
for some monomial $\vec Q^{\cdots}$ (see e.g. \cite[Theorem
  3.1]{Ross2017}). Presumably this monomial is exactly the CY limit of
$\tilde\sR_{\check\fC_- \leftarrow \check\fC_+}$.

\phantomsection
\addcontentsline{toc}{section}{References}

\begin{small}
\bibliographystyle{alpha}
\bibliography{results}
\end{small}

\end{document}